\documentclass[10pt, oneside, a4paper]{article}
\usepackage{amsmath, amsthm, amssymb, amsfonts}
\usepackage{geometry}
\usepackage{graphicx}
\usepackage{natbib}
\usepackage{setspace}
\usepackage{algorithm}
\usepackage{algorithmicx}
\usepackage{algpseudocode}
\usepackage{longtable}
\usepackage{caption}
\usepackage{graphicx,subcaption}
\usepackage{float}
\usepackage{enumerate}
\usepackage{multirow}
\usepackage{cancel}
\usepackage[normalem]{ulem}

\newtheorem{theorem}{Theorem}[section]
\newtheorem{lemma}[theorem]{Lemma}

\theoremstyle{definition}

\newtheorem{assumption}{Assumption}
\newtheorem{remark}[theorem]{Remark}

\input{tcilatex}

\begin{document}

\title{How Reliable are Bootstrap-based Heteroskedasticity Robust Tests?%
\thanks{%
Financial support of the second author by the Program of Concerted Research
Actions (ARC) of the Universit\'{e} libre de Bruxelles is gratefully
acknowledged. We thank two referees and a co-editor for helpful comments.
Address correspondence to Benedikt P\"{o}tscher, Department of Statistics,
University of Vienna, A-1090 Oskar-Morgenstern Platz 1. E-Mail:
benedikt.poetscher@univie.ac.at.}}
\author{Benedikt M. P\"{o}tscher and David Preinerstorfer \\
%EndAName
Department of Statistics, University of Vienna\\
SEPS-SEW, University of St.~Gallen}
\date{First version: April 2020\\
This version: November 2021\\
}
\maketitle

\begin{abstract}
We develop theoretical finite-sample results concerning the size of wild
bootstrap-based heteroskedasticity robust tests in linear regression models.
In particular, these results provide an efficient diagnostic check, which
can be used to weed out tests that are unreliable for a given testing
problem in the sense that they overreject substantially. This allows us to
assess the reliability of a large variety of wild bootstrap-based tests in
an extensive numerical study.
\end{abstract}

\section{Introduction}

Testing hypotheses on the parameters in a regression model with potentially
heteroskedastic errors is a time-honored problem in econometrics and
statistics. As the classical $t$-statistic ($F$-statistic, respectively) is
not pivotal, or asymptotically pivotal, in such a case in general, even
under Gaussianity of the errors, so-called heteroskedasticity robust (aka
heteroskedasticity consistent) modifications of these test statistics have
been proposed. These statistics are asymptotically standard normally
(chi-square, respectively) distributed under the null. The first generation
of such procedures is rooted in the results of \cite{E63,E67}, see also \cite%
{H77}, and has been popularized in econometrics by \cite{W80}. It soon
transpired that tests obtained from these heteroskedasticity robust test
statistics by relying on critical values obtained from the respective
asymptotic distributions are prone to overrejecting the null hypothesis in
finite samples, especially so if the design matrix contains leverage points;
see, e.g., \cite{MacW85}, \cite{DavidsonMacKinnon1985}, and \cite%
{CheshJewitt1987}. One factor contributing to this tendency to overreject is
a downward bias present in the covariance matrix estimators used in these
test statistics, see \cite{CheshJewitt1987}. Attempts at remedying the
overrejection problem have led to the development of second generation
heteroskedasticity robust test statistics (often denoted by HC1 through HC4,
with HC0 denoting the first generation test statistic). These statistics use
various ways of rescaling the least-squares residuals before computing the
covariance matrix estimator employed in the construction of the test
statistic; see \cite{H77}, \cite{MacW85}, and \cite{Crib2004}. Simulation
studies reported in, e.g., \cite{DavidsonMacKinnon1985} and \cite{Crib2004}
show that these modifications, especially HC3 and HC4, ameliorate the
overrejection problem to some extent, but do not eliminate it. Further
numerical results are provided in \cite{CheshAust_1991}, see also \cite%
{Chesh_1989}. \cite{DavidsonMacKinnon1985} also consider variants of
HC0-HC3, denoted by HC0R-HC3R, obtained by using restricted instead of
unrestricted least-squares residuals in the computation of the covariance
matrix estimators employed by the various test statistics (the restriction
alluded to being the restriction defining the null hypothesis). In their
simulation experiments, this typically leads to tests that do not overreject
(but that may underreject); see also the simulation results in \cite%
{Godfrey2006}, who additionally also considers HC4R. Of course, these
simulation results do not rule out that the tests based on HC0R-HC4R
(relying on critical values suggested by asymptotic theory) may overreject
in some situations outside of the scope of the simulation studies; in fact, 
\cite{PP5HC} provide numerical proof that also these tests can suffer from
considerable overrejection. Note that under the typical assumptions used in
the literature all of the modifications of HC0 discussed so far have the
same asymptotic distribution under the null, and thus use the same critical
value.

An alternative approach is to use bootstrap methods to compute critical
values for the test statistics HC0-HC4 or HC0R-HC4R, with the intention to
improve upon the critical values derived from the asymptotic null
distributions.\footnote{%
Another possibility is to use Edgeworth expansions to find better critical
values, see \cite{Rothenberg1988} for the case of the HC0 test statistic and 
\cite{DavidsonMacKinnon1985} for the HC0R test statistic. Simulation results
in \cite{MacW85} and \cite{DavidsonMacKinnon1985} indicate that this does
not work too well in practice. Of course, such expansions could also be
worked out for the other versions of the test statistics mentioned, but this
does not seem to have been pursued in the literature. Other adjustments are
discussed in \cite{Imbkoles2016}; however, as shown in \cite{PP5HC}, these
adjustments do also not resolve the overrejection problem in general.}
Inspired by earlier work in the statistics literature (e.g., \cite{WuCFJ1986}
and the discussion in \cite{beran1986}, \cite{LiuR1988}, \cite{Mammen1993}), 
\cite{horowitz1997} used the wild bootstrap to obtain critical values for
HC0. This was followed up by \cite{Flachaire1999, Flachaire2005} and \cite%
{DavidsFlach2008}, who considered the test statistics HC0-HC3 as well as
HC0R-HC3R, and who stressed the version of the wild bootstrap that imposes
the null restriction on the bootstrap data generating process; see also \cite%
{GodfreyOrme2004} and the more recent survey \cite{Mackinnon2013}. Further
simulation studies, some of which also include HC4 and HC4R, can be found in 
\cite{Crib2004}, \cite{CribariLima2009}, \cite{Godfrey2006}, and \cite%
{PRich2017}. Once one has turned to bootstrap methods, one can also think of
reverting to the classical (i.e., uncorrected) $t$-statistic ($F$-statistic,
respectively) and to apply the bootstrap methods to determine appropriate
critical values. This has been considered in \cite{Mammen1993}; see also 
\cite{Godfrey2006} for some Monte Carlo results. Since the majority of the
literature on bootstrap-based heteroskedasticity robust tests favors the
wild bootstrap over other bootstrap methods, we shall concentrate on the
wild bootstrap in the sequel.

While the before mentioned bootstrap procedures have their merits and
overrejection is ameliorated in many of the cases considered in the
simulation studies cited, it is unclear whether these observations
generalize beyond the situations studied in these simulation experiments. In
particular, it is unclear if -- and under which conditions -- these
bootstrap procedures (and which of the many variants thereof) are immune to
overrejection in finite samples.\footnote{%
We are not interested in asymptotic justifications here.}$^{\text{,}}$%
\footnote{%
In the quite special case where the number of restrictions tested equals the
number of regression parameters,\ \cite{DavidsFlach2008} have a result which
implies that certain wild bootstrap-based tests have size equal to the
nominal significance level (and hence do not overreject) in finite samples.
We note that this result in \cite{DavidsFlach2008} is not entirely correct
as stated, but needs some amendments and corrections; see also Footnote \ref%
{FN_counterex}.} In the present paper we set out to study this question
theoretically and numerically. On the theoretical side we show the following
finite-sample result: For any test statistic $T$ from a large class of test
statistics (including HC0-HC4, HC0R-HC4R, the classical $F$-statistic and a
variant thereof that uses restricted residuals) and for any bootstrap method
from a large class of wild bootstrap methods (including virtually all wild
bootstrap methods considered in the literature) there is a computable number 
$\vartheta $ (depending only on observables like the design matrix, the
restriction to be tested, etc),\ such that the size of the corresponding
bootstrap-based test is $1$ for nominal significance levels $\alpha $
satisfying $\alpha >\vartheta $.\footnote{%
The size is the maximal (i.e., worst-case) null rejection probability, where
one maximizes over all possible forms of heteroskedasticity, reflecting
agnosticism about the form of the heteroskedasticity; see (\ref{size}) for a
formal definition. It is also assumed that the normalized regression error
vector $(\mathbf{u}_{1}/\limfunc{Var}^{1/2}(\mathbf{u}_{1}),\ldots ,\mathbf{u%
}_{n}/\limfunc{Var}^{1/2}(\mathbf{u}_{n}))^{\prime }$ follows a given
(fixed) distribution (e.g., a normal distribution). Note that the
theoretical result mentioned in the main text does \textbf{not} depend on
the choice of this distribution (as long as it is absolutley continuous),
see Section \ref{gen}.} That is, for $\alpha >\vartheta $ the
bootstrap-based test fails miserably in that it has null rejection
probabilities arbitrarily close to $1$ for some forms of heteroskedasticity.%
\footnote{%
By construction $\vartheta \leq 1$ always holds. If $\vartheta <1$ (which
will often be the case) we can then conclude that the bootstrap-based test
has size $1$ at least for some values of $\alpha $.} We note that our
results also provide information concerning the infimal coverage
probabilities of confidence sets obtained by \textquotedblleft
inverting\textquotedblright\ the bootstrap-based tests under consideration.
We discuss this in more detail in Remark \ref{conf_set}.

In practice our theoretical finite-sample result can be used as a diagnostic
tool to weed out procedures in the following sense: as mentioned before,
there is a large menu of heteroskedasticity robust test statistics and wild
bootstrap methods available in the literature from which the applied
researcher has to choose. As it is unlikely that simulation results in the
literature are available that precisely fit the problem the researcher is
interested in (i.e., use the same design matrix and the same restriction to
be tested), the researcher is typically left with little guidance on which
of the many bootstrap-based test procedures to choose for the problem at
hand. Based on our theoretical results, the applied researcher can now
eliminate procedures that break down in the researchers problem, by
computing -- for any initially selected procedure -- the corresponding $%
\vartheta $ for the given design matrix and restriction to be tested. If it
turns out that $\vartheta <\alpha $ holds, this procedure should not be
used, because these tests have size equal to one according to our
theoretical results. Numerical routines for computing $\vartheta $ are
provided in the associated R-package \textbf{wbsd} by \cite{wbsd}.\footnote{%
As discussed in Section \ref{sec:comp} and Sections \ref{sec:stp2}, \ref%
{sec:numch} of Appendix \ref{App compu}, computing $\vartheta $ is a
nontrivial numerical problem. Supplementing the calculation of $\vartheta $
by numerically evaluating null rejection probabilities for strategically
chosen heteroskedasticity structures as discussed in Section \ref{sec:stp2}
of Appendix \ref{App compu} may be advisable.}

The before mentioned theoretical result will typically have practical
consequences only in testing problems for which $\vartheta $ is sufficiently
small so that standard choices of $\alpha $ like $0.05$ or $0.1$ satisfy the
condition $\alpha >\vartheta $. We hence investigate this numerically for
the test statistics HC0-HC4, HC0R-HC4R, for the classical $F$-statistic, and
for a variant thereof that uses restricted residuals, each combined with a
large variety of wild bootstrap methods.\footnote{%
In the wild bootstrap methods we vary the following elements: (i) centering
the bootstrap sample at the unrestricted versus at the restriced least
squares estimator, (ii) bootstrapping from unrestricted versus restricted
residuals, (iii)\ the distribution of the bootstrap noise, and (iv) various
bootstrap multiplicator weights. See Section \ref{sec_Num} for more details.}
We now summarize the results of our numerical experiments for $n=10$ (the
results for $n=20,30$ being similar): For each combination of test statistic
and wild bootstrap method ($960$ combinations) we compute $\vartheta $ for a
range of design matrices and null hypotheses (i.e., restrictions to be
tested) and report $\vartheta _{\min }$, the smallest of these values of $%
\vartheta $.\footnote{%
For reasons of numerical stability we actually compute an upper bound for $%
\vartheta $, see Section \ref{sec:numch} in Appendix \ref{App compu} for
more information.}$^{\text{,}}$\footnote{%
Some of the $960$ combinations actually give rise to one and the same
bootstrap-based test. The reasons for nevertheless considering all $960$
combinations are discussed in Sections \ref{sec:descr} and \ref{sec:results}.%
} We find that for $826$ of the $960$ combinations $\vartheta _{\min }$ is
less than $0.05$, and for $936$ combinations $\vartheta _{\min }$ is less
than $0.1$. As a consequence, for the bootstrap-based tests corresponding to
these $826$ ($936$, respectively) combinations our theoretical results imply
that size is equal to $1$ for some design matrices and null hypotheses, if a
nominal significance level of $0.05$ ($0.1$, respectively) is being used.
Thus these bootstrap-based tests are found not to be reliable in general, in
that they suffer from severe overrejection for some design matrices and null
hypotheses. Furthermore, for each combination of test statistic/wild
bootstrap method we also compute a lower bound for the size of the
bootstrap-based test conducted at nominal significance level $\alpha =0.05$
(as well as at $\alpha =0.1$).\footnote{%
For the size computations we assume the errors to be normally distributed.}
We find that for $95$ out of the remaining $134$ combinations ($11$ out of
the remaining $24$ combinations, respectively) the (lower bound for the)
size exceeds $3\alpha $ for some of the design matrices and null hypotheses,
sometimes by a considerable margin. Thus also these combinations do not lead
to reliable bootstrap-based tests. This leaves us with $39$ ($13$,
respectively) combinations. Exploiting that some of these combinations left
are in fact equivalent to some of the above mentioned \emph{unreliable}
procedures (see Sections \ref{sec:descr} and \ref{sec:results} for an
explanation), allows us to even conclude that in the end only $16$ ($4$,
respectively) bootstrap-based heteroskedasticity robust test procedures do
not exhibit severe overrejection within the range of our numerical study
when $n=10$.

Combining the just-described results with similar findings for the other
sample sizes $n=20,30$ leads to the sobering conclusion that \emph{none} of
the bootstrap-based tests considered is reliable for all sample sizes and
for $\alpha =0.05$ as well as $\alpha =0.1$. That is, for every combination
of test statistic and bootstrap method considered, there is a sample size $%
n\in \{10,20,30\}$, a significance level $\alpha \in \{0.05,0.1\}$, a
testing problem and a design matrix, such that the size of the corresponding
bootstrap-based test equals $1$ by our theoretical results or is numerically
found to exceed $3\alpha $. We must hence conclude that \emph{none }of the
bootstrap-based tests considered is guaranteed to be immune to
overrejection, and thus such tests are no reliable panacea for
heteroskedasticity robust testing.

If one considers a fixed $\alpha $, the situation is somewhat more
encouraging. While there is no bootstrap-based test that is reliable for all
sample sizes for the significance level $\alpha =0.1$, for $\alpha =0.05$
there are two bootstrap-based tests that are found not to break down in the
above sense for any of the sample sizes considered in the numerical study.
Both of these tests use a heteroskedasticity robust test statistic based on
a HC3R covariance estimator, a wild bootstrap method based on the Mammen
distribution, and impose the null restriction on the bootstrap data
generating process. For more details see Section \ref{sec_Num}. It is
interesting to note that these findings call into question the
recommendation in \cite{DavidsFlach2008} to base the wild bootstrap on the
Rademacher distribution.

Of course, the above are worst-case results in spirit and do not preclude a
given bootstrap-based test to be reasonably sized for certain instances of
design matrix and null hypothesis. Therefore, in a given application, one
could in principle imagine the following strategy: Numerically evaluate the
size of the given bootstrap-based test (this will require to commit to a
distributional assumption on the errors) and use the test only if the
so-evaluated size does not exceed the nominal level $\alpha $ (by much).%
\footnote{%
Of course, this could also be pursued with non-bootstrap-based tests.}
Otherwise, switch to another one of the many other bootstrap-based tests,
repeat, and stop upon finding an acceptable test. As mentioned before, a
partial shortcut for this strategy could be to compute $\vartheta $ first
and to check if $\alpha >\vartheta $, as we then know from our theoretical
results that size must be equal to $1$. Of course, such a strategy would be
computationally expensive and moreover would only be a stab into the dark,
as there is no guarantee that one would end up with a bootstrap-based test
that performs well in the sense of delivering size less than or equal to $%
\alpha $. It seems that a better and more direct strategy is to forgo the
bootstrap idea and rather to construct size-controlling critical values for
the original test statistics, e.g., for HC0-HC4 or HC0R-HC4R. This is
pursued in the companion paper \cite{PP5HC}. Certainly, this also leads to a
computationally intensive method, but one that comes with guaranteed size
control.

All test statistics mentioned so far are based on the ordinary least squares
estimator. An alternative is to start from a feasible generalized least
squares estimator, computed from a (potentially misspecified) model for
heteroskedasticity. Again heteroskedasticity robust test procedures can then
be developed in a similar manner, see, e.g., \cite{Cragg_1983, Cragg_1992}, 
\cite{Flachaire_2005b}, \cite{RomanoWolf2017}, \cite{Lin_Chou_2018}, \cite%
{DiCiccio_Romao_Wolf_2019}. While results similar to the ones given in the
present paper can probably also be developed for this alternative class of
heteroskedasticity robust test procedures, we do not pursue this avenue here.

\section{Framework\label{frame}}

Consider the linear regression model 
\begin{equation}
\mathbf{Y}=X\beta +\mathbf{U},  \label{lm}
\end{equation}%
where $X$ is a (real) nonstochastic regressor (design) matrix of dimension $%
n\times k$ and where $\beta \in \mathbb{R}^{k}$ denotes the unknown
regression parameter vector. We always assume $\limfunc{rank}(X)=k$ and $%
1\leq k<n$. We furthermore assume that the $n\times 1$ disturbance vector $%
\mathbf{U}=(\mathbf{u}_{1},\ldots ,\mathbf{u}_{n})^{\prime }$ has mean zero
and unknown covariance matrix $\sigma ^{2}\Sigma $, where $\Sigma $ varies
in a user-specified (nonempty) set $\mathfrak{C}$ describing the allowed
forms of heteroskedasticity, with $\mathfrak{C}$ satisfying $\mathfrak{C}%
\subseteq \mathfrak{C}_{Het}$, and where $0<\sigma ^{2}<\infty $ holds ($%
\sigma $ always denoting the positive square root).\footnote{%
Since we are concerned with finite-sample results only, the elements of $%
\mathbf{Y}$, $X$, and $\mathbf{U}$ (and even the probability space
supporting $\mathbf{Y}$ and $\mathbf{U}$) may depend on sample size $n$, but
this will not be expressed in the notation. Furthermore, the obvious
dependence of$\ \mathfrak{C}$ on $n$ will also not be shown in the notation.}
The set $\mathfrak{C}$ will be referred to as the \textquotedblleft
heteroskedasticity model\textquotedblright . Here%
\begin{equation*}
\mathfrak{C}_{Het}=\left\{ \limfunc{diag}(\tau _{1}^{2},\ldots ,\tau
_{n}^{2}):\tau _{i}^{2}>0\text{ for all }i\text{, }\sum_{i=1}^{n}\tau
_{i}^{2}=1\right\} ,
\end{equation*}%
where $\limfunc{diag}(\tau _{1}^{2},\ldots ,\tau _{n}^{2})$ denotes the $%
n\times n$ matrix with diagonal elements given by $\tau _{i}^{2}$. That is,
the errors in the regression model are uncorrelated but can be
heteroskedastic. In particular, if $\mathfrak{C}$ is chosen to be $\mathfrak{%
C}_{Het}$, one allows for heteroskedasticity of completely unknown form. The
normalization condition $\sum_{i=1}^{n}\tau _{i}^{2}=1$ is included here
only in order to guarantee identifiability of $\sigma ^{2}$ and $\Sigma $,
and could be replaced by any other normalization condition such as, e.g., $%
\max \tau _{i}^{2}=1$, or $\tau _{1}^{2}=1$, without affecting the final
results (because any of these normalizations leads to the same overall set
of covariance matrices $\sigma ^{2}\Sigma $ when $\sigma ^{2}$ varies
through the positive real line).

\emph{Although of no real significance for the results of this paper as
explained in Section \ref{gen}, we shall, for ease of exposition, maintain
in the sequel that the disturbance vector }$\mathbf{U}$\emph{\ is normally
distributed.} The linear model described in (\ref{lm}), together with the
just made Gaussianity assumption on $\mathbf{U}$ and with the given
heteroskedasticity model $\mathfrak{C}$, then induces a collection of
distributions on the Borel-sets of $\mathbb{R}^{n}$, the sample space of $%
\mathbf{Y}$. Denoting a Gaussian probability measure with mean $\mu \in 
\mathbb{R}^{n}$ and (possibly singular) covariance matrix $A$ by $P_{\mu ,A}$%
, the induced collection of distributions is then given by 
\begin{equation}
\left\{ P_{\mu ,\sigma ^{2}\Sigma }:\mu \in \mathrm{\limfunc{span}}%
(X),0<\sigma ^{2}<\infty ,\Sigma \in \mathfrak{C}\right\} .  \label{lm2}
\end{equation}%
Since every $\Sigma \in \mathfrak{C}$ is positive definite by assumption,
each element of the set in the previous display is absolutely continuous
with respect to (w.r.t.) Lebesgue measure on $\mathbb{R}^{n}$.

We shall consider the problem of testing a linear (better: affine)
hypothesis on the parameter vector $\beta \in \mathbb{R}^{k}$, i.e., the
problem of testing the null $R\beta =r$ against the alternative $R\beta \neq
r$, where $R$ is a $q\times k$ matrix always of rank $q\geq 1$ and $r\in 
\mathbb{R}^{q}$. Set $\mathfrak{M}=\limfunc{span}(X)$. Define the affine
space 
\begin{equation*}
\mathfrak{M}_{0}=\left\{ \mu \in \mathfrak{M}:\mu =X\beta \text{ and }R\beta
=r\right\}
\end{equation*}%
and let 
\begin{equation*}
\mathfrak{M}_{1}=\left\{ \mu \in \mathfrak{M}:\mu =X\beta \text{ and }R\beta
\neq r\right\} .
\end{equation*}%
Adopting these definitions, the above testing problem can then be written
more precisely as 
\begin{equation}
H_{0}:\mu \in \mathfrak{M}_{0},\ 0<\sigma ^{2}<\infty ,\ \Sigma \in 
\mathfrak{C}\quad \text{ vs. }\quad H_{1}:\mu \in \mathfrak{M}_{1},\
0<\sigma ^{2}<\infty ,\ \Sigma \in \mathfrak{C}.  \label{testing problem}
\end{equation}%
With $\mathfrak{M}_{0}^{lin}$ we shall denote the linear space parallel to $%
\mathfrak{M}_{0}$, i.e., $\mathfrak{M}_{0}^{lin}=\mathfrak{M}_{0}-\mu
_{0}=\left\{ X\beta :R\beta =0\right\} $ where $\mu _{0}\in \mathfrak{M}_{0}$%
. Of course, $\mathfrak{M}_{0}^{lin}$ does not depend on the choice of $\mu
_{0}\in \mathfrak{M}_{0}$.

As already mentioned, the assumption of Gaussianity is made for the sake of
exposition only and does not really restrict the scope of the results in the
paper as is discussed in Section \ref{gen}. The assumption of nonstochastic
regressors entails little loss of generality either: For example, if $X$ is
random and $\mathbf{U}$ is conditionally on $X$ distributed as $N(0,\sigma
^{2}\Sigma )$, with $\sigma ^{2}=\sigma ^{2}(X)$ and $\Sigma =\Sigma (X)\in 
\mathfrak{C}_{Het}$, the results of the paper can be applied after one
conditions on $X$ (and a similar statement applies to the generalizations to
non-Gaussianity discussed in Section \ref{gen}). See Section \ref{gen} for
more discussion. For arguments supporting conditional inference see, e.g., 
\cite{RO1979}. Note that such a "strict exogeneity" assumption is quite
natural in the situation considered here.

We next collect some further terminology and notation used throughout the
paper. A (nonrandomized) \textit{test} is the indicator function of a
Borel-set $W$ in $\mathbb{R}^{n}$, with $W$ called the corresponding \textit{%
rejection region}. The \textit{size} of such a test (rejection region) is --
as usual -- defined as the supremum over all rejection probabilities under
the null hypothesis $H_{0}$ given in (\ref{testing problem}), i.e., 
\begin{equation}
\sup_{\mu \in \mathfrak{M}_{0}}\sup_{0<\sigma ^{2}<\infty }\sup_{\Sigma \in 
\mathfrak{C}}P_{\mu ,\sigma ^{2}\Sigma }(W).  \label{size}
\end{equation}%
In slight abuse of terminology, we shall sometimes refer to this quantity as
`the size of $W$ over $\mathfrak{C}$' when we want to emphasize the r\^{o}le
of $\mathfrak{C}$. Throughout the paper we let $\hat{\beta}(y)=\left(
X^{\prime }X\right) ^{-1}X^{\prime }y$, where $X$ is the design matrix
appearing in (\ref{lm}) and $y\in \mathbb{R}^{n}$. The corresponding
ordinary least squares (OLS) residual vector is denoted by $\hat{u}(y)=y-X%
\hat{\beta}(y)$ and its elements are denoted by $\hat{u}_{t}(y)$. The
elements of $X$ are denoted by $x_{ti}$, while $x_{t\cdot }$ and $x_{\cdot
i} $ denote the $t$-th row and $i$-th column of $X$, respectively. For $%
\mathcal{A}$ an affine subspace of $\mathbb{R}^{n}$ satisfying $\mathcal{A}%
\subseteq \limfunc{span}(X)$ let $\tilde{\beta}_{\mathcal{A}}(y)$ denote the
restricted least-squares estimator, i.e., $X\tilde{\beta}_{\mathcal{A}}(y)$
solves 
\begin{equation*}
\min_{z\in \mathcal{A}}(y-z)^{\prime }(y-z).
\end{equation*}%
Lebesgue measure on the Borel-sets of $\mathbb{R}^{n}$ will be denoted by $%
\lambda _{\mathbb{R}^{n}}$. The set of real matrices of dimension $l\times m$
is denoted by $\mathbb{R}^{l\times m}$ (all matrices in the paper will be
real matrices). The Euclidean norm is denoted by $\left\Vert \cdot
\right\Vert $. Let $B^{\prime }$ denote the transpose of a matrix $B\in 
\mathbb{R}^{l\times m}$ and let $\mathrm{\limfunc{span}}(B)$ denote the
subspace in $\mathbb{R}^{l}$ spanned by its columns. For a symmetric and
nonnegative definite matrix $B$ we denote the unique symmetric and
nonnegative definite square root by $B^{1/2}$. For a linear subspace $%
\mathcal{L}$ of $\mathbb{R}^{n}$ we let $\mathcal{L}^{\bot }$ denote its
orthogonal complement and we let $\Pi _{\mathcal{L}}$ denote the orthogonal
projection onto $\mathcal{L}$. The $j$-th standard basis vector in $\mathbb{R%
}^{n}$ is written as $e_{j}(n)$. Furthermore, we let $\mathbb{N}$ denote the
set of all positive integers. A sum (product, respectively) over an empty
index set is to be interpreted as $0$ ($1$, respectively). For a subset $A$
of a topological space we denote by $\limfunc{int}(A)$ the interior of $A$
(w.r.t. the ambient space). Finally, for $\mathcal{A}$ an affine subspace of 
$\mathbb{R}^{n}$, let $G(\mathcal{A})$ denote the group of all affine
transformations $y\mapsto \delta (y-a)+a^{\ast }$ where $\delta \in \mathbb{R%
}$, $\delta \neq 0$, and $a$ as well as $a^{\ast }$ are elements of $%
\mathcal{A}$; for more information see Section 5.1 of \cite{PP2016}.

\section{Heteroskedasticity robust test statistics using unrestricted
residuals \label{test_statistic}}

We next introduce two test statistics that will feature prominently.
Variants of these statistics using restricted residuals are discussed in
Section \ref{restrict_res}. For a result pertaining to a more general class
of test statistics see Theorem \ref{thm:main} in Appendix \ref{App A}. The
test statistic we shall consider first is a standard heteroskedasticity
robust test statistic frequently considered in the literature and is given by%
\begin{equation}
T_{Het}\left( y\right) =\left\{ 
\begin{array}{cc}
(R\hat{\beta}\left( y\right) -r)^{\prime }\hat{\Omega}_{Het}^{-1}\left(
y\right) (R\hat{\beta}\left( y\right) -r) & \text{if }\det \hat{\Omega}%
_{Het}\left( y\right) \neq 0, \\ 
0 & \text{if }\det \hat{\Omega}_{Het}\left( y\right) =0,%
\end{array}%
\right.  \label{T_het}
\end{equation}%
where $\hat{\Omega}_{Het}=R\hat{\Psi}_{Het}R^{\prime }$ and where $\hat{\Psi}%
_{Het}$ is a heteroskedasticity robust estimator as considered in \cite%
{E63,E67}, which later on has found its way into the econometrics literature
(e.g., \cite{W80}). It is of the form%
\begin{equation*}
\hat{\Psi}_{Het}\left( y\right) =(X^{\prime }X)^{-1}X^{\prime }\limfunc{diag}%
\left( d_{1}\hat{u}_{1}^{2}\left( y\right) ,\ldots ,d_{n}\hat{u}%
_{n}^{2}\left( y\right) \right) X(X^{\prime }X)^{-1},
\end{equation*}%
where the constants $d_{i}>0$ sometimes depend on the design matrix. Typical
choices for $d_{i}$ suggested in the literature are $d_{i}=1$, $%
d_{i}=n/(n-k) $, $d_{i}=\left( 1-h_{ii}\right) ^{-1}$, or $d_{i}=\left(
1-h_{ii}\right) ^{-2}$, where $h_{ii}$ denotes the $i$-th diagonal element
of the projection matrix $X(X^{\prime }X)^{-1}X^{\prime }$, see \cite{LE2000}
for an overview. Another suggestion is $d_{i}=\left( 1-h_{ii}\right)
^{-\delta _{i}}$ for $\delta _{i}=\min (nh_{ii}/k,4)$, see \cite{Crib2004}.
For the last three choices of $d_{i}$ just given, we use the convention that
we set $d_{i}=1$ in case $h_{ii}=1$. Note that $h_{ii}=1$ implies $\hat{u}%
_{i}\left( y\right) =0$ for every $y$, and hence it is irrelevant which real
value is assigned to $d_{i}$ in case $h_{ii}=1$.\footnote{%
In fact, $h_{ii}=1$ is equivalent to $\hat{u}_{i}\left( y\right) =0$ for
every $y$, each of which in turn is equivalent to $e_{i}(n)\in $ $\limfunc{%
span}(X)$.} The five examples for the weights $d_{i}$ just given correspond
to what is often called HC0-HC4 weights in the literature.

In conjunction with the test statistic $T_{Het}$, we shall consider the
following mild assumption, which is Assumption 3 in \cite{PP2016}. As
discussed further below, this assumption is in a certain sense unavoidable
when using $T_{Het}$. It furthermore also entails that our choice of
assigning $T_{Het}\left( y\right) $ the value zero in case $\hat{\Omega}%
_{Het}\left( y\right) $ is singular has no import on the rejection
probabilities of the (non-bootstrap-based) tests obtained from $T_{Het}$
(because of Lemma \ref{lem_B}(c) below and absolute continuity of the
measures $P_{\mu ,\sigma ^{2}\Sigma }$). As will be seen later, our results
for the corresponding bootstrap-based tests do also not depend on this
choice.

\begin{assumption}
\label{R_and_X}Let $1\leq i_{1}<\ldots <i_{s}\leq n$ denote all the indices
for which $e_{i_{j}}(n)\in \limfunc{span}(X)$ holds where $e_{j}(n)$ denotes
the $j$-th standard basis vector in $\mathbb{R}^{n}$. If no such index
exists, set $s=0$. Let $X^{\prime }\left( \lnot (i_{1},\ldots i_{s})\right) $
denote the matrix which is obtained from $X^{\prime }$ by deleting all
columns with indices $i_{j}$, $1\leq i_{1}<\ldots <i_{s}\leq n$ (if $s=0$ no
column is deleted). Then $\limfunc{rank}\left( R(X^{\prime }X)^{-1}X^{\prime
}\left( \lnot (i_{1},\ldots i_{s})\right) \right) =q$ holds.
\end{assumption}

Observe that this assumption only depends on $X$ and $R$ and hence can be
checked. Obviously, a simple sufficient condition for Assumption \ref%
{R_and_X} to hold is that $s=0$ (i.e., that $e_{j}(n)\notin \limfunc{span}%
(X) $ for all $j$), a generically satisfied condition. Furthermore, we
introduce the matrix%
\begin{eqnarray}
B(y) &=&R(X^{\prime }X)^{-1}X^{\prime }\limfunc{diag}\left( \hat{u}%
_{1}(y),\ldots ,\hat{u}_{n}(y)\right)  \notag \\
&=&R(X^{\prime }X)^{-1}X^{\prime }\limfunc{diag}\left( e_{1}^{\prime }(n)\Pi
_{\limfunc{span}(X)^{\bot }}y,\ldots ,e_{n}^{\prime }(n)\Pi _{\limfunc{span}%
(X)^{\bot }}y\right) .  \label{B_matrix}
\end{eqnarray}%
The facts collected in the subsequent lemma will be used in the sequel.
Parts (a)-(c) have been shown in Lemma 4.1 in \cite{PP2016}, while Part (d)
is taken from Lemma 5.18 of \cite{PP3}. Part (e) is obvious (observe that $%
B(y)$ depends only on $\hat{u}(y)$ and that $\hat{u}(\gamma (y-\mu )+\mu
^{\bullet })=\gamma \hat{u}(y)$ for every $\gamma \in \mathbb{R}$, every $%
\mu \in \limfunc{span}(X)$, and every $\mu ^{\bullet }\in \limfunc{span}(X)$%
).

\begin{lemma}
\label{lem_B}(a) $\hat{\Omega}_{Het}\left( y\right) $ is nonnegative
definite for every $y\in \mathbb{R}^{n}$.

(b) $\hat{\Omega}_{Het}\left( y\right) $ is singular (zero, respectively) if
and only if $\limfunc{rank}\left( B(y)\right) <q$ ($B(y)=0$, respectively).

(c) The set $\mathsf{B}$ given by $\left\{ y\in \mathbb{R}^{n}:\limfunc{rank}%
\left( B(y)\right) <q\right\} $ (or in view of (b) equivalently given by $%
\{y\in \mathbb{R}^{n}:\det (\hat{\Omega}_{Het}\left( y\right) )=0\}$) is
either a $\lambda _{\mathbb{R}^{n}}$-null set or the entire sample space $%
\mathbb{R}^{n}$. The latter occurs if and only if Assumption \ref{R_and_X}
is violated (in which case the test based on $T_{Het}$ becomes trivial, as
then $T_{Het}$ is identically zero).

(d) Under Assumption \ref{R_and_X}, the set $\mathsf{B}$ is a finite union
of proper linear subspaces of $\mathbb{R}^{n}$; in case $q=1$, $\mathsf{B}$
is even a proper linear subspace itself.\footnote{%
If Assumption \ref{R_and_X} is violated, $\mathsf{B}$ equals $\mathbb{R}^{n}$
by Part (c).}

(e) $\mathsf{B}$ is a closed set and contains $\limfunc{span}(X)$\textsf{. }%
Furthermore, $\mathsf{B}$ is $G(\mathfrak{M})$-invariant and, in particular, 
$\mathsf{B}+\limfunc{span}(X)=\mathsf{B}$ holds.
\end{lemma}

In light of Part (c) of the lemma, we see that Assumption \ref{R_and_X} is a
natural and unavoidable condition if one wants to obtain a sensible test
from $T_{Het}$.\footnote{%
If this assumption is violated then $T_{Het}$ is identically zero, an
uninteresting trivial case.} Furthermore, note that, if $\mathsf{B}=\limfunc{%
span}(X)$ is true, then Assumption \ref{R_and_X} must be satisfied (since $%
\limfunc{span}(X)$ is a $\lambda _{\mathbb{R}^{n}}$-null set due to the
maintained assumption $k<n$). As shown in Lemma A.3 in \cite{PP3}, for any
given restriction matrix $R$, the relation $\mathsf{B}=\limfunc{span}(X)$
holds generically in various universes of design matrices. For later use we
also mention that under Assumption \ref{R_and_X} the test statistic $T_{Het}$
is continuous at every $y\in \mathbb{R}^{n}\backslash \mathsf{B}$.\footnote{%
If Assumption \ref{R_and_X} is violated, then $T_{Het}$ is constant equal to
zero, and hence is trivially continuous everywhere.}

Next, we also consider the classical (i.e., uncorrected) F-test statistic,
i.e.,%
\begin{equation}
T_{uc}(y)=\left\{ 
\begin{array}{cc}
(R\hat{\beta}\left( y\right) -r)^{\prime }\left( \hat{\sigma}^{2}(y)R\left(
X^{\prime }X\right) ^{-1}R^{\prime }\right) ^{-1}(R\hat{\beta}\left(
y\right) -r) & \text{if }y\notin \limfunc{span}(X), \\ 
0 & \text{if }y\in \limfunc{span}(X),%
\end{array}%
\right.  \label{T_uncorr}
\end{equation}%
where $\hat{\sigma}^{2}(y)=\hat{u}\left( y\right) ^{\prime }\hat{u}\left(
y\right) /(n-k)\geq 0$ (which vanishes if and only if $y\in \limfunc{span}%
(X) $). Our choice to set $T_{uc}(y)=0$ for $y\in \limfunc{span}(X)$ has no
import on the rejection probabilities of the (non-bootstrap-based) tests
obtained from $T_{uc}$, since $\limfunc{span}(X)$ is a $\lambda _{\mathbb{R}%
^{n}}$-null set as a consequence of the maintained assumption that $k<n$
(and since the measures $P_{\mu ,\sigma ^{2}\Sigma }$ are absolutely
continuous). It will turn out also not to affect our results for
bootstrap-based tests obtained from $T_{uc}$. For reasons of comparability
with (\ref{T_het}) we have chosen not to normalize the numerator in (\ref%
{T_uncorr}) by $q$, the number of restrictions to be tested, as is often
done in the definition of the classical F-test statistic. This also has no
import on the results as the bootstrap automatically adapts to scaling. For
later use we also mention that the test statistic $T_{uc}$ is continuous at
every $y\in \mathbb{R}^{n}\backslash \limfunc{span}(X)$.

\begin{remark}
\label{rem:GM0}(i) The test statistics $T_{Het}$ as well as $T_{uc}$ are $G(%
\mathfrak{M}_{0})$-invariant as is easily seen (with the respective
exceptional sets $\mathsf{B}$ and $\limfunc{span}(X)$ being $G(\mathfrak{M})$%
-invariant).

(ii) Both statistics actually belong to the class of nonsphericity-corrected
F-type test statistics in the sense of Section 5.4 in \cite{PP2016}
(terminology being somewhat unfortunate in the case of $T_{uc}$ as no
correction for the non-sphericity is applied in this case). See Remark \ref%
{F-type} in Appendix \ref{App B} for more discussion.
\end{remark}

\section{Some intuition for the size one results}

The mechanism leading to the size one results put forward formally in the
next section is a concentration effect in the distribution generating the
data $\mathbf{Y}=(y_{1},\ldots ,y_{n})^{\prime }$, entailing a similar
effect in the distribution of $T(\mathbf{Y})$, where we denote by $T$ any of
the test statistics considered in the paper. This concentration effect
emerges when the data-generating process (DGP) is \textquotedblleft strongly
heteroskedastic\textquotedblright . For simplicity, in this section we call
a data-generating process strongly heteroskedastic, if a single observation
has a (relatively) high variance, whereas all other observations have
(relatively) low variance. We shall denote the index corresponding to the
highly varying observation by $i^{\ast }$. Denote the expectation of the
data vector $\mathbf{Y}$ by $\mu _{0}$, where we assume for the discussion
in this section that $\mu _{0}\in \mathfrak{M}_{0}$, i.e., that the null
hypothesis is satisfied.

We now provide a nonrigorous explanation of the above-mentioned
concentration effect and how it leads to the size one results:

\begin{enumerate}
\item If the DGP is strongly heteroskedastic, only the single highly varying
observation $y_{i^{\ast }}$ will substantially deviate from its expectation $%
\mu _{0}^{(i^{\ast })}$, whereas all other observations will be very close
to their expectations. That is, under such a DGP we approximately have%
\begin{equation*}
\mathbf{Y}\approx \mu _{0}+(y_{i^{\ast }}-\mu _{0}^{(i^{\ast })})e_{i^{\ast
}}(n),
\end{equation*}%
where we recall that $e_{i^{\ast }}(n)$ is the $i^{\ast }$-th $n\times 1$
standard basis vector. That is, essentially, the data are concentrated on a
one-dimensional affine subspace of the sample space $\mathbb{R}^{n}$.

\item Invariance properties of $T$ common to all test statistics used in
this paper (and in practice) imply that%
\begin{equation*}
T(\mu _{0}+ce_{i^{\ast }}(n))=T(\mu _{0}+e_{i^{\ast }}(n))\text{ for every }%
c\neq 0.
\end{equation*}%
That is, the test statistic under consideration is essentially constant on
the one-dimensional affine subspace just obtained in the previous item.

\item Combining the two previous observations (and ignoring the case where $%
y_{i^{\ast }}=\mu _{0}^{(i^{\ast })}$), suggests that for strongly
heteroskedastic DGPs we have%
\begin{equation*}
T(\mathbf{Y})\approx T(\mu _{0}+e_{i^{\ast }}(n)).
\end{equation*}%
That is, essentially, the distribution of the test statistic collapses at
the value $T(\mu _{0}+e_{i^{\ast }}(n))$.
\end{enumerate}

Now, recall that a wild bootstrap-based test rejects if the test statistic
evaluated at the data $T(\mathbf{Y})$ exceeds the bootstrap critical value.
This bootstrap critical value is a $1-\alpha $ quantile of the distribution
of the test statistic, but now induced by the distribution that corresponds
to a bootstrap scheme $\mathbf{Y}^{\ast }$, say. In general, the
distribution of $T(\mathbf{Y}^{\ast })$ depends on two sources of
randomness: first, the DGP itself, and second, the randomization mechanism
used to generate the bootstrap scheme $\mathbf{Y}^{\ast }$. Making use of
the concentration mechanism outlined above, one can, for the class of
bootstrap schemes considered, however, show that the dependence on the DGP
essentially vanishes for strongly heteroskedastic DGPs. That is, the
distribution function of $T(\mathbf{Y}^{\ast })$ approximately equals a
distribution function $\digamma _{i^{\ast }}$, say, which depends on $%
i^{\ast }$, but on no other aspect of the DGP. The bootstrap critical value
is then a $1-\alpha $ quantile of $\digamma _{i^{\ast }}$. Recalling from
3.~that for strongly heteroskedastic DGPs $T(\mathbf{Y})\approx T(\mu
_{0}+e_{i^{\ast }}(n))$, it follows that for such DGPs the event that the
wild bootstrap based test rejects the null hypothesis essentially coincides
with the event that $T(\mu _{0}+e_{i^{\ast }}(n))$ exceeds the $1-\alpha $
quantile of $\digamma _{i^{\ast }}$. Both $T(\mu _{0}+e_{i^{\ast }}(n))$ and 
$\digamma _{i^{\ast }}$ are non-random. Hence (recall that we are operating
under the null hypothesis) for strongly heteroskedastic DGPs the test will
have a rejection probability close to one if $\digamma _{i^{\ast }}(T(\mu
_{0}+e_{i^{\ast }}(n)))>1-\alpha $.\footnote{%
More precisely, if the left hand side limit $\digamma _{i^{\ast }}(T(\mu
_{0}+e_{i^{\ast }}(n))-)$ exceeds $1-\alpha $. We ignore this technical
detail here for the sake of simplicity.} In the above argument $i^{\ast }$
was fixed. Varying $i^{\ast }\in \{1,\ldots ,n\}$, we finally come to the
conclusion that the maximal rejection probability under the null will be
close to $1$ in case%
\begin{equation*}
\max_{i^{\ast }=1,\ldots ,n}\digamma _{i^{\ast }}(T(\mu _{0}+e_{i^{\ast
}}(n)))>1-\alpha .
\end{equation*}%
In other words, the bootstrap-based test under consideration will have
rejection probabilities close to one for all levels of significance
satisfying%
\begin{equation*}
\alpha >1-\max_{i^{\ast }=1,\ldots ,n}\digamma _{i^{\ast }}\left( T(\mu
_{0}+e_{i^{\ast }}(n))\right) .
\end{equation*}%
The quantity to the right is closely related to our constants $\vartheta $.
Note that the above reasoning is nonrigorous and, in particular, does not
take into consideration some technical subtleties that arise in the just
given approximation arguments and that we have tacitly ignored in the
preceding discussion. Therefore, the expressions for the constants $%
\vartheta $ we arrive at in the theorems in the subsequent section are
somewhat more complicated, albeit the underlying intuition is the same.

As transpires from the preceding heuristic discussion, the method for
establishing the size one results given in the next section relies on the
assumption that the heteroskedasticity model employed is rich enough to
approximate extreme cases of strongly heteroskedastic DGPs, namely the ones
where all but one observation have zero variance, arbitrarily well. This is
certainly so for the leading case of the heteroskedasticity model $\mathfrak{%
C}_{Het}$, which describes agnosticism about the form of heteroskedasticity.
Therefore, the results in the next section are presented for this case, and
a discussion to which other heteroskedasticity models these results
generalize is given in Section \ref{gen}.

If one maintains a heteroskedasticity model that does not allow one to
approximate any of the above mentioned extreme cases of strongly
heteroskedastic DGPs (such as, e.g., the heteroskedasticity model $\mathfrak{%
C}_{Het}(a)$ which consists of all error-covariance matrices in $\mathfrak{C}%
_{Het}$ with diagonal elements bounded from below by $a>0$), then the method
of proof underlying our size one results no longer is applicable. However,
this does \emph{not }imply that the size of a bootstrap-based test over $%
\mathfrak{C}_{Het}(a)$ is about right: Since the rejection probabilities are
continuous in the parameters (in particular, in $\Sigma $), the size over $%
\mathfrak{C}_{Het}(a)$ will be much larger than the nominal significance
level at least for small $a$ (in fact, it will be close to one if $a$ is
sufficiently small) in any situation where the size over $\mathfrak{C}_{Het}$
equals one (e.g., in the situations described in the theorems further
below). [The actual size over $\mathfrak{C}_{Het}(a)$ depends on the chosen
bound $a$ and on the design matrix, the hypothesis to be tested, the test
statistic, and also on the bootstrap scheme used.] Furthermore, the bound $a$
has to be decided on prior to the data analysis and is part of modeling the
form of heteroskedasticity. It is difficult to see how one would come up
with a reasonable bound $a$ in practice: if $a$ is chosen to be very small,
this may result in a heteroskedasticity model under which the tests are
still severely oversized as just discussed, while choosing $a$ large will
typically not be defendable as it presumes considerable knowledge about the
admissible forms of heteroskedasticity.

\section{Size one results\label{Theory}}

In this section we provide sufficient conditions for the size of
bootstrap-based heteroskedasticity robust tests to be equal to one when the
heteroskedasticity model is $\mathfrak{C}_{Het}$, which is the largest
possible heteroskedasticity model and which reflects agnosticism regarding
the form of heteroskedasticity. For extensions to other heteroskedasticity
models see Section \ref{gen}. We next discuss the bootstrap schemes that
will be considered and which all are based on the wild bootstrap idea. The
first bootstrap scheme is given by

\begin{equation}
y^{\ast }(y,\xi )=X\tilde{\beta}_{\mathfrak{M}_{0}}(y)+\limfunc{diag}(\xi
)(y-X\tilde{\beta}_{\mathcal{A}}(y))  \label{boot_1}
\end{equation}%
for every $y\in \mathbb{R}^{n}$, where $\mathcal{A}$ will always be an
affine subspace of $\mathbb{R}^{n}$ satisfying $\mathfrak{M}_{0}\subseteq 
\mathcal{A}\subseteq \limfunc{span}(X)$, and where $\xi $ is a draw from $%
\Xi $, a given (Borel) probability measure on $\mathbb{R}^{n}$. Typical
choices in the literature are $\mathcal{A}=\mathfrak{M}_{0}$, i.e., one uses
restricted residuals in the wild bootstrap, or $\mathcal{A}=\limfunc{span}%
(X) $, in which case unrestricted residuals are used.\footnote{%
Because all test statistics (and associated exceptional sets) considered
below are at least $G(\mathfrak{M}_{0})$-invariant (see Remarks \ref{rem:GM0}
and \ref{rem:tildeGM0}), the bootstrap scheme (\ref{boot_1}) can be replaced
by $y^{\ast \ast }(y,\xi )=\mu _{0}+\limfunc{diag}(\xi )(y-X\tilde{\beta}_{%
\mathcal{A}}(y))$ for an arbitrary value $\mu _{0}\in \mathfrak{M}_{0}$
without affecting the bootstrapped test statistic.} In practice only these
two choices will typically arise, but the theory given below covers the more
general case where $\mathfrak{M}_{0}\subseteq \mathcal{A}\subseteq \limfunc{%
span}(X)$ at no extra cost. The measure $\Xi $ may depend on observable
quantities like, e.g., $X$, $R$, or $\mathcal{A}$, but not on $y$. For
example, $\Xi $ could be the $n$-fold product of Mammen or Rademacher
distributions, but other choices (e.g., ones obtained by modifying the
aforementioned distributions by weights, or non-discrete distributions) are
also covered. See Section \ref{sec_Num} for some examples. For the
theoretical results in this section there is no need to specify a particular
form of $\Xi $.\footnote{%
Suppose $\Xi $ is the empirical distribution of $B$ draws (possibly modified
by weights) from an underlying distribution $\Xi _{0}$, which will often be
the case if $n$ is large and the ideal bootstrap using $\Xi _{0}$ is
infeasible. In this case $\Xi $ is strictly speaking a random probability
measure (depending on the particular sample of size $B$ drawn from $\Xi _{0}$%
) and the bootstrap-based tests also depend on this sample. However, working
conditionally on this sample, brings us back into the current framework.}

The second bootstrap scheme differs from the first one only insofar as
centering is at the unrestricted estimator $X\hat{\beta}(y)$ rather than at
the restricted estimator $X\tilde{\beta}_{\mathfrak{M}_{0}}(y)$. That is,
the second bootstrap scheme is given by 
\begin{equation}
y^{\maltese }(y,\xi )=X\hat{\beta}(y)+\limfunc{diag}(\xi )(y-X\tilde{\beta}_{%
\mathcal{A}}(y))  \label{boot_2}
\end{equation}%
for every $y\in \mathbb{R}^{n}$. Note that $y^{\ast }(y,\xi )$ as well as $%
y^{\maltese }(y,\xi )$ depend also on the choice of $\mathcal{A}$, but we
shall not show this dependence in the notation.

\subsection{Bootstrap-based tests derived from $T_{Het}$ and $T_{uc}$ \label%
{unrestr_res}}

In the subsequent theorems $\Xi $ is always a (Borel) probability measure on 
$\mathbb{R}^{n}$, and $\mathcal{A}$ is an affine subspace of $\mathbb{R}^{n}$
satisfying $\mathfrak{M}_{0}\subseteq \mathcal{A}\subseteq \limfunc{span}(X)$%
. If we use the first bootstrap scheme, i.e., (\ref{boot_1}), the
bootstrapped test statistic corresponding to $T_{Het}$ is given by $%
T_{Het}^{\ast }$, where $T_{Het}^{\ast }:\mathbb{R}^{n}\times \mathbb{R}%
^{n}\rightarrow \mathbb{R}$ is defined via 
\begin{equation*}
T_{Het}^{\ast }(y,\xi )=T_{Het}\left( y^{\ast }(y,\xi )\right) .
\end{equation*}%
Furthermore, for every $y\in \mathbb{R}^{n}$ denote the distribution
function of the bootstrapped test statistic under $\Xi $ by $F_{Het,y}$,
i.e., $F_{Het,y}(t)=\Xi \mathbb{(}T_{Het}^{\ast }(y,\xi )\leq t)$ for $t\in 
\mathbb{R}$. For reasons that are discussed further below, we also need to
consider a modification of $T_{Het}$ defined by $T_{Het}^{\blacktriangle
}\left( y\right) =T_{Het}\left( y\right) $ if $y\notin \mathsf{B}$ and $%
T_{Het}^{\blacktriangle }(y)=\infty $ otherwise. Its bootstrapped version is
then given by $T_{Het}^{\blacktriangle ,\ast }(y,\xi
)=T_{Het}^{\blacktriangle }\left( y^{\ast }(y,\xi )\right) $. Similarly as
before, for every $y\in \mathbb{R}^{n}$ we denote its distribution function
under $\Xi $ by $F_{Het,y}^{\blacktriangle }$, i.e., $F_{Het,y}^{%
\blacktriangle }(t)=\Xi \mathbb{(}T_{Het}^{\blacktriangle ,\ast }(y,\xi
)\leq t)$ for $t\in \mathbb{R\cup \{\infty \}}$.

\begin{theorem}
\label{Theo_Het_unrestr}Suppose Assumption \ref{R_and_X} holds.

(a) For every $\alpha \in (0,1)$, let $f_{Het,1-\alpha }(y)$ denote a $%
(1-\alpha )$-quantile of $F_{Het,y}$. Define $\vartheta _{Het}=1-\max
(\vartheta _{1,Het},\vartheta _{2,Het})$, where%
\begin{equation}
\vartheta _{1,Het}=\max_{\substack{ i=1,\ldots ,n,  \\ e_{i}(n)\notin 
\mathsf{B}}}\Xi \left( \left\{ \xi :T_{Het}^{\ast }(\mu _{0}+e_{i}(n),\xi
)<T_{Het}(\mu _{0}+e_{i}(n)),y^{\ast }(\mu _{0}+e_{i}(n),\xi )\notin \mathsf{%
B}\right\} \right)  \label{theta1_Het}
\end{equation}%
and%
\begin{equation}
\vartheta _{2,Het}=\max_{\substack{ i=1,\ldots ,n,  \\ e_{i}(n)\in \limfunc{%
span}(X),R\hat{\beta}(e_{i}(n))\neq 0}}\Xi \left( \left\{ \xi :y^{\ast }(\mu
_{0}+e_{i}(n),\xi )\notin \mathsf{B}\right\} \right)  \label{theta2_Het}
\end{equation}%
for some $\mu _{0}\in \mathfrak{M}_{0}$, with the convention that $\vartheta
_{1,Het}=0$ ($\vartheta _{2,Het}=0$, respectively) if the index set in the
maximum operator in (\ref{theta1_Het}) ((\ref{theta2_Het}), respectively) is
empty. Then neither $\vartheta _{1,Het}$ nor $\vartheta _{2,Het}$ depend on
the choice of $\mu _{0}\in \mathfrak{M}_{0}$. Furthermore, for every $\alpha
\in (0,1)$ such that $\alpha >\vartheta _{Het}$ holds, we have%
\begin{equation}
\sup_{\Sigma \in \mathfrak{C}_{Het}}P_{\mu _{0},\sigma ^{2}\Sigma }\left(
T_{Het}\geq f_{Het,1-\alpha }\right) \geq \sup_{\Sigma \in \mathfrak{C}%
_{Het}}P_{\mu _{0},\sigma ^{2}\Sigma }\left( T_{Het}>f_{Het,1-\alpha
}\right) =1  \label{eqn:bootsize_het}
\end{equation}%
for every $\mu _{0}\in \mathfrak{M}_{0}$ and every $0<\sigma ^{2}<\infty $
(where the probabilities in (\ref{eqn:bootsize_het}) are to be interpreted
as inner probabilities\footnote{%
This allows one to ignore measurability issues regarding $f_{Het,1-\alpha }$.%
}).

(b) For every $\alpha \in (0,1)$, let $f_{Het,1-\alpha }^{\blacktriangle
}(y) $ denote a $(1-\alpha )$-quantile of $F_{Het,y}^{\blacktriangle }$.
Then, with $\vartheta _{Het}$ defined in Part (a), for every $\alpha \in
(0,1)$ such that $\alpha >\vartheta _{Het}$ holds, we have%
\begin{equation}
\sup_{\Sigma \in \mathfrak{C}_{Het}}P_{\mu _{0},\sigma ^{2}\Sigma }\left(
T_{Het}\geq f_{Het,1-\alpha }^{\blacktriangle }\right) \geq \sup_{\Sigma \in 
\mathfrak{C}_{Het}}P_{\mu _{0},\sigma ^{2}\Sigma }\left(
T_{Het}>f_{Het,1-\alpha }^{\blacktriangle }\right) =1
\label{eqn:bootsize_het_triangle}
\end{equation}%
for every $\mu _{0}\in \mathfrak{M}_{0}$ and every $0<\sigma ^{2}<\infty $
(where the probabilities in (\ref{eqn:bootsize_het_triangle}) are to be
interpreted as inner probabilities).
\end{theorem}

Part (a) of the preceding theorem implies that for every nominal
significance level $\alpha >\vartheta _{Het}$ the size (over $\mathfrak{C}%
_{Het}$) of the bootstrap-based test derived from $T_{Het}$ is equal to $1$
and thus is inflated (and this is true whether the bootstrap-based test uses
the rejection region $\left\{ y:T_{Het}(y)\geq f_{Het,1-\alpha }(y)\right\} $
or $\left\{ y:T_{Het}(y)>f_{Het,1-\alpha }(y)\right\} $).\footnote{%
We note that in principle it is conceivable that these two rejection regions
have different probabilities under $P_{\mu _{0},\sigma ^{2}\Sigma }$.} Note
that the lower bound $\vartheta _{Het}$ is observable and can be computed,
see Section \ref{rems} for some more detail. As we shall see from the
numerical results in Section \ref{sec_Num}, the lower bound $\vartheta _{Het}
$ can be quite small, the results in the theorem thus covering standard
choices for $\alpha $ such as $\alpha =0.05$. A consequence of Theorem \ref%
{Theo_Het_unrestr} thus is, in particular, that there is in general no
guarantee for bootstrap-based tests derived from $T_{Het}$ (or from the
other statistics considered in the theorems further below), conducted at a
nominal significance level $\alpha $, to be truly level $\alpha $ tests.
Although trivial, we note that Theorem \ref{Theo_Het_unrestr} provides only
a \emph{sufficient} condition for size being equal to one and thus, in case $%
\alpha \leq \vartheta _{Het}$ holds, the size of the bootstrap-based test
may nevertheless be much larger than $\alpha $ (and may perhaps even be
equal to $1$).

The significance of Part (b) of the theorem is as follows: Recall from Lemma %
\ref{lem_B} that under Assumption \ref{R_and_X} the way $T_{Het}$ is defined
on $\mathsf{B}$ is immaterial for the rejection probabilities of
(non-bootstrap-based) tests obtained from this test statistic since the set $%
\mathsf{B}$ is a Lebesgue null set and since the probability measures in (%
\ref{lm2}) are all absolutely continuous w.r.t. Lebesgue measure; in
particular, the (non-bootstrap-based) tests derived from $T_{Het}$ and $%
T_{Het}^{\blacktriangle }$ have the same rejection probabilities. However,
when it comes to the \emph{bootstrapped} test statistics, the situation
becomes more complicated as $\Xi $ often will be a discrete measure. That
is, it is a priori conceivable that the value we assign to the test
statistic on the set $\mathsf{B}$ may have an effect on the \emph{%
bootstrapped} test statistic and thus on the $(1-\alpha )$-quantile computed
from it; in particular, it might be that an assignment of a value different
from zero on the set $\mathsf{B}$ may lead to a larger $(1-\alpha )$%
-quantile. This then raises the question, whether a bootstrap-based test
that uses such a (potentially) larger $(1-\alpha )$-quantile may have a
smaller size than when the quantile $f_{Het,1-\alpha }$ is being used.
Within the context of the theorem, Part (b) answers this in the negative by
showing that, even if one defines the \emph{bootstrapped} test statistic as $%
\infty $ on the event where the bootstrap sample $y^{\ast }(y,\xi )$ falls
into the exceptional set $\mathsf{B}$ and uses a resulting $(1-\alpha )$%
-quantile, the bootstrap-based test again has size $1$ under the same
condition on $\alpha $. [As any other way of defining the bootstrapped test
statistic on the event $y^{\ast }(y,\xi )\in \mathsf{B}$ obviously leads to $%
(1-\alpha )$-quantiles not larger than an (appropriately chosen) $(1-\alpha )
$-quantile of $F_{Het,y}^{\blacktriangle }$, Part (b) covers also any such
alternative definition of the bootstrapped test statistic.]\footnote{%
An alternative approach, which -- if successful -- would make considering
Part (b) obsolete, would be to try to show that the set of $y^{\prime }s$
for which $T_{Het}^{\blacktriangle ,\ast }$ and $T_{Het}^{\ast }$ coincide $%
\Xi $-a.e., and thus their quantiles coincide, is the complement of a
Lebesgue null set. While this alternative approach actually can be shown to
work in some special cases, it does not so in general, as can be seen from
examples.} For additional discussion see also Remark \ref{rem:null}.

We also stress that the results in the preceding theorem hold for \emph{any}
choice $f_{Het,1-\alpha }$ ($f_{Het,1-\alpha }^{\blacktriangle }(y)$,
respectively) from the set of $(1-\alpha )$-quantiles of $F_{Het,y}$ ($%
F_{Het,y}^{\blacktriangle }$, respectively).\footnote{\label{qu}Suppose $%
0<\delta <1$ and $F$ is a cdf defined on $\mathbb{R}$ ($\mathbb{R\cup
\{\infty \}}$, respectively). An element $q\in \mathbb{R}$ ($q\in \mathbb{%
R\cup \{\infty \}}$, respectively) is said to be a $\delta $-quantile of $F$
iff it satisfies $F(q)\geq \delta \geq F(q-)$, where $F(q-)$ denotes the
left-hand limit of $F$ at $q$. Note that $q$ need not be unique in general.
There is always a smallest and a largest $\delta $-quantile among all $%
\delta $-quantiles. The smallest one is given by $F^{-1}(\delta $), where $%
F^{-1}$ is the "generalized" inverse of $F$. If $\delta $ does not belong to
the range of $F$, then $F^{-1}(\delta $) is also the largest $\delta $%
-quantile. Otherwise, the largest $\delta $-quantile is given by $\sup
\left\{ x\in \mathbb{R}:F(x)=\delta \right\} $ ($\sup \left\{ x\in \mathbb{%
R\cup \{\infty \}}:F(x)=\delta \right\} $, respectively), which may or may
not coincide with $F^{-1}(\delta )$.}

Furthermore, we note that the preceding theorem holds with the \emph{same}
lower bound $\vartheta _{Het}$ for a much larger class of error
distributions than just Gaussian errors (an assumption we have made only for
convenience), see Section \ref{gen}. Hence, in this sense the lower bound $%
\vartheta _{Het}$ is \textquotedblleft distribution free\textquotedblright .

We next turn to the test statistic $T_{uc}$. Again using the first bootstrap
scheme, the bootstrapped test statistic is then given by $T_{uc}^{\ast }$
where $T_{uc}^{\ast }:\mathbb{R}^{n}\times \mathbb{R}^{n}\rightarrow \mathbb{%
R}$ is defined via 
\begin{equation*}
T_{uc}^{\ast }(y,\xi )=T_{uc}\left( y^{\ast }(y,\xi )\right) .
\end{equation*}%
Furthermore, for every $y\in \mathbb{R}^{n}$ denote the distribution
function of the bootstrapped test statistic under $\Xi $ by $F_{uc,y}$,
i.e., $F_{uc,y}(t)=\Xi \mathbb{(}T_{uc}^{\ast }(y,\xi )\leq t)$ for $t\in 
\mathbb{R}$. As before, we also need to consider the modification of $T_{uc}$
defined by $T_{uc}^{\blacktriangle }\left( y\right) =T_{uc}\left( y\right) $
if $y\notin \limfunc{span}(X)$ and $T_{uc}^{\blacktriangle }(y)=\infty $
otherwise. Its bootstrapped version is then given by $T_{uc}^{\blacktriangle
,\ast }(y,\xi )=T_{uc}^{\blacktriangle }\left( y^{\ast }(y,\xi )\right) $.
Similarly as before, for every $y\in \mathbb{R}^{n}$ we denote its
distribution function under $\Xi $ by $F_{uc,y}^{\blacktriangle }$, i.e., $%
F_{uc,y}^{\blacktriangle }(t)=\Xi \mathbb{(}T_{uc}^{\blacktriangle ,\ast
}(y,\xi )\leq t)$ for $t\in \mathbb{R\cup \{\infty \}}$.

\begin{theorem}
\label{Theo_uc_unrestr}(a) For every $\alpha \in (0,1)$, let $f_{uc,1-\alpha
}(y)$ denote a $(1-\alpha )$-quantile of $F_{uc,y}$. Define $\vartheta
_{uc}=1-\max (\vartheta _{1,uc},\vartheta _{2,uc})$, where%
\begin{equation}
\vartheta _{1,uc}=\max_{\substack{ i=1,\ldots ,n,  \\ e_{i}(n)\notin 
\limfunc{span}(X)}}\Xi \left( \left\{ \xi :T_{uc}^{\ast }(\mu
_{0}+e_{i}(n),\xi )<T_{uc}(\mu _{0}+e_{i}(n)),y^{\ast }(\mu
_{0}+e_{i}(n),\xi )\notin \limfunc{span}(X)\right\} \right)
\label{theta1_uncorr}
\end{equation}%
and%
\begin{equation}
\vartheta _{2,uc}=\max_{\substack{ i=1,\ldots ,n,  \\ e_{i}(n)\in \limfunc{%
span}(X),R\hat{\beta}(e_{i}(n))\neq 0}}\Xi \left( \left\{ \xi :y^{\ast }(\mu
_{0}+e_{i}(n),\xi )\notin \limfunc{span}(X)\right\} \right)
\label{theta2_uncorr}
\end{equation}%
for some $\mu _{0}\in \mathfrak{M}_{0}$, with the convention that $\vartheta
_{2,uc}=0$ if the index set in the maximum operator in (\ref{theta2_uncorr})
is empty.\footnote{%
Note that the index set in the maximum operator in (\ref{theta1_uncorr}) can
not be empty since we have assumed $k<n$.} Then neither $\vartheta _{1,uc}$
nor $\vartheta _{2,uc}$ depend on the choice of $\mu _{0}\in \mathfrak{M}%
_{0} $. Furthermore, for every $\alpha \in (0,1)$ such that $\alpha
>\vartheta _{uc}$ holds, we have%
\begin{equation}
\sup_{\Sigma \in \mathfrak{C}_{Het}}P_{\mu _{0},\sigma ^{2}\Sigma }\left(
T_{uc}\geq f_{uc,1-\alpha }\right) \geq \sup_{\Sigma \in \mathfrak{C}%
_{Het}}P_{\mu _{0},\sigma ^{2}\Sigma }\left( T_{uc}>f_{uc,1-\alpha }\right)
=1  \label{eqn:bootsize_uncorr}
\end{equation}%
for every $\mu _{0}\in \mathfrak{M}_{0}$ and every $0<\sigma ^{2}<\infty $
(where the probabilities in (\ref{eqn:bootsize_uncorr}) are to be
interpreted as inner probabilities).

(b) For every $\alpha \in (0,1)$, let $f_{uc,1-\alpha }^{\blacktriangle }(y)$
denote a $(1-\alpha )$-quantile of $F_{uc,y}^{\blacktriangle }$. Then, with $%
\vartheta _{uc}$ defined in Part (a), for every $\alpha \in (0,1)$ such that 
$\alpha >\vartheta _{uc}$ holds, we have%
\begin{equation}
\sup_{\Sigma \in \mathfrak{C}_{Het}}P_{\mu _{0},\sigma ^{2}\Sigma }\left(
T_{uc}\geq f_{uc,1-\alpha }^{\blacktriangle }\right) \geq \sup_{\Sigma \in 
\mathfrak{C}_{Het}}P_{\mu _{0},\sigma ^{2}\Sigma }\left(
T_{uc}>f_{uc,1-\alpha }^{\blacktriangle }\right) =1
\label{eqn:bootsize_uncorr_triangle}
\end{equation}%
for every $\mu _{0}\in \mathfrak{M}_{0}$ and every $0<\sigma ^{2}<\infty $
(where the probabilities in (\ref{eqn:bootsize_uncorr_triangle}) are to be
interpreted as inner probabilities).
\end{theorem}

Mutatis mutandis, a discussion similar to the one given subsequently to
Theorem \ref{Theo_Het_unrestr} also applies here.

So far we have only considered the bootstrap scheme (\ref{boot_1}). We now
turn to the second bootstrap scheme given by (\ref{boot_2}). Here the
bootstrapped version of $T_{Het}$ is given by%
\begin{equation*}
T_{Het}^{\maltese }(y,\xi )=(R\hat{\beta}\left( y^{\maltese }(y,\xi )\right)
-R\hat{\beta}\left( y\right) )^{\prime }\hat{\Omega}_{Het}^{-1}\left(
y^{\maltese }(y,\xi )\right) (R\hat{\beta}\left( y^{\maltese }(y,\xi
)\right) -R\hat{\beta}\left( y\right) )
\end{equation*}%
if $y^{\maltese }(y,\xi )\notin \mathsf{B}$, and by $T_{Het}^{\maltese
}(y,\xi )=0$ if $y^{\maltese }(y,\xi )\in \mathsf{B}$. And the bootstrapped
version of $T_{uc}$ is given by%
\begin{equation*}
T_{uc}^{\maltese }(y,\xi )=(R\hat{\beta}\left( y^{\maltese }(y,\xi )\right)
-R\hat{\beta}\left( y\right) )^{\prime }\left( \hat{\sigma}^{2}(y^{\maltese
}(y,\xi ))R\left( X^{\prime }X\right) ^{-1}R^{\prime }\right) ^{-1}(R\hat{%
\beta}\left( y^{\maltese }(y,\xi )\right) -R\hat{\beta}\left( y\right) )
\end{equation*}%
if $y^{\maltese }(y,\xi )\not\in \limfunc{span}(X)$, and by $%
T_{uc}^{\maltese }(y,\xi )=0$ if $y^{\maltese }(y,\xi )\in \limfunc{span}(X)$%
. Furthermore, $T_{Het}^{\blacktriangle ,\maltese }(y,\xi )$ and $%
T_{uc}^{\blacktriangle ,\maltese }(y,\xi )$, are defined in exactly the same
way, except that $T_{Het}^{\blacktriangle ,\maltese }(y,\xi )=\infty $ if $%
y^{\maltese }(y,\xi )\in \mathsf{B}$ and that $T_{uc}^{\blacktriangle
,\maltese }(y,\xi )=\infty $ if $y^{\maltese }(y,\xi )\in \limfunc{span}(X)$.

We will show in the next lemma that $T_{Het}^{\maltese }(y,\xi )$ coincides
with $T_{Het}^{\ast }(y,\xi )$, and that the same is true for $%
T_{uc}^{\maltese }(y,\xi )$ and $T_{uc}^{\ast }(y,\xi )$ (as well as for $%
T_{Het}^{\blacktriangle ,\maltese }(y,\xi )$ and $T_{Het}^{\blacktriangle
,\ast }(y,\xi )$, and $T_{uc}^{\blacktriangle ,\maltese }(y,\xi )$ and $%
T_{uc}^{\blacktriangle ,\ast }(y,\xi )$), provided the same affine space $%
\mathcal{A}$ is used in (\ref{boot_1}) and (\ref{boot_2}). As a consequence,
this -- together with Remark \ref{rem:vgk} -- shows that Theorems \ref%
{Theo_Het_unrestr} and \ref{Theo_uc_unrestr} also apply immediately to the
bootstrap-based test when the second bootstrap scheme, i.e., (\ref{boot_2}),
is used (with the same $\mathcal{A}$ and $\Xi $). The lemma is certainly not
new and is a variant of a similar result given as Proposition 1 in \cite%
{vG_K_2002}.

\begin{lemma}
\label{VGK}We have $T_{Het}^{\ast }(y,\xi )=T_{Het}^{\maltese }(y,\xi )$, $%
T_{uc}^{\ast }(y,\xi )=T_{uc}^{\maltese }(y,\xi )$, $T_{Het}^{\blacktriangle
,\ast }(y,\xi )=T_{Het}^{\blacktriangle ,\maltese }(y,\xi )$, and $%
T_{uc}^{\blacktriangle ,\ast }(y,\xi )=T_{uc}^{\blacktriangle ,\maltese
}(y,\xi )$ for every $y\in \mathbb{R}^{n}$ and every $\xi \in \mathbb{R}^{n}$%
. [Here it is understood that both bootstrap schemes are based on the same
affine space $\mathcal{A}$.]
\end{lemma}

\begin{remark}
\label{rem:vgk}Define $\theta _{Het}$ exactly in the same way as $\vartheta
_{Het}$, except that $T_{Het}^{\ast }$ and $y^{\ast }$ are replaced by $%
T_{Het}^{\maltese }$ and $y^{\maltese }$. Similarly define $\theta _{uc}$.
Then $\theta _{Het}=\vartheta _{Het}$ and $\theta _{uc}=\vartheta _{uc}$ in
view of Lemma \ref{VGK} and the fact that $y^{\ast }(y,\xi )\notin \mathsf{B}
$ iff $y^{\maltese }(y,\xi )\notin \mathsf{B}$ and $y^{\ast }(y,\xi )\notin 
\limfunc{span}(X)$ iff $y^{\maltese }(y,\xi )\notin \limfunc{span}(X)$ (note
that $y^{\ast }(y,\xi )-y^{\maltese }(y,\xi )\in \limfunc{span}(X)$ and that 
$\mathsf{B}+\limfunc{span}(X)=\mathsf{B}$).
\end{remark}

\subsection{Bootstrap-based tests derived from $\tilde{T}_{Het}$ and $\tilde{%
T}_{uc}$ \label{restrict_res}}

Based on suggestions in the literature on bootstrapping heteroskedasticity
robust tests, we next consider two further test statistics, which are
versions of $T_{Het}$ and $T_{uc}$ with the only difference that the
covariance matrix estimators used are computed from restricted -- instead of
unrestricted -- residuals. We thus define%
\begin{equation}
\tilde{T}_{Het}\left( y\right) =\left\{ 
\begin{array}{cc}
(R\hat{\beta}\left( y\right) -r)^{\prime }\tilde{\Omega}_{Het}^{-1}\left(
y\right) (R\hat{\beta}\left( y\right) -r) & \text{if }\det \tilde{\Omega}%
_{Het}\left( y\right) \neq 0, \\ 
0 & \text{if }\det \tilde{\Omega}_{Het}\left( y\right) =0,%
\end{array}%
\right.  \label{T_Het_tilde}
\end{equation}%
where $\tilde{\Omega}_{Het}=R\tilde{\Psi}_{Het}R^{\prime }$ and where $%
\tilde{\Psi}_{Het}$ is given by%
\begin{equation*}
\tilde{\Psi}_{Het}\left( y\right) =(X^{\prime }X)^{-1}X^{\prime }\limfunc{%
diag}\left( \tilde{d}_{1}\tilde{u}_{1}^{2}\left( y\right) ,\ldots ,\tilde{d}%
_{n}\tilde{u}_{n}^{2}\left( y\right) \right) X(X^{\prime }X)^{-1},
\end{equation*}%
where the constants $\tilde{d}_{i}>0$ sometimes depend on the design matrix
and on the restriction matrix $R$. Here $\tilde{u}\left( y\right) =y-X\tilde{%
\beta}_{\mathfrak{M}_{0}}(y)=\Pi _{(\mathfrak{M}_{0}^{lin})^{\bot }}(y-\mu
_{0})$, where the last expression does not depend on the choice of $\mu
_{0}\in \mathfrak{M}_{0}$, and where $\tilde{u}_{t}\left( y\right) $ denotes
the $t$-th component of $\tilde{u}\left( y\right) $. Typical choices for $%
\tilde{d}_{i}$ are $\tilde{d}_{i}=1$, $\tilde{d}_{i}=n/(n-(k-q))$, $\tilde{d}%
_{i}=(1-\tilde{h}_{ii})^{-1}$, or $\tilde{d}_{i}=(1-\tilde{h}_{ii})^{-2}$
where $\tilde{h}_{ii}$ denotes the $i$-th diagonal element of the projection
matrix $\Pi _{\mathfrak{M}_{0}^{lin}}$, see, e.g., \cite%
{DavidsonMacKinnon1985}. Another suggestion is $\tilde{d}_{i}=(1-\tilde{h}%
_{ii})^{-\tilde{\delta}_{i}}$ for $\tilde{\delta}_{i}=\min (n\tilde{h}%
_{ii}/(k-q),4)$ with the convention that $\tilde{\delta}_{i}=0$ if $k=q$.%
\footnote{%
Note that in case $k=q$ we have $\tilde{h}_{ii}=0$, and hence $\tilde{d}%
_{i}=1$ regardless of our convention for $\tilde{\delta}_{i}$.} For the last
three choices of $\tilde{d}_{i}$ just given we use the convention that we
set $\tilde{d}_{i}=1$ in case $\tilde{h}_{ii}=1$. Note that $\tilde{h}%
_{ii}=1 $ implies $\tilde{u}_{i}\left( y\right) =0$ for every $y$, and hence
it is irrelevant which real value is assigned to $\tilde{d}_{i}$ in case $%
\tilde{h}_{ii}=1$.\footnote{%
In fact, $\tilde{h}_{ii}=1$ is equivalent to $\tilde{u}_{i}\left( y\right)
=0 $ for every $y$, each of which in turn is equivalent to $e_{i}(n)\in $ $%
\mathfrak{M}_{0}^{lin}$.} The five examples for the weights $\tilde{d}_{i}$
just given correspond to what is often called HC0R-HC4R weights in the
literature.\footnote{%
In the case $k=q$ the HC0R-HC4R weights all coincide ($\tilde{d}_{i}=1$ for
every $i$), and hence result in the same test statistic.}

The subsequent assumption ensures that the set of $y$'s for which $\tilde{%
\Omega}_{Het}\left( y\right) $ is singular is a Lebesgue null set, implying
that our choice of assigning $\tilde{T}_{Het}\left( y\right) $ the value
zero in case $\tilde{\Omega}_{Het}\left( y\right) $ is singular has no
import on the rejection probabilities of the (non-bootstrap-based) tests
obtained from $\tilde{T}_{Het}$ (as the measures $P_{\mu ,\sigma ^{2}\Sigma
} $ are absolutely continuous). As will be seen later, our results for the
corresponding bootstrap-based tests do also not depend on this choice. Also,
as discussed further below, the assumption is in a certain sense unavoidable
when using $\tilde{T}_{Het}$.

\begin{assumption}
\label{R_and_X_tilde}Let $1\leq i_{1}<\ldots <i_{s}\leq n$ denote all the
indices for which $e_{i_{j}}(n)\in \mathfrak{M}_{0}^{lin}$ holds where $%
e_{j}(n)$ denotes the $j$-th standard basis vector in $\mathbb{R}^{n}$. If
no such index exists, set $s=0$. Let $X^{\prime }\left( \lnot (i_{1},\ldots
i_{s})\right) $ denote the matrix which is obtained from $X^{\prime }$ by
deleting all columns with indices $i_{j}$, $1\leq i_{1}<\ldots <i_{s}\leq n$
(if $s=0$ no column is deleted). Then $\limfunc{rank}\left( R(X^{\prime
}X)^{-1}X^{\prime }\left( \lnot (i_{1},\ldots i_{s})\right) \right) =q$
holds.
\end{assumption}

Observe that this assumption only depends on $X$ and $R$ and hence can be
checked. Obviously, a simple sufficient condition for Assumption \ref%
{R_and_X_tilde} to hold is that $s=0$ (i.e., that $e_{j}(n)\notin \mathfrak{M%
}_{0}^{lin}$ for all $j$), a generically satisfied condition. Furthermore,
we introduce the matrix%
\begin{eqnarray}
\tilde{B}(y) &=&R(X^{\prime }X)^{-1}X^{\prime }\limfunc{diag}\left( \tilde{u}%
_{1}(y),\ldots ,\tilde{u}_{n}(y)\right)  \notag \\
&=&R(X^{\prime }X)^{-1}X^{\prime }\limfunc{diag}\left( e_{1}^{\prime }(n)\Pi
_{(\mathfrak{M}_{0}^{lin})^{\bot }}(y-\mu _{0}),\ldots ,e_{n}^{\prime
}(n)\Pi _{(\mathfrak{M}_{0}^{lin})^{\bot }}(y-\mu _{0})\right) .
\label{B_tilde_matrix}
\end{eqnarray}%
Note that this matrix does not depend on the choice of $\mu _{0}\in 
\mathfrak{M}_{0}$. The following lemma collects some important properties of 
$\tilde{\Omega}_{Het}$ and $\mathsf{\tilde{B}}$ (defined in that lemma). Its
proof is given in Appendix \ref{App C}.

\begin{lemma}
\label{lem:B_tilde}(a) $\tilde{\Omega}_{Het}\left( y\right) $ is nonnegative
definite for every $y\in \mathbb{R}^{n}$.

(b) $\tilde{\Omega}_{Het}\left( y\right) $ is singular (zero, respectively)
if and only if $\limfunc{rank}(\tilde{B}(y))<q$ ($\tilde{B}(y)=0$,
respectively).

(c) The set $\mathsf{\tilde{B}}$ given by $\{y\in \mathbb{R}^{n}:\limfunc{%
rank}(\tilde{B}(y))<q\}$ (or, in view of (b), equivalently given by $\{y\in 
\mathbb{R}^{n}:\det (\tilde{\Omega}_{Het}\left( y\right) )=0\}$) is either a 
$\lambda _{\mathbb{R}^{n}}$-null set or the entire sample space $\mathbb{R}%
^{n}$. The latter occurs if and only if Assumption \ref{R_and_X_tilde} is
violated (in which case the test based on $\tilde{T}_{Het}$ becomes trivial,
as then $\tilde{T}_{Het}$ is identically zero).

(d) Suppose Assumption \ref{R_and_X_tilde} holds. Then for every $\mu
_{0}\in \mathfrak{M}_{0}$ the set $\mathsf{\tilde{B}}-\mu _{0}$ is a finite
union of proper linear subspaces; in case $q=1$, $\mathsf{\tilde{B}}-\mu
_{0} $ is even a proper linear subspace itself. \footnote{%
Consequently, $\mathsf{\tilde{B}}$ is a finite union of proper affine
subspaces, and is a proper affine subspace itself in case $q=1$.}$^{\text{,}%
} $\footnote{%
If Assumption \ref{R_and_X_tilde} is violated, then $\mathsf{\tilde{B}}-\mu
_{0}=\mathsf{\tilde{B}}=\mathbb{R}^{n}$ in view of Part (c).} [Note that $%
\mathsf{\tilde{B}}-\mu _{0}$ does not depend on the choice of $\mu _{0}\in 
\mathfrak{M}_{0}$. In particular, if $r=0$, i.e., if $\mathfrak{M}_{0}$ is
linear, we thus may set $\mu _{0}=0$.]

(e) $\mathsf{\tilde{B}}$ is a closed set and contains $\mathfrak{M}_{0}$.
Also $\mathsf{\tilde{B}}$ is $G(\mathfrak{M}_{0})$-invariant, and in
particular $\mathsf{\tilde{B}}+\mathfrak{M}_{0}^{lin}=\mathsf{\tilde{B}}$.
\end{lemma}

In light of Part (c) of the lemma, we see that Assumption \ref{R_and_X_tilde}
is a natural and unavoidable condition if one wants to obtain a sensible
test from $\tilde{T}_{Het}$.\footnote{%
If this assumption is violated then $\tilde{T}_{Het}$ is identically zero,
an uninteresting trivial case.} Furthermore, note that if $\mathsf{\tilde{B}}%
=\mathfrak{M}_{0}$ is true, then Assumption \ref{R_and_X_tilde} must be
satisfied (since $\mathfrak{M}_{0}$ is a $\lambda _{\mathbb{R}^{n}}$-null
set as $k-q<n$ is always the case). For later use we also mention that under
Assumption \ref{R_and_X_tilde} the statistic $\tilde{T}_{Het}$ is continuous
at every $y\in \mathbb{R}^{n}\backslash \mathsf{\tilde{B}}$.\footnote{%
If Assumption \ref{R_and_X_tilde} is violated, then $\tilde{T}_{Het}$ is
constant equal to zero, and hence trivially continuous everywhere.}

We finally consider in analogy with $T_{uc}$%
\begin{equation}
\tilde{T}_{uc}(y)=\left\{ 
\begin{array}{cc}
(R\hat{\beta}\left( y\right) -r)^{\prime }\left( \tilde{\sigma}%
^{2}(y)R\left( X^{\prime }X\right) ^{-1}R^{\prime }\right) ^{-1}(R\hat{\beta}%
\left( y\right) -r) & \text{if }y\notin \mathfrak{M}_{0}, \\ 
0 & \text{if }y\in \mathfrak{M}_{0},%
\end{array}%
\right.  \label{T_uncorr_tilde}
\end{equation}%
where $\tilde{\sigma}^{2}(y)=\tilde{u}\left( y\right) ^{\prime }\tilde{u}%
\left( y\right) /(n-(k-q))\geq 0$ (which vanishes if and only if $y\in 
\mathfrak{M}_{0}$). Of course, our choice to set $\tilde{T}_{uc}(y)=0$ for $%
y\in \mathfrak{M}_{0}$ has no import on the rejection probabilities of the
(non-bootstrap-based) tests obtained from $\tilde{T}_{uc}$, since $\mathfrak{%
M}_{0}$ is a $\lambda _{\mathbb{R}^{n}}$-null set (and since the measures $%
P_{\mu ,\sigma ^{2}\Sigma }$ are absolutely continuous). It will turn out
also not to affect our results for bootstrap-based tests obtained from $%
\tilde{T}_{uc}$. For later use we also mention that $\tilde{T}_{uc}$ is
continuous at every $y\in \mathbb{R}^{n}\backslash \mathfrak{M}_{0}$.

\begin{remark}
\label{rem:tildeGM0}The test statistics $\tilde{T}_{Het}$ as well as $\tilde{%
T}_{uc}$ are $G(\mathfrak{M}_{0})$-invariant as is easily seen (with the
respective exceptional sets $\mathsf{\tilde{B}}$ and $\mathfrak{M}_{0}$ also
being $G(\mathfrak{M}_{0})$-invariant), but typically they are \emph{not}
nonsphericity-corrected F-type tests in the sense of Section 5.4 in \cite%
{PP2016}.
\end{remark}

In the theorems given in the next two subsections we use the same bootstrap
schemes as before (i.e., (\ref{boot_1}) and (\ref{boot_2})); in particular,
recall that $\Xi $ is a (Borel) probability measure on $\mathbb{R}^{n}$, and
that $\mathcal{A}$ is an affine subspace of $\mathbb{R}^{n}$ satisfying $%
\mathfrak{M}_{0}\subseteq \mathcal{A}\subseteq \limfunc{span}(X)$.

\subsubsection{The first bootstrap scheme}

We start with results where the first bootstrap scheme,\ i.e., (\ref{boot_1}%
), is being used. The bootstrapped test statistic corresponding to $\tilde{T}%
_{Het}$ is then given by $\tilde{T}_{Het}^{\ast }$, where $\tilde{T}%
_{Het}^{\ast }:\mathbb{R}^{n}\times \mathbb{R}^{n}\rightarrow \mathbb{R}$ is
defined via 
\begin{equation*}
\tilde{T}_{Het}^{\ast }(y,\xi )=\tilde{T}_{Het}\left( y^{\ast }(y,\xi
)\right) .
\end{equation*}%
Furthermore, for every $y\in \mathbb{R}^{n}$ denote the distribution
function of the bootstrapped test statistic under $\Xi $ by $\tilde{F}%
_{Het,y}$, i.e., $\tilde{F}_{Het,y}(t)=\Xi \mathbb{(}\tilde{T}_{Het}^{\ast
}(y,\xi )\leq t)$ for $t\in \mathbb{R}$. For similar reasons as in Section %
\ref{unrestr_res}, we also consider the modification of $\tilde{T}_{Het}$
defined by $\tilde{T}_{Het}^{\blacktriangle }\left( y\right) =\tilde{T}%
_{Het}\left( y\right) $ if $y\notin \mathsf{\tilde{B}}$ and $\tilde{T}%
_{Het}^{\blacktriangle }(y)=\infty $ otherwise. Its bootstrapped version is
then given by $\tilde{T}_{Het}^{\blacktriangle ,\ast }(y,\xi )=\tilde{T}%
_{Het}^{\blacktriangle }\left( y^{\ast }(y,\xi )\right) $. For every $y\in 
\mathbb{R}^{n}$ we denote its distribution function under $\Xi $ by $\tilde{F%
}_{Het,y}^{\blacktriangle }$, i.e., $\tilde{F}_{Het,y}^{\blacktriangle
}(t)=\Xi \mathbb{(}\tilde{T}_{Het}^{\blacktriangle ,\ast }(y,\xi )\leq t)$
for $t\in \mathbb{R\cup \{\infty \}}$.

\begin{theorem}
\label{Theo_Het_restr_1}Suppose Assumption \ref{R_and_X_tilde} holds.

(a) For every $\alpha \in (0,1)$, let $\tilde{f}_{Het,1-\alpha }(y)$ denote
a $(1-\alpha )$-quantile of $\tilde{F}_{Het,y}$. Define%
\begin{equation}
\tilde{\vartheta}_{Het}=1-\max_{\substack{ i=1,\ldots ,n,  \\ \mu
_{0}+e_{i}(n)\notin \mathsf{\tilde{B}}}}\Xi \left( \left\{ \xi :\tilde{T}%
_{Het}^{\ast }(\mu _{0}+e_{i}(n),\xi )<\tilde{T}_{Het}(\mu
_{0}+e_{i}(n)),y^{\ast }(\mu _{0}+e_{i}(n),\xi )\notin \mathsf{\tilde{B}}%
\right\} \right)  \label{theta1_Het_tilde}
\end{equation}%
for some $\mu _{0}\in \mathfrak{M}_{0}$, with the convention that $\tilde{%
\vartheta}_{Het}=1$ if the index set in the maximum operator in (\ref%
{theta1_Het_tilde}) is empty. Then $\tilde{\vartheta}_{Het}$ does not depend
on the choice of $\mu _{0}\in \mathfrak{M}_{0}$. Furthermore, for every $%
\alpha \in (0,1)$ such that $\alpha >\tilde{\vartheta}_{Het}$ holds, we have%
\begin{equation}
\sup_{\Sigma \in \mathfrak{C}_{Het}}P_{\mu _{0},\sigma ^{2}\Sigma }\left( 
\tilde{T}_{Het}\geq \tilde{f}_{Het,1-\alpha }\right) \geq \sup_{\Sigma \in 
\mathfrak{C}_{Het}}P_{\mu _{0},\sigma ^{2}\Sigma }\left( \tilde{T}_{Het}>%
\tilde{f}_{Het,1-\alpha }\right) =1  \label{eqn:bootsize_het_tilde}
\end{equation}%
for every $\mu _{0}\in \mathfrak{M}_{0}$ and every $0<\sigma ^{2}<\infty $
(where the probabilities in (\ref{eqn:bootsize_het_tilde}) are to be
interpreted as inner probabilities).

(b) For every $\alpha \in (0,1)$, let $\tilde{f}_{Het,1-\alpha
}^{\blacktriangle }(y)$ denote a $(1-\alpha )$-quantile of $\tilde{F}%
_{Het,y}^{\blacktriangle }$. Then, with $\tilde{\vartheta}_{Het}$ defined in
Part (a), for every $\alpha \in (0,1)$ such that $\alpha >\tilde{\vartheta}%
_{Het}$ holds, we have%
\begin{equation}
\sup_{\Sigma \in \mathfrak{C}_{Het}}P_{\mu _{0},\sigma ^{2}\Sigma }\left( 
\tilde{T}_{Het}\geq \tilde{f}_{Het,1-\alpha }^{\blacktriangle }\right) \geq
\sup_{\Sigma \in \mathfrak{C}_{Het}}P_{\mu _{0},\sigma ^{2}\Sigma }\left( 
\tilde{T}_{Het}>\tilde{f}_{Het,1-\alpha }^{\blacktriangle }\right) =1
\label{eqn:bootsize_het_tilde_triangle}
\end{equation}%
for every $\mu _{0}\in \mathfrak{M}_{0}$ and every $0<\sigma ^{2}<\infty $
(where the probabilities in (\ref{eqn:bootsize_het_tilde_triangle}) are to
be interpreted as inner probabilities).
\end{theorem}

We next turn to the test statistic $\tilde{T}_{uc}$. Again using the first
bootstrap scheme, the bootstrapped test statistic is then given by $\tilde{T}%
_{uc}^{\ast }$, where $\tilde{T}_{uc}^{\ast }:\mathbb{R}^{n}\times \mathbb{R}%
^{n}\rightarrow \mathbb{R}$ is defined via 
\begin{equation*}
\tilde{T}_{uc}^{\ast }(y,\xi )=\tilde{T}_{uc}\left( y^{\ast }(y,\xi )\right)
.
\end{equation*}%
Furthermore, for every $y\in \mathbb{R}^{n}$ denote the distribution
function of the bootstrapped test statistic under $\Xi $ by $\tilde{F}%
_{uc,y} $, i.e., $\tilde{F}_{uc,y}(t)=\Xi \mathbb{(}\tilde{T}_{uc}^{\ast
}(y,\xi )\leq t)$ for $t\in \mathbb{R}$. We also consider the modification
of $\tilde{T}_{uc}$ defined by $\tilde{T}_{uc}^{\blacktriangle }\left(
y\right) =\tilde{T}_{uc}\left( y\right) $ if $y\notin \mathfrak{M}_{0}$ and $%
\tilde{T}_{uc}^{\blacktriangle }(y)=\infty $ otherwise. Its bootstrapped
version is then given by $\tilde{T}_{uc}^{\blacktriangle ,\ast }(y,\xi )=%
\tilde{T}_{uc}^{\blacktriangle }\left( y^{\ast }(y,\xi )\right) $. For every 
$y\in \mathbb{R}^{n}$ we denote its distribution function under $\Xi $ by $%
\tilde{F}_{uc,y}^{\blacktriangle }$, i.e., $\tilde{F}_{uc,y}^{\blacktriangle
}(t)=\Xi \mathbb{(}\tilde{T}_{uc}^{\blacktriangle ,\ast }(y,\xi )\leq t)$
for $t\in \mathbb{R\cup \{\infty \}}$.

\begin{theorem}
\label{Theo_uc_restr_1}(a) For every $\alpha \in (0,1)$, let $\tilde{f}%
_{uc,1-\alpha }(y)$ denote a $(1-\alpha )$-quantile of $\tilde{F}_{uc,y}$.
Define%
\begin{equation}
\tilde{\vartheta}_{uc}=1-\max_{\substack{ i=1,\ldots ,n,  \\ \mu
_{0}+e_{i}(n)\notin \mathfrak{M}_{0}}}\Xi \left( \left\{ \xi :\tilde{T}%
_{uc}^{\ast }(\mu _{0}+e_{i}(n),\xi )<\tilde{T}_{uc}(\mu
_{0}+e_{i}(n)),y^{\ast }(\mu _{0}+e_{i}(n),\xi )\notin \mathfrak{M}%
_{0}\right\} \right)  \label{theta1_uncorr_tilde}
\end{equation}%
for some $\mu _{0}\in \mathfrak{M}_{0}$.\footnote{%
Note that the index set in the maximum operator in (\ref{theta1_uncorr_tilde}%
) can not be empty since $k-q\leq k$ and we have assumed $k<n$.} Then $%
\tilde{\vartheta}_{uc}$ does not depend on the choice of $\mu _{0}\in 
\mathfrak{M}_{0}$. Furthermore, for every $\alpha \in (0,1)$ such that $%
\alpha >\tilde{\vartheta}_{uc}$ holds, we have%
\begin{equation}
\sup_{\Sigma \in \mathfrak{C}_{Het}}P_{\mu _{0},\sigma ^{2}\Sigma }\left( 
\tilde{T}_{uc}\geq \tilde{f}_{uc,1-\alpha }\right) \geq \sup_{\Sigma \in 
\mathfrak{C}_{Het}}P_{\mu _{0},\sigma ^{2}\Sigma }\left( \tilde{T}_{uc}>%
\tilde{f}_{uc,1-\alpha }\right) =1  \label{eqn:bootsize_uncorr_tilde}
\end{equation}%
for every $\mu _{0}\in \mathfrak{M}_{0}$ and every $0<\sigma ^{2}<\infty $
(where the probabilities in (\ref{eqn:bootsize_uncorr_tilde}) are to be
interpreted as inner probabilities).

(b) For every $\alpha \in (0,1)$, let $\tilde{f}_{uc,1-\alpha
}^{\blacktriangle }(y)$ denote a $(1-\alpha )$-quantile of $\tilde{F}%
_{uc,y}^{\blacktriangle }$. Then, with $\tilde{\vartheta}_{uc}$ defined in
Part (a), for every $\alpha \in (0,1)$ such that $\alpha >\tilde{\vartheta}%
_{uc}$ holds, we have%
\begin{equation}
\sup_{\Sigma \in \mathfrak{C}_{Het}}P_{\mu _{0},\sigma ^{2}\Sigma }\left( 
\tilde{T}_{uc}\geq \tilde{f}_{uc,1-\alpha }^{\blacktriangle }\right) \geq
\sup_{\Sigma \in \mathfrak{C}_{Het}}P_{\mu _{0},\sigma ^{2}\Sigma }\left( 
\tilde{T}_{uc}>\tilde{f}_{uc,1-\alpha }^{\blacktriangle }\right) =1
\label{eqn:bootsize_uncorr_tilde_triangle}
\end{equation}%
for every $\mu _{0}\in \mathfrak{M}_{0}$ and every $0<\sigma ^{2}<\infty $
(where the probabilities in (\ref{eqn:bootsize_uncorr_tilde_triangle}) are
to be interpreted as inner probabilities).
\end{theorem}

Mutatis mutandis, a discussion similar to the one given subsequently to
Theorem \ref{Theo_Het_unrestr} also applies to the preceding two theorems.

\subsubsection{The second bootstrap scheme}

For the test statistics $\tilde{T}_{Het}$ and $\tilde{T}_{uc}$ an analogon
to Lemma \ref{VGK} is not available. Hence, we need to provide separate
theorems for the case where the second bootstrap scheme, i.e., (\ref{boot_2}%
), is being used. This is done next. With this bootstrap scheme, the
bootstrapped test statistic corresponding to $\tilde{T}_{Het}$ is given by%
\begin{equation*}
\tilde{T}_{Het}^{\maltese }(y,\xi )=(R\hat{\beta}\left( y^{\maltese }(y,\xi
)\right) -R\hat{\beta}\left( y\right) )^{\prime }\tilde{\Omega}%
_{Het}^{-1}\left( y^{\maltese }(y,\xi )\right) (R\hat{\beta}\left(
y^{\maltese }(y,\xi )\right) -R\hat{\beta}\left( y\right) )
\end{equation*}%
if $y^{\maltese }(y,\xi )\notin \mathsf{\tilde{B}}$, and by $\tilde{T}%
_{Het}^{\maltese }(y,\xi )=0$ if $y^{\maltese }(y,\xi )\in \mathsf{\tilde{B}}
$. For every $y\in \mathbb{R}^{n}$ denote the distribution function of the
bootstrapped test statistic under $\Xi $ by $\tilde{H}_{Het,y}$, i.e., $%
\tilde{H}_{Het,y}(t)=\Xi \mathbb{(}\tilde{T}_{Het}^{\maltese }(y,\xi )\leq
t) $ for $t\in \mathbb{R}$. Furthermore, $\tilde{T}_{Het}^{\blacktriangle
,\maltese }(y,\xi )$ is defined exactly as is $\tilde{T}_{Het}^{\maltese
}(y,\xi )$, except that $\tilde{T}_{Het}^{\blacktriangle ,\maltese }(y,\xi
)=\infty $ if $y^{\maltese }(y,\xi )\in \mathsf{\tilde{B}}$. For every $y\in 
\mathbb{R}^{n}$ denote its distribution function under $\Xi $ by $\tilde{H}%
_{Het,y}^{\blacktriangle }$, i.e., $\tilde{H}_{Het,y}^{\blacktriangle
}(t)=\Xi \mathbb{(}\tilde{T}_{Het}^{\blacktriangle ,\maltese }(y,\xi )\leq
t) $ for $t\in \mathbb{R\cup \{\infty \}}$.

\begin{theorem}
\label{Theo_Het_restr_2}Suppose Assumption \ref{R_and_X_tilde} holds.

(a) For every $\alpha \in (0,1)$, let $\tilde{h}_{Het,1-\alpha }(y)$ denote
a $(1-\alpha )$-quantile of $\tilde{H}_{Het,y}$. Define%
\begin{equation}
\tilde{\theta}_{Het}=1-\max_{\substack{ i=1,\ldots ,n,  \\ \mu
_{0}+e_{i}(n)\notin \mathsf{\tilde{B}}}}\Xi \left( \left\{ \xi :\tilde{T}%
_{Het}^{\maltese }(\mu _{0}+e_{i}(n),\xi )<\tilde{T}_{Het}(\mu
_{0}+e_{i}(n)),y^{\maltese }(\mu _{0}+e_{i}(n),\xi )\notin \mathsf{\tilde{B}}%
\right\} \right)  \label{alttheta1_Het_tilde}
\end{equation}%
for some $\mu _{0}\in \mathfrak{M}_{0}$, with the convention that $\tilde{%
\theta}_{Het}=1$ if the index set in the maximum operator in (\ref%
{alttheta1_Het_tilde}) is empty. Then $\tilde{\theta}_{Het}$ does not depend
on the choice of $\mu _{0}\in \mathfrak{M}_{0}$. Furthermore, for every $%
\alpha \in (0,1)$ such that $\alpha >\tilde{\theta}_{Het}$ holds, we have%
\begin{equation}
\sup_{\Sigma \in \mathfrak{C}_{Het}}P_{\mu _{0},\sigma ^{2}\Sigma }\left( 
\tilde{T}_{Het}\geq \tilde{h}_{Het,1-\alpha }\right) \geq \sup_{\Sigma \in 
\mathfrak{C}_{Het}}P_{\mu _{0},\sigma ^{2}\Sigma }\left( \tilde{T}_{Het}>%
\tilde{h}_{Het,1-\alpha }\right) =1  \label{eqn:bootsize_het_tilde_malt}
\end{equation}%
for every $\mu _{0}\in \mathfrak{M}_{0}$ and every $0<\sigma ^{2}<\infty $
(where the probabilities in (\ref{eqn:bootsize_het_tilde_malt}) are to be
interpreted as inner probabilities).

(b) For every $\alpha \in (0,1)$, let $\tilde{h}_{Het,1-\alpha
}^{\blacktriangle }(y)$ denote a $(1-\alpha )$-quantile of $\tilde{H}%
_{Het,y}^{\blacktriangle }$. Then, with $\tilde{\theta}_{Het}$ defined in
Part (a), for every $\alpha \in (0,1)$ such that $\alpha >\tilde{\theta}%
_{Het}$ holds, we have%
\begin{equation}
\sup_{\Sigma \in \mathfrak{C}_{Het}}P_{\mu _{0},\sigma ^{2}\Sigma }\left( 
\tilde{T}_{Het}\geq \tilde{h}_{Het,1-\alpha }^{\blacktriangle }\right) \geq
\sup_{\Sigma \in \mathfrak{C}_{Het}}P_{\mu _{0},\sigma ^{2}\Sigma }\left( 
\tilde{T}_{Het}>\tilde{h}_{Het,1-\alpha }^{\blacktriangle }\right) =1
\label{eqn:bootsize_het_tilde_triangle_malt}
\end{equation}%
for every $\mu _{0}\in \mathfrak{M}_{0}$ and every $0<\sigma ^{2}<\infty $
(where the probabilities in (\ref{eqn:bootsize_het_tilde_triangle_malt}) are
to be interpreted as inner probabilities).
\end{theorem}

With the bootstrap scheme considered in this subsection, the bootstrapped
test statistic corresponding to $\tilde{T}_{uc}$ is given by%
\begin{equation*}
\tilde{T}_{uc}^{\maltese }(y,\xi )=(R\hat{\beta}\left( y^{\maltese }(y,\xi
)\right) -R\hat{\beta}\left( y\right) )^{\prime }\left( \tilde{\sigma}%
^{2}(y^{\maltese }(y,\xi ))R\left( X^{\prime }X\right) ^{-1}R^{\prime
}\right) ^{-1}(R\hat{\beta}\left( y^{\maltese }(y,\xi )\right) -R\hat{\beta}%
\left( y\right) )
\end{equation*}%
if $y^{\maltese }(y,\xi )\notin \mathfrak{M}_{0}$, and by $\tilde{T}%
_{uc}^{\maltese }(y,\xi )=0$ if $y^{\maltese }(y,\xi )\in \mathfrak{M}_{0}$.
For every $y\in \mathbb{R}^{n}$ denote the distribution function of the
bootstrapped test statistic under $\Xi $ by $\tilde{H}_{uc,y}$, i.e., $%
\tilde{H}_{uc,y}(t)=\Xi \mathbb{(}\tilde{T}_{uc}^{\maltese }(y,\xi )\leq t)$
for $t\in \mathbb{R}$. Furthermore, $\tilde{T}_{uc}^{\blacktriangle
,\maltese }(y,\xi )$ is defined exactly as is $\tilde{T}_{uc}^{\maltese
}(y,\xi )$, except that $\tilde{T}_{uc}^{\blacktriangle ,\maltese }(y,\xi
)=\infty $ if $y^{\maltese }(y,\xi )\in \mathfrak{M}_{0}$. For every $y\in 
\mathbb{R}^{n}$ denote its distribution function under $\Xi $ by $\tilde{H}%
_{uc,y}^{\blacktriangle }$, i.e., $\tilde{H}_{uc,y}^{\blacktriangle }(t)=\Xi 
\mathbb{(}\tilde{T}_{uc}^{\blacktriangle ,\maltese }(y,\xi )\leq t)$ for $%
t\in \mathbb{R\cup \{\infty \}}$.

\begin{theorem}
\label{Theo_uc_restr_2}(a) For every $\alpha \in (0,1)$, let $\tilde{h}%
_{uc,1-\alpha }(y)$ denote a $(1-\alpha )$-quantile of $\tilde{H}_{uc,y}$.
Define%
\begin{equation}
\tilde{\theta}_{uc}=1-\max_{\substack{ i=1,\ldots ,n,  \\ \mu
_{0}+e_{i}(n)\notin \mathfrak{M}_{0}}}\Xi \left( \left\{ \xi :\tilde{T}%
_{uc}^{\maltese }(\mu _{0}+e_{i}(n),\xi )<\tilde{T}_{uc}(\mu
_{0}+e_{i}(n)),y^{\maltese }(\mu _{0}+e_{i}(n),\xi )\notin \mathfrak{M}%
_{0}\right\} \right)  \label{alttheta1_uncorr_tilde}
\end{equation}%
for some $\mu _{0}\in \mathfrak{M}_{0}$.\footnote{%
Note that the index set in the maximum operator in (\ref%
{alttheta1_uncorr_tilde}) can not be empty since $k-q\leq k$ and we have
assumed $k<n$.} Then $\tilde{\theta}_{uc}$ does not depend on the choice of $%
\mu _{0}\in \mathfrak{M}_{0}$. Furthermore, for every $\alpha \in (0,1)$
such that $\alpha >\tilde{\theta}_{uc}$ holds, we have%
\begin{equation}
\sup_{\Sigma \in \mathfrak{C}_{Het}}P_{\mu _{0},\sigma ^{2}\Sigma }\left( 
\tilde{T}_{uc}\geq \tilde{h}_{uc,1-\alpha }\right) \geq \sup_{\Sigma \in 
\mathfrak{C}_{Het}}P_{\mu _{0},\sigma ^{2}\Sigma }\left( \tilde{T}_{uc}>%
\tilde{h}_{uc,1-\alpha }\right) =1  \label{eqn:bootsize_uncorr_tilde_malt}
\end{equation}%
for every $\mu _{0}\in \mathfrak{M}_{0}$ and every $0<\sigma ^{2}<\infty $
(where the probabilities in (\ref{eqn:bootsize_uncorr_tilde_malt}) are to be
interpreted as inner probabilities).

(b) For every $\alpha \in (0,1)$, let $\tilde{h}_{uc,1-\alpha
}^{\blacktriangle }(y)$ denote a $(1-\alpha )$-quantile of $\tilde{H}%
_{uc,y}^{\blacktriangle }$. Then, with $\tilde{\theta}_{uc}$ defined in Part
(a), for every $\alpha \in (0,1)$ such that $\alpha >\tilde{\theta}_{uc}$
holds, we have%
\begin{equation}
\sup_{\Sigma \in \mathfrak{C}_{Het}}P_{\mu _{0},\sigma ^{2}\Sigma }\left( 
\tilde{T}_{uc}\geq \tilde{h}_{uc,1-\alpha }^{\blacktriangle }\right) \geq
\sup_{\Sigma \in \mathfrak{C}_{Het}}P_{\mu _{0},\sigma ^{2}\Sigma }\left( 
\tilde{T}_{uc}>\tilde{h}_{uc,1-\alpha }^{\blacktriangle }\right) =1
\label{eqn:bootsize_uncorr_tilde_triangle_malt}
\end{equation}%
for every $\mu _{0}\in \mathfrak{M}_{0}$ and every $0<\sigma ^{2}<\infty $
(where the probabilities in (\ref{eqn:bootsize_uncorr_tilde_triangle_malt})
are to be interpreted as inner probabilities).
\end{theorem}

Mutatis mutandis, a discussion similar to the one given subsequently to
Theorem \ref{Theo_Het_unrestr} also applies to the preceding two theorems.

\subsection{Further remarks\label{rems}}

\begin{remark}
As already noted earlier, the various lower bounds for $\alpha $ given in
the theorems depend only on observable quantities, and can thus be computed
numerically. In particular, $\vartheta _{Het}$ in Theorem \ref%
{Theo_Het_unrestr} depends only on $X$, $R$, $r$, $\mathcal{A}$, $\Xi $, and
on the $d_{i}$'s appearing in the definition of the test statistic, while $%
\vartheta _{uc}$ in Theorem \ref{Theo_uc_unrestr} depends only on $X$, $R$, $%
r$, $\mathcal{A}$, and $\Xi $. Similarly, $\tilde{\vartheta}_{Het}$ in
Theorem \ref{Theo_Het_restr_1} and $\tilde{\theta}_{Het}$ in Theorem \ref%
{Theo_Het_restr_2} depend only on $X$, $R$, $r$, $\mathcal{A}$, $\Xi $, and
on the $\tilde{d}_{i}$'s appearing in the definition of the test statistic.
Finally, $\tilde{\vartheta}_{uc}$ in Theorem \ref{Theo_uc_restr_1} and $%
\tilde{\theta}_{uc}$ in Theorem \ref{Theo_uc_restr_2} depend only on $X$, $R$%
, $r$, $\mathcal{A}$, and $\Xi $.
\end{remark}

\begin{remark}
\label{rem:null}(i) Relations (\ref{eqn:bootsize_het}) and (\ref%
{eqn:bootsize_het_triangle}) continue to hold a fortiori if $T_{Het}$ is
replaced by $T_{Het}^{\blacktriangle }$, as $T_{Het}^{\blacktriangle }$ is
never smaller than $T_{Het}$ (in fact, $T_{Het}^{\blacktriangle }$ and $%
T_{Het}$ coincide except on a Lebesgue null set under Assumption \ref%
{R_and_X}).

(ii) Relations (\ref{eqn:bootsize_uncorr}) and (\ref%
{eqn:bootsize_uncorr_triangle}) continue to hold a fortiori if $T_{uc}$ is
replaced by $T_{uc}^{\blacktriangle }$, as $T_{uc}^{\blacktriangle }$ is
never smaller than $T_{uc}$ (in fact, both coincide except on $\limfunc{span}%
(X)$, a Lebesgue null set).

(iii) Relations (\ref{eqn:bootsize_het_tilde}), (\ref%
{eqn:bootsize_het_tilde_triangle}), (\ref{eqn:bootsize_het_tilde_malt}), and
(\ref{eqn:bootsize_het_tilde_triangle_malt}) continue to hold a fortiori if $%
\tilde{T}_{Het}$ is replaced by $\tilde{T}_{Het}^{\blacktriangle }$, as $%
\tilde{T}_{Het}^{\blacktriangle }$ is never smaller than $\tilde{T}_{Het}$
(in fact, both coincide except on a Lebesgue null set under Assumption \ref%
{R_and_X_tilde}).

(iv) Relations (\ref{eqn:bootsize_uncorr_tilde}), (\ref%
{eqn:bootsize_uncorr_tilde_triangle}), (\ref{eqn:bootsize_uncorr_tilde_malt}%
), and (\ref{eqn:bootsize_uncorr_tilde_triangle_malt}) continue to hold a
fortiori if $\tilde{T}_{uc}$ is replaced by $\tilde{T}_{uc}^{\blacktriangle
} $, as $\tilde{T}_{uc}^{\blacktriangle }$ is never smaller than $\tilde{T}%
_{uc}$ (in fact, both coincide except on $\mathfrak{M}_{0}$, a Lebesgue null
set).
\end{remark}

\begin{remark}
In case $q=1$, inspection of the proof of Theorem \ref{Theo_Het_unrestr}
shows that the bound $\vartheta _{Het}$ can be somewhat improved by allowing
in the definition of $\vartheta _{2,Het}$ the index $i$ to range over all
indices such that $e_{i}(n)\in \mathsf{B}$ and $R\hat{\beta}(e_{i}(n))\neq 0$%
. [This is so, since in case $q=1$ singularity of $\hat{\Omega}_{Het}\left(
y\right) $ is equivalent to $\hat{\Omega}_{Het}\left( y\right) =0$.]
\end{remark}

\begin{remark}
\label{rem:equiv}The test statistics $T_{Het}$ using HC0 and HC1 weights,
respectively, differ only by a multiplicative constant, and hence result in
the same bootstrap-based test. For design matrices $X$ with $h_{ii}$ not
depending on $i$, the same conclusion applies for all weights HC0-HC4. A
similar remark applies to $\tilde{T}_{Het}$ (with $\tilde{h}_{ii}$ taking
the r\^{o}le of $h_{ii}$).
\end{remark}

\begin{remark}
\label{conf_set}The size one results for bootstrap-based tests given in the
preceding theorems are easily seen to imply infimal coverage zero results
for the corresponding confidence sets for $R\beta $ obtained by
\textquotedblleft inverting\textquotedblright\ the tests. The computation of
such confidence sets is straightforward and leads to ellipsoids in the case
where the bootstrap-based test is obtained from $T_{Het}$ or $T_{uc}$ and
the bootstrap scheme (\ref{boot_2}) with $\mathcal{A}=\limfunc{span}(X)$ is
used. This is so, since the covariance matrix estimator employed in $T_{Het}$
(or $T_{uc}$, respectively) does not depend on $r$, and since the quantile
of the bootstrap distribution is easily seen also not to depend on $r$ in
this case. By Lemma \ref{VGK} the same is true for the bootstrap-based tests
obtained from $T_{Het}$ or $T_{uc}$ and the bootstrap scheme (\ref{boot_1})
with $\mathcal{A}=\limfunc{span}(X)$. In all other combinations of test
statistics and bootstrap schemes the "inversion" is typically more
complicated and becomes numerically burdensome, as then the covariance
matrix estimator employed in the test statistic and/or the quantile of the
bootstrap distribution will typically depend on $r$.
\end{remark}

\section{Some special cases\label{special}}

Here we consider the special case where the null hypothesis is simple (i.e., 
$q=k$). If restricted residuals are used in the bootstrap scheme (\ref%
{boot_1}) (i.e., if $\mathcal{A}=\mathfrak{M}_{0}$ holds), the following
result shows that in these cases our theorems become vacuous, and thus do
not allow us to draw any conclusion about the sizes of the corresponding
bootstrap-based tests. [Of course, this by itself does not preclude the
possibility that in these cases the size may be equal to one or may
substantially exceed $\alpha $.] The observations made in the theorem below
are in line with a result in \cite{DavidsFlach2008} implying that -- in the
case corresponding to Part (c) of the subsequent theorem -- the
bootstrap-based test using the bootstrap scheme (\ref{boot_1}) with $%
\mathcal{A}=\mathfrak{M}_{0}$ indeed has size equal to the nominal
significance level $\alpha $, \emph{provided} a particular choice of $\Xi $
and particular values of $\alpha $ are used. [In fact, this result, which is
Theorem 1 in \cite{DavidsFlach2008}, is not entirely correct in the form
given, but needs some amendments and corrections, which we shall not provide
here.\footnote{\label{FN_counterex}A simple counterexample to Theorem 1 in 
\cite{DavidsFlach2008} is provided by a regression model which has a
standard basis vector as its only regressor. It is then easy to see that the
test statistic is (almost surely) constant and coincides with the
bootstrapped test statistic. Consequently, the bootstrap-based test becomes
trivial. Its null rejection probabilities are equal to $0$ if the p-value is
defined as in \cite{DavidsFlach2008}. [They are equal to $1$ if an
alternative definition of the p-value is used.]}]

\begin{theorem}
\label{special_cases}Suppose $q=k$ and $\Xi (\left\{ \xi \in \mathbb{R}%
^{n}:\xi _{i}\neq 0\right\} )=1$ for every $i=1,\ldots ,n$. Then:

(a) $\vartheta _{Het}=1$ holds in Theorem \ref{Theo_Het_unrestr}, if $%
\mathcal{A}=\mathfrak{M}_{0}$ is used in the bootstrap scheme.\footnote{%
\label{FN1}Theorem \ref{Theo_Het_unrestr} maintains Assumption \ref{R_and_X}%
. If this assumption is violated, then $T_{Het}$ is identically equal to
zero, leading to a useless test. If one would formally apply bootstrap
scheme (\ref{boot_1}), the bootstrapped test statistic $T_{Het}^{\ast }$
would then also be identically zero, leading to a bootstrap critical value
of zero. The rejection probability is then always equal to zero or always
equal to one, depending on whether one uses a strict or weak inequality in
the definition of the rejection region.}

(b) $\vartheta _{uc}=1$ holds in Theorem \ref{Theo_uc_unrestr}, if $\mathcal{%
A}=\mathfrak{M}_{0}$ is used in the bootstrap scheme.

(c) $\tilde{\vartheta}_{Het}=1$ holds in Theorem \ref{Theo_Het_restr_1}, if $%
\mathcal{A}=\mathfrak{M}_{0}$ is used in the bootstrap scheme.\footnote{%
Theorem \ref{Theo_Het_restr_1} maintains Assumption \ref{R_and_X_tilde}. If
this assumption is violated, then a similar comment as in Footnote \ref{FN1}
applies.}

(d) $\tilde{\vartheta}_{uc}=1$ holds in Theorem \ref{Theo_uc_restr_1}, if $%
\mathcal{A}=\mathfrak{M}_{0}$ is used in the bootstrap scheme.
\end{theorem}

\begin{remark}
(i) Part (a) (Part (b), respectively) of the preceding theorem applies to
the bootstrap-based test derived from $T_{Het}$ ($T_{uc}$, respectively)
when the bootstrap scheme (\ref{boot_1}) with $\mathcal{A}=\mathfrak{M}_{0}$
is employed. In view of Lemma \ref{VGK} and Remark \ref{rem:vgk}, these
results also apply if the bootstrap scheme (\ref{boot_2}), again with $%
\mathcal{A}=\mathfrak{M}_{0}$, is used. However, as can be seen from simple
examples, this is not so in the context of Parts (c) and (d) of the
preceding theorem (i.e., $\tilde{\theta}_{Het}<1$ and $\tilde{\theta}_{uc}<1$
can occur in Theorems \ref{Theo_Het_restr_2} and \ref{Theo_uc_restr_2}.
respectively, even if $q=k$, $\mathcal{A}=\mathfrak{M}_{0}$, and $\Xi $ is
as in Theorem \ref{special_cases}).

(ii) If bootstrap schemes (\ref{boot_1}) or (\ref{boot_2}) are used for the
bootstrap-based tests derived from any of $T_{Het}$, $T_{uc}$, $\tilde{T}%
_{Het}$, and $\tilde{T}_{uc}$, but now with $\mathfrak{M}_{0}\subsetneqq 
\mathcal{A}$, simple examples show that $\vartheta _{Het}<1$, $\vartheta
_{uc}<1$, $\tilde{\vartheta}_{Het}<1$, $\tilde{\vartheta}_{uc}<1$, $\tilde{%
\theta}_{Het}<1$, and $\tilde{\theta}_{uc}<1$ can occur.
\end{remark}

As a consequence of the preceding theorem, the function in the R-package 
\textbf{wbsd} for computing $\vartheta _{Het}$, etc. first checks if the
condition of the theorem are satisfied, and if so, outputs $1$ for $%
\vartheta _{Het}$, etc.

\section{Extensions and generalizations\label{gen}}

\subsection{Other covariance models}

The results given so far refer to the size of bootstrap-based tests when the
covariance model $\mathfrak{C}_{Het}$ is maintained. Inspection of the
proofs of the theorems in Sections \ref{Theory} and \ref{special} as well as
of Theorem \ref{thm:main} in Appendix \ref{App A} shows that they also hold
with $\mathfrak{C}_{Het}$ replaced by any covariance model $\mathfrak{C}%
\subseteq \mathfrak{C}_{Het}$ other than $\mathfrak{C}_{Het}$, provided the
closure of $\mathfrak{C}$ contains, for $i=1,\ldots ,n$, the matrices $%
e_{i}(n)e_{i}(n)^{\prime }$. More generally, if the closure of $\mathfrak{C}$
contains the matrices $e_{i}(n)e_{i}(n)^{\prime }$ only for $i\in I\subseteq
\{1,\ldots ,n\}$, then Theorem \ref{thm:main} continues to hold with $%
\mathfrak{C}_{Het}$ replaced by $\mathfrak{C}$, provided the range of the
maximum operator in (\ref{eqn:newbound}) is intersected with $I$, and the
other theorems mentioned before continue to hold with $\mathfrak{C}_{Het}$
replaced by $\mathfrak{C}$, provided the ranges of the maximum operators
appearing in the definitions of the various quantities $\vartheta _{1,Het}$, 
$\vartheta _{2,Het}$, $\vartheta _{1,uc}$, $\vartheta _{2,uc}$, $\tilde{%
\vartheta}_{Het}$, $\tilde{\vartheta}_{uc}$, $\tilde{\theta}_{Het}$, and $%
\tilde{\theta}_{uc}$ are intersected with $I$.\footnote{%
It is here understood that a maximum is interpreted as zero if it extends
over an empty range.}

\subsection{Nongaussian errors}

Consider now the regression model as in Section \ref{frame}, except for the
Gaussianity assumption.

(i) If we assume a (possibly semiparametric) model for the distribution of
the errors such that the implied model for the distributions of $\mathbf{Y}$
contains all the Gaussian distributions shown in (\ref{lm2}), then the
results of the paper continue to hold a fortiori (with the same lower bounds
for $\alpha $), since the size of any test computed w.r.t. such a larger
model for the distributions of $\mathbf{Y}$ is certainly not smaller than
the size of the same test when computed w.r.t. to the Gaussian model (\ref%
{lm2}).

(ii) Suppose next we assume that the standardized errors $\sigma ^{-1}\Sigma
^{-1/2}\mathbf{U}$ follow a (fixed) distribution $G$ that does not depend on 
$(\mu ,\sigma ,\Sigma )$, and let $Q_{\mu ,\sigma ^{2}\Sigma ,G}$ denote the
implied distribution of $\mathbf{Y}$. If $G$ is absolutely continuous w.r.t. 
$\lambda _{\mathbb{R}^{n}}$, then the results of the paper continue to hold
with $P_{\mu ,\sigma ^{2}\Sigma }$ replaced by $Q_{\mu ,\sigma ^{2}\Sigma
,G} $ (and with the same lower bounds for $\alpha $). This is easily seen
from an inspection of the proofs.\footnote{\label{FNabs}The assumption of
absolute continuity of $G$ can, in fact, be relaxed to the assumption that
none of its one-dimensional marginals has positive mass at zero in case of
Theorems \ref{thm:main}, \ref{Theo_Het_restr_1} - \ref{Theo_uc_restr_2}, and
of the weaker versions of Theorems \ref{Theo_Het_unrestr} and \ref%
{Theo_uc_unrestr} discussed in Remark \ref{rem:weaker} in Appendix \ref{App
B}. The proofs of Theorems \ref{Theo_Het_unrestr} and \ref{Theo_uc_unrestr}
also extend to the situation discussed here under any additional condition
that guarantees that the exceptional sets $\mathsf{B}$ and $\limfunc{span}%
(X) $, respectively, have probability zero under any $Q_{\mu _{0},\sigma
^{2}\Sigma ,G}$.}

(iii) Suppose we have the same framework as in (ii), except that now $G$
varies in a set $\mathfrak{G}$ (independently of $(\mu ,\sigma ,\Sigma )$),
i.e., we have a semiparametric model. If at least one member $G\in \mathfrak{%
G}$ is absolutely continuous, then the results of the paper continue to hold
a fortiori (with the same lower bounds for $\alpha $) for reasons similar to
the ones given in (i).\footnote{%
Again, the absolute continuity assumption can be weakened, cf. Footnote \ref%
{FNabs}.}

(iv) Also note that the lower bounds for $\alpha $ in all the result do not
involve the distribution of $\mathbf{Y}$ (and thus of $\mathbf{U})$, and, in
particular, do not involve the Gaussianity assumption. Hence, in this sense
the lower bounds are \textquotedblleft distribution free\textquotedblright .

\subsection{Stochastic regressors}

The assumption of nonstochastic regressors can be easily relaxed as follows:
Suppose $X$ is random and $\mathbf{U}$ is conditionally on $X$ distributed
as $N(0,\sigma ^{2}\Sigma )$, with $\sigma ^{2}=\sigma ^{2}(X)>0$ and $%
\Sigma =\Sigma (X)\in \mathfrak{C}_{Het}$, where $\sigma ^{2}(\cdot )$ and $%
\Sigma (\cdot )$ vary in given classes of functions. Suppose further that $%
\sigma ^{2}(X)$ and $\Sigma (X)$ vary independently through all of $%
(0,\infty )$ and $\mathfrak{C}_{Het}$, respectively, for (almost) every
realization of $X$, when the functions $\sigma ^{2}(\cdot )$ and $\Sigma
(\cdot )$ vary in the before mentioned function classes.\footnote{%
This is certainly the case if no restrictions on the functions $\sigma
^{2}(\cdot )$ and $\Sigma (\cdot )$ are imposed beyond $\sigma ^{2}(\cdot )$
and $\Sigma (\cdot )$, respectively, taking values in $(0,\infty )$ and in
the set of diagonal matrices with positive diagnal elements. Another
instance where this is seen to be satisfied is the case where $\sigma
^{2}(X)\Sigma (X)=\limfunc{diag}(\varpi (x_{1\cdot }),\ldots ,\varpi
(x_{n\cdot }))$ with no further restrictions on the function $\varpi $
(besides positivity) and with at least one regressor being absolutely
continuous.} Then the results of the paper obviously apply after one
conditions on $X$ provided (almost) all realizations of $X$ satisfy the
assumptions of our theorems, which will typically be the case (for brevity
we do not provide a formal statement here). And again similar
generalizations to non-Gaussianity as discussed in the preceding subsection
are possible here.

\section{Numerical results\label{sec_Num}}

There is a considerable body of simulation studies investigating finite
sample properties of bootstrap-based heteroskedasticity robust tests, see
the references mentioned in the Introduction. While these studies provide
helpful information, there is -- as always with simulation studies -- an
issue to what extent conclusions of such a study generalize. This is
particularly so with \emph{positive }findings (such as, e.g., that a
particular bootstrap-based test has null rejection probabilities close to
the nominal significance level) as it is less than clear that such a finding
allows for generalization beyond the design matrices $X$, the restrictions
(given by $R$, $r$), and the forms of heteroskedasticity considered in the
simulation study. It is less of an issue with \emph{negative} results (such
as, e.g., that a particular bootstrap-based test has null rejection
probabilities much larger than the nominal significance level), since they
can be viewed as counterexamples disproving good behavior of the
bootstrap-based test in general.

For these reasons we set out to study the \emph{worst-case size performance}
of a variety of bootstrap-based tests.\footnote{%
The concept of the size of a test is by itself already a worst-case concept,
as it is the supremum of the rejection probabilites over the null hypothesis
(where $X$ and the restrictions are being held fixed), cf. (\ref{size}). The
term "worst-case" in "worst-case size performance" here refers to varying $X$
and the restrictions to be tested.} That is, for any given bootstrap-based
test in a large class, we try to "break" the test by searching for a design
matrix $X$, a restriction (given by $R$, $r$) to be tested, and a form of
heteroskedasticity, such that the null rejection probability of the test is
substantially larger than the nominal significance level $\alpha $. Note
that this is equivalent to finding $X$, $R$, and $r$ such that the size of
the bootstrap-based test computed over the heteroskedasticity model $%
\mathfrak{C}_{Het}$ is substantially larger than $\alpha $. A
bootstrap-based test that is "broken" in our study, should probably not be
used by practitioners (at least not without first assessing its properties
in the particular testing problem put before the practitioner, e.g., by
attempting to determining the size of the test in that problem by Monte
Carlo methods). A bootstrap-based test that "survives" the "stress test"
imposed by our study may perhaps be considered to be a better choice, but
note that our numerical results do \emph{not} provide any guarantee for good
performance in the practitioner's testing problem either (and thus again an
assessment of its properties in the practitioner's testing problem may be
called for).

Our theoretical results obtained in the previous sections play an important r%
\^{o}le in our study of the size performance of bootstrap-based tests, as
these results allow us to deduce abysmal size behavior (i.e., size equal to $%
1$) of a bootstrap-based test (for a given design matrix $X$ and restriction 
$(R$, $r)$ to be tested) by comparing the (numerically evaluated) $\vartheta 
$ with the nominal significance level $\alpha $; in this section the symbol $%
\vartheta $ serves as a generic abbreviation for $\vartheta _{Het}$($=\theta
_{Het}$) (cf.~Theorem \ref{Theo_Het_unrestr}), $\vartheta _{uc}$($=\theta
_{uc}$) (cf.~Theorem \ref{Theo_uc_unrestr}), $\tilde{\vartheta}_{Het}$
(cf.~Theorem \ref{Theo_Het_restr_1}), $\tilde{\theta}_{Het}$ (cf.~Theorem %
\ref{Theo_Het_restr_2}), $\tilde{\vartheta}_{uc}$ (cf.~Theorem \ref%
{Theo_uc_restr_1}), or $\tilde{\theta}_{uc}$ (cf.~Theorem \ref%
{Theo_uc_restr_2}), depending on which of the theorems listed in parentheses
(possibly after an appeal to Lemma \ref{VGK} and Remark \ref{rem:vgk})
applies to the bootstrap-based test under consideration. Recall that these
theorems show that if $\alpha >\vartheta $ holds, then the corresponding
bootstrap-based test has size $1$ (over the heteroskedasticity model $%
\mathfrak{C}_{Het}$), and thus certainly breaks down (for the given testing
problem, i.e., for the given $X$, $R$, $r$). Hence, we shall search for
worst-case $X$, $R$, and $r$ that lead to small values of $\vartheta $.%
\footnote{%
Evaluating $\vartheta $ numerically is a nontrivial task as discussed in
Section \ref{sec:numch} in Appendix \ref{App compu}. In order to be on the
safe side and to bias our results in favor of the tests (recall that we are
after negative results), the $\vartheta $ we shall report will actually be a
numerical obtained upper bound for the true $\vartheta $.}

More precisely, we shall consider three \emph{settings}, where \emph{setting}
refers to sample size $n=10,20$, and $30$, and various \emph{scenarios},
where \emph{scenario }refers to a combination of $k$ (number of regressors)
and $q$ (number of restrictions to be tested and where $R=(0:I_{q})$, $r=0$%
). In every setting, we roughly do the following: we compute for every
bootstrap-based test included in our study the value of $\vartheta $ for a
variety of design matrices in a range of scenarios. We then determine the
minimal value of $\vartheta $ over all design matrices considered. Then, we
compare this minimum with two commonly used levels of significance ($\alpha
=0.05$ and $\alpha =0.1$). In addition to studying the behavior of this
minimal value of $\vartheta $, we shall complement this by numerical size
lower bound computations.

All computations were carried out in R (\cite{R}) version 3.6.3 using
version 1.0.0 of the R-package \textbf{wbsd} (\textquotedblleft wild
bootstrap size diagnostics\textquotedblright ) by \cite{wbsd} generated with
Rtools35. The package \textbf{wbsd} provides computationally efficient
routines for determining the quantities $\vartheta _{Het}$($=\theta _{Het}$%
), $\vartheta _{uc}$($=\theta _{uc}$), $\tilde{\vartheta}_{Het}$, $\tilde{%
\theta}_{Het}$, $\tilde{\vartheta}_{uc}$, and $\tilde{\theta}_{uc}$, and for
obtaining bootstrap p-values in order to obtain the numerical results
reported here. The tools for computing~$\vartheta $ provided in the
R-package \textbf{wbsd} can be used by practitioners as a diagnostic device
to check whether a bootstrap-based test is provably unreliable (in that $%
\alpha >\vartheta $, which implies size equal to $1$) in a given testing
problem. The package is available on CRAN.

In the following subsections we describe the bootstrap-based tests studied,
we explain the computations carried out for each test, and discuss the
results obtained. Some of the details are deferred to Appendix \ref{App
compu}.

\subsection{Description of the bootstrap-based tests studied\label{sec:descr}%
}

The number of bootstrap-based tests we cover in our study is vast: In total
we consider the $960$ possible combinations of the $12$ test statistics
discussed in Section \ref{test_statistic} and at the beginning of Section %
\ref{restrict_res} around Equations (\ref{T_het}), (\ref{T_uncorr}), (\ref%
{T_Het_tilde}), and (\ref{T_uncorr_tilde}) with the bootstrap schemes
discussed further below ($80$ in total), which are popular special cases of
the two general bootstrap schemes discussed in Section \ref{Theory}.

To be precise, the $12$ \emph{test statistics studied} are:\footnote{\label%
{FNconv}We use here the conventions for $d_{i}$ and $\tilde{d}_{i}$ given
below (\ref{T_het}) and (\ref{T_Het_tilde}), respectively.}

\begin{enumerate}
\item \textbf{Test statistics based on unrestricted residuals}: $T_{uc}$;
and $T_{Het}$ with $d_{i}=1$ (HC0); with $d_{i}=n/(n-k)$ (HC1); with $%
d_{i}=(1-h_{ii})^{-1}$ (HC2); with $d_{i}=(1-h_{ii})^{-2}$ (HC3); and with $%
d_{i}=(1-h_{ii})^{\delta _{i}}$ for $\delta _{i}=\min (nh_{ii}/k,4)$ (HC4).

\item \textbf{Test statistics based on restricted residuals}: $\tilde{T}%
_{uc} $; $\tilde{T}_{Het}$ with $\tilde{d}_{i}=1$ (HC0R); with $\tilde{d}%
_{i}=n/(n-(k-q))$ (HC1R); with $\tilde{d}_{i}=(1-\tilde{h}_{ii})^{-1}$
(HC2R); with $\tilde{d}_{i}=(1-\tilde{h}_{ii})^{-2}$ (HC3R); and with $%
\tilde{d}_{i}=(1-\tilde{h}_{ii})^{\tilde{\delta}_{i}}$ for $\tilde{\delta}%
_{i}=\min (n\tilde{h}_{ii}/(k-q),4)$ (HC4R).
\end{enumerate}

The \emph{bootstrap schemes we study} are $y^{\ast }$ as defined in (\ref%
{boot_1}), and $y^{\maltese }$ as defined in (\ref{boot_2}). Both bootstrap
schemes are applied with $\mathcal{A}=\mathfrak{M}_{0}$ as well as with $%
\mathcal{A}=\limfunc{span}(X)$. In addition to choosing $\mathcal{A}$, both
bootstrap schemes require a concrete choice of $\Xi $. All distributions $%
\Xi $ we consider are constructed in the following way: first an auxiliary
distribution $\Xi ^{\bullet }$ on $\{-1,1\}^{n}$ has to be chosen. The way
we choose this auxiliary distribution depends on the magnitude of $n$. We
consider three cases: Setting A ($n=10$), Setting B ($n=20$), and Setting C (%
$n=30$).

\begin{itemize}
\item In Setting A, we consider (i) $\Xi ^{\bullet }$ equal to the $n$-fold
Rademacher distribution\ (i.e., the $n$-fold product of the uniform
distribution on $\{-1,1\}$), and (ii) $\Xi ^{\bullet }$ equal to the $n$%
-fold Mammen distribution\ (i.e., the $n$-fold product of the distribution
on $\{-(\sqrt{5}-1)/2,(\sqrt{5}+1)/2\}$ that assigns mass $(\sqrt{5}+1)/(2%
\sqrt{5})$ to $-(\sqrt{5}-1)/2$).

\item In Settings B and C, we consider $\Xi ^{\bullet }$ equal to an \emph{%
empirical} distribution of a sample of size $10n-1$ from the $n$-fold
Rademacher distribution and from the $n$-fold Mammen distribution,
respectively.\footnote{%
The reason for treating Settings B and C differently from Setting A is that
for values of $n$ such as $20$ or $30$ an enumeration of all support points
of the $n$-fold Rademacher or $n$-fold Mammen distribution is numerically
too costly.}
\end{itemize}

Given an auxiliary distribution $\Xi ^{\bullet }$, the distribution $\Xi $
actually used in the bootstrap scheme depends on a vector of weights $w$,
itself typically depending on $X$ or on $X$ and $R$. Given a weights vector $%
w$, $\Xi $ is then obtained as the distribution of $\limfunc{diag}(w)\xi
^{^{\bullet }}$ where $\xi ^{^{\bullet }}$ follows the distribution $\Xi
^{\bullet }$. We consider the following choices for the vector of weights $w$%
:\footnote{%
We use here the same conventions as mentioned in Footnote \ref{FNconv}.}

\begin{enumerate}
\item \textbf{Unrestricted HC0-HC4 weights} $w=(w_{1},\ldots ,w_{n})$\textbf{%
:} $w_{i}=1$ (HC0), $w_{i}=[n/(n-k)]^{1/2}$ (HC1), $w_{i}=(1-h_{ii})^{-1/2}$
(HC2), $w_{i}=(1-h_{ii})^{-1}$ (HC3), and $w_{i}=(1-h_{ii})^{\delta _{i}/2}$
for $\delta _{i}=\min (nh_{ii}/k,4)$ (HC4).

\item \textbf{Null-restricted HC0R-HC4R weights} $w=(\tilde{w}_{1},\ldots ,%
\tilde{w}_{n})$\textbf{:} $\tilde{w}_{i}=1$ (HC0R), $\tilde{w}%
_{i}=[n/(n-(k-q))]^{1/2}$ (HC1R), $\tilde{w}_{i}=(1-\tilde{h}_{ii})^{-1/2}$
(HC2R), $\tilde{w}_{i}=(1-\tilde{h}_{ii})^{-1}$ (HC3R), and $\tilde{w}%
_{i}=(1-\tilde{h}_{ii})^{\tilde{\delta}_{i}/2}$ for $\tilde{\delta}_{i}=\min
(n\tilde{h}_{ii}/(k-q),4)$ (HC4R).
\end{enumerate}

In total this gives $960$ possible combinations of test statistics and
bootstrap schemes. We emphasize that some of these combinations result in
the same bootstrap-based test: (i) For reasons discussed in Lemma \ref{VGK},
(ii) when changing HC0 weights to HC0R weights in the bootstrap scheme, and
(iii) when changing HC0 (HC0R) weights to HC1 (HC1R) weights in the
definition of $T_{Het}$ ($\tilde{T}_{Het}$), cf.~Remark \ref{rem:equiv}.
Concerning the run-time of the simulations, one could certainly argue that,
for the computations, one should keep only one of the combinations that lead
to the same bootstrap-based test. However, we have chosen not to, because we
can then exploit the ensuing additional computations as a double-check for
the methods that \textquotedblleft survive\textquotedblright\ the worst-case
analysis (as the design matrices are generated separately for each of the $%
960$ combinations).

Given a test statistic, a bootstrap scheme, and a level of significance $%
\alpha $, the corresponding bootstrap-based test is throughout taken as the
test that, observing $y$, rejects the null hypothesis, if the bootstrap
p-value computed for $y$ is strictly smaller than $\alpha $. Here, bootstrap
p-value refers to the mass assigned by $\Xi $ to the points $\xi $ that give
rise to elements in the bootstrap sample at which the test statistic is
greater than or equal to the test statistic evaluated at $y$. To be precise,
if, e.g., $T_{Het}$ is used as a test statistic, we define the bootstrap
p-value as $\Xi (\xi :T_{Het}^{\blacktriangle ,\ast }(y,\xi )\geq
T_{Het}(y)) $ if a bootstrap scheme of the form $y^{\ast }$ is used, and as $%
\Xi (\xi :T_{Het}^{\blacktriangle ,\maltese }(y,\xi )\geq T_{Het}(y))$ if a
bootstrap scheme of the form $y^{\maltese }$ is used; for the other test
statistics $T_{uc}$, $\tilde{T}_{Het}$, and $\tilde{T}_{uc}$ we proceed
similarly. Recall that $T_{Het}^{\blacktriangle }$ coincides with $T_{Het}$,
except on the exceptional set $\mathsf{B}$, on which $T_{Het}^{%
\blacktriangle }$ is set equal to $\infty $ (and a similar statement applies
for the other test statistics). The reason for using $T_{Het}^{%
\blacktriangle ,\ast }$ ($T_{Het}^{\blacktriangle ,\maltese }$,
respectively) rather than $T_{Het}^{\ast }$ ($T_{Het}^{\maltese }$,
respectively) in the definition of the p-value (and similarly for the other
test statistics)\ is that this potentially gives a smaller rejection region,
thus biasing the result in favor of the test (recall we are after negative
results!); cf. the discussion relating to Part (b) of Theorem \ref%
{Theo_Het_unrestr} given subsequent to this theorem. For the same reason we
use $\geq $, and not $>$, in the definition of the p-value. It is easy to
see that -- in case a bootstrap scheme of the form $y^{\ast }$ is used --
the bootstrap-based test just defined via p-values can be rewritten as the
test that rejects if $T_{Het}(y)>f_{Het,1-\alpha }^{\blacktriangle
,upper}(y) $, where $f_{Het,1-\alpha }^{\blacktriangle ,upper}(y)$ is the 
\emph{upper} (i.e., largest) $(1-\alpha )$-quantile of $F_{Het,y}^{%
\blacktriangle }$; and a similar statement applies if a bootstrap scheme of
the form $y^{\maltese }$ (or one of the other test statistics) is being
used. [This also shows that the above defined rejection region is the
smallest among all the rejection regions that can appear in the formulations
of the theorems in Section \ref{Theory}.]

\subsection{Computations carried out in each setting\label{sec:comp}}

In each setting ($n=10,20,30$) and for each of the $960$ combinations of
test statistics and bootstrap schemes described above we perform a two-step
procedure. A detailed description of the computations carried out can be
found in Sections \ref{sec:stp1} and \ref{sec:stp2} in Appendix \ref{App
compu}. Here we only provide a brief summary of the two-step procedure to
the extent needed for an understanding of the results presented in Section %
\ref{sec:results}.

\begin{enumerate}
\item The main goal of Step 1 is to find a scenario and a corresponding
design matrix leading to a small value of $\vartheta $.

Essentially, this is done by randomly generating $n\times k$ design matrices
(with first column the intercept, and the remaining coordinates
i.i.d.~log-(standard) normally distributed) and by computing the
corresponding values of $\vartheta $ for the testing problems $R=(0:I_{q})$
and $r=0$. In preparation for Step 2, for a suitably chosen subset of the
design matrices generated, we also compute null rejection probabilities for
strategically chosen variance parameters (assuming normality). All this is
done for every pair $(k,q)$ with $k=2,\ldots ,5$ and $q=1,\ldots ,k-1$.

\item The goals of Step 2 are twofold: (a) to check the numerical
reliability of the computation of $\vartheta $ in Step 1; and (b) to compute
lower bounds on the size of the test, if necessary. We do the following for $%
\alpha \in \{0.05,0.1\}$:

If $\vartheta _{\min }$, the overall smallest $\vartheta $ identified in
Step 1, turns out to be smaller than $\alpha $, we further check the
numerical reliability of $\vartheta _{\min }$ by making use of the null
rejection probabilities computed in Step 1. If this numerical check,
described in Section \ref{sec:stp2} in Appendix \ref{App compu}, is not
passed, we update $\vartheta _{\min }$. Once this check is passed, we
distinguish two cases: (i) If $\vartheta _{\min }<\alpha $ or if the maximum
of the null rejection probabilities just referred to (maximized over the
strategically chosen variance parameters) exceeds $3\alpha $, we stop. For
these cases we report the value of $\vartheta _{\min }$ together with the
maximal rejection probability obtained for the design matrix pertaining to $%
\vartheta _{\min }$. (ii) For the exceptional set of tests for which $%
\vartheta _{\min }\geq \alpha $ and the maximum of the null rejection
probabilities does not exceed $3\alpha $ we perform a second search (again
sampling as in Step 1) to find design matrices leading to high rejection
probabilities under the null. For these cases we report the highest null
rejection probability found, and the value of $\vartheta $ corresponding to
the design matrix that led to the highest null rejection probability.
\end{enumerate}

\subsection{Results and discussion\label{sec:results}}

The results of the two-step procedure described in Section \ref{sec:comp}
are summarized in Figure \ref{fig:p} in the form of $6$ plots corresponding
to the six combinations of the three settings A, B, and C, and of the two
values for $\alpha $ ($\alpha \in \{0.05,0.1\}$). In each plot, the vertical
dashed line intersects the axis at $\alpha $, the lower (upper,
respectively) horizontal dashed line intersects the axis at $\alpha $ (at $%
3\alpha $, respectively).

For every combination of the setting and the value of $\alpha $, the plot is
obtained as follows: for every bootstrap-based test procedure (i.e.,
combination of test statistic and bootstrap scheme), a null rejection
probability is plotted against a corresponding $\vartheta $ indicated by a
black or red circle. The black circles correspond to test procedures for
which Step 2 terminated without starting a second set of searches (which was
the case for the vast majority of procedures, see Footnote \ref{FN_vast} in
Appendix \ref{App compu}). The red circles correspond to the remaining
(exceptional) test procedures for which further null rejection probabilities
were computed in Step 2.

For all test procedures corresponding to black circles with $\vartheta
<\alpha $, the reliability check applied in Step 1 guarantees that the
corresponding null rejection probability found in Step 1 is greater than $%
0.4 $. Note that these null rejection probabilities do \emph{not} coincide
with the sizes of the respective tests, which actually all are equal to $1$
by our theoretical results; they are only lower bounds for the size that are
reported for completeness. For all black circles with $\vartheta \geq \alpha 
$, the null rejection probabilities are not less than $3\alpha $, which can
be gathered from an inspection of the plots (and which is so by construction
of the two-step procedure, see Section \ref{sec:stp2} in Appendix \ref{App
compu}); hence, they are much too large compared to the nominal significance
level $\alpha $. A bootstrap-based test procedure corresponding to a black
circle hence \textquotedblleft fails the worst-case check\textquotedblright\
(in the setting and for the $\alpha $ considered), because a scenario (i.e., 
$k$ and $q$) and a corresponding design matrix has been found for which the
size of the test is $1$, or is exceedingly large, i.e., larger than or equal
to $3\alpha $.

For the test procedures corresponding to red circles extra computations were
carried out in Step 2. A test procedure resulting in a red circle is
declared to \textquotedblleft fail the worst-case check\textquotedblright\
(in the setting and for the $\alpha $ considered) if the null rejection
probability plotted is not less than $3\alpha $ (as then a scenario (i.e., $k
$ and $q$) and a corresponding design matrix have been found for which the
size of the test is exceedingly large, namely larger than or equal to $%
3\alpha $); otherwise it is declared to \textquotedblleft pass the
worst-case check\textquotedblright\ (in the setting and for the $\alpha $
considered) \emph{for the time being}. [There are a few instances of red
circles for which the null rejection probability plotted is less than $%
3\alpha $, but $\vartheta <\alpha $ holds. While our theoretical results
then tell us that the size of the test should be equal to $1$, and we hence
should classify the test procedure as "failing the worst-case check", we do
not do so as we do not want to rely too much on the information provided by $%
\vartheta $ in such cases, since no reliability check for computing $%
\vartheta $ is included in the computation in Step 2.]

Bootstrap-based test procedures that fail the worst-case check (in at least
one setting and for one of the values of $\alpha $) should thus not be
expected to be reliable in general. Therefore, such a test procedure should
not be used in practice without first obtaining further guarantees
concerning its size properties in the specific problem at hand, e.g., by
running additional simulations geared towards the problem at hand.

Figure \ref{fig:p} shows that in all settings and for both significance
levels considered, the vast majority of circles are black (see Footnote \ref%
{FN_vast} in Appendix \ref{App compu}), already leading to the conclusion
that most bootstrap-based test procedures fail the worst-case check.
Furthermore, most of these test procedures fail in such a way that the
corresponding $\vartheta $ is smaller than the significance level $\alpha $,
and thus the test is known to have size equal to $1$ (at least in one of the
scenarios and for one of the design matrices considered) as a consequence of
our theoretical results. Inspection of Figure \ref{fig:p} also shows that a
good portion of the test procedures corresponding to red circles fail the
worst-case check in that the red circles are above or on the $3\alpha $%
-line. Comparing the figures across different settings and values of $\alpha 
$, we also see that the number of red circles decreases when passing from $%
\alpha =0.05$ to $\alpha =0.1$. The figure also shows that the number of red
circles increases when increasing $n$. One reason could be that the
randomized search for design matrices leading to low values of $\vartheta $
becomes more difficult as $n$ increases. We used the same randomized search
algorithm in all three scenarios, which could explain the difference. A
conceptual difference between the methods used in Setting A and Settings B
and C is that for Setting A ($n=10$) exact computations (concerning $\Xi
^{\bullet }$) are carried out, while for Settings B ($n=20$) and C ($n=30$)
approximate computations (based on empirical distributions $\Xi ^{\bullet }$%
) are done. This introduces an additional source of variation in the
computations in Settings B and C, which could also be responsible for the
increase in the number of red circles.

Important questions now are whether (i) there is a bootstrap-based test
procedure left that passes the worst-case check in all settings and for both
significance levels considered; (ii) there is a test procedure that passes
the worst-case check in all settings considered for a fixed $\alpha $; (iii)
there is a pattern, in the sense that certain combinations of test
statistics and bootstrap schemes often pass the worst-case check? To answer
such questions, we shall next provide information on the procedures that
pass the worst-case check in each setting and for each $\alpha $ considered.

\begin{figure}
	\centering
	\includegraphics[width=\linewidth]{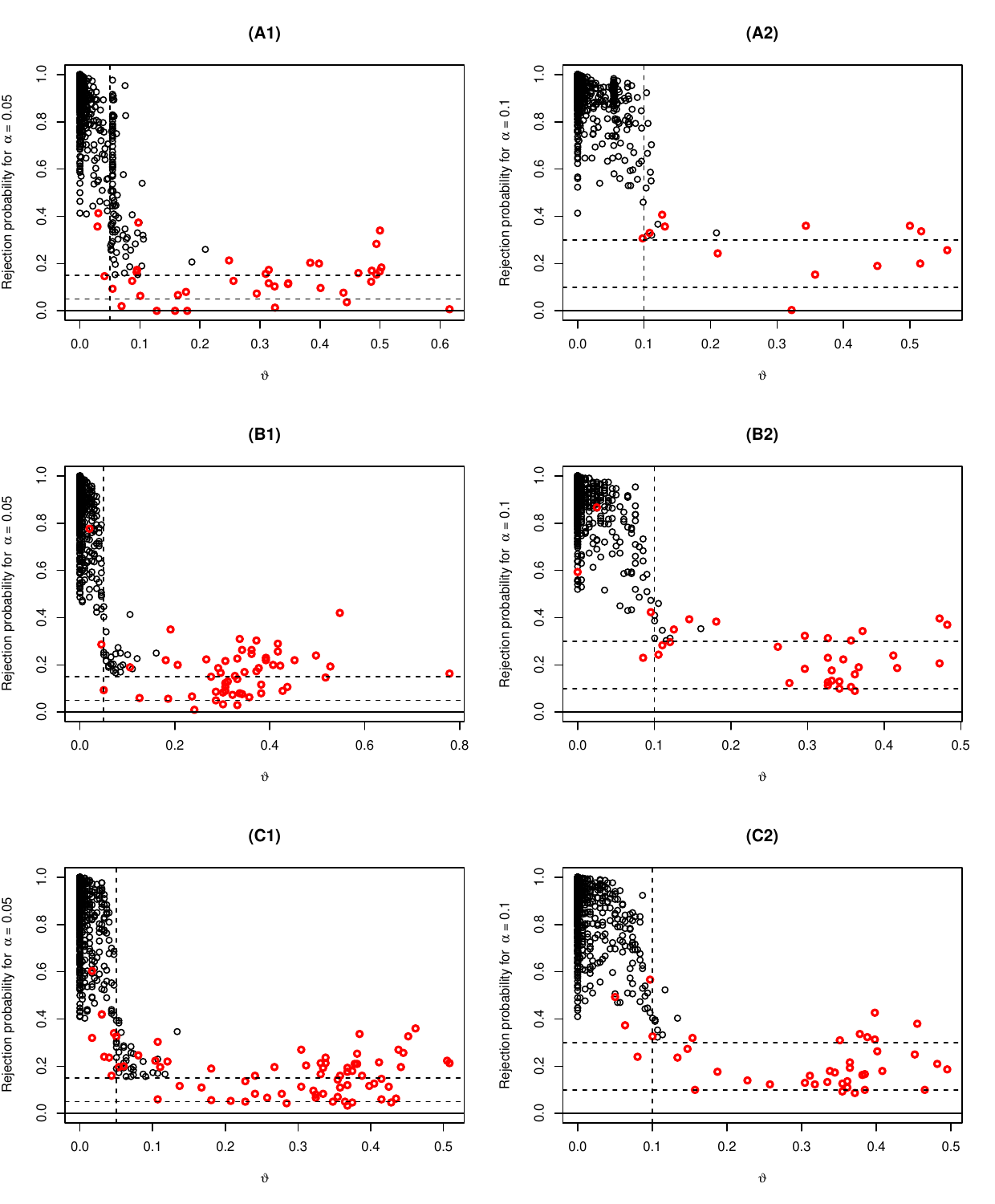}
	\caption{Results in Settings A, B, and C.}
	\label{fig:p}
\end{figure}

%\FRAME{fphFU}{5.8141in}{7.0275in}{0pt}{\Qcb{Results in Settings A, B, and C.}%
%}{\Qlb{fig:p}}{p.jpg}{\special{language "Scientific Word";type
%"GRAPHIC";maintain-aspect-ratio TRUE;display "USEDEF";valid_file "F";width
%5.8141in;height 7.0275in;depth 0pt;original-width 8.2667in;original-height
%9.9998in;cropleft "0";croptop "1";cropright "1";cropbottom "0";filename
%'p.jpg';file-properties "XNPEU";}}

The bootstrap-based test procedures which (for the time being) pass the
worst-case check in Setting A, B, and C, respectively, are summarized in
Tables \ref{tab:A}, \ref{tab:B}, and \ref{tab:C}. In each table, the first
row contains the test procedures that pass the worst-case check for $\alpha
=0.05$, the second row the ones that pass for $\alpha =0.1$, and the third
row the ones that pass at both nominal levels of significance.

In these tables, to facilitate the exposition, we use the following way of
encoding a bootstrap-based test procedure: to each of the $960$ possible
procedures (i.e., combinations of test statistics and bootstrap schemes)
considered we associate a $7$ digit code: $x=x_{1}{:}x_{2}{:}x_{3}{:}x_{4}{:}%
x_{5}{:}x_{6}{:}x_{7}$. The encoding of the digits is as follows:

\begin{enumerate}
\item[$x_{1}$ ...] indicates which covariance matrix estimator is used in
the test statistic (\textquotedblleft -1" stands for the uncorrected
estimator based on restricted or unrestricted residuals depending on whether 
$x_{2}$ is set to F or T; and $0,...,4$ stand for HC0,...,HC4, respectively,
or for HC0R,...,HC4R, respectively, depending on whether $x_{2}$ is set to F
or T).

\item[$x_{2}$ ...] indicates whether the covariance matrix estimator used in
the test statistic is based on null-restricted residuals (T) or not (F).

\item[$x_{3}$ ...] indicates the distribution underlying $\Xi ^{\bullet }$
(\textquotedblleft r\textquotedblright\ stands for Rademacher,
\textquotedblleft m\textquotedblright\ stands for Mammen).

\item[$x_{4}$ ...] indicates the weights $w$ used in constructing $\Xi $
from $\Xi ^{\bullet }$ ($0,...,4$ stands for HC0,...,HC4, respectively, or
for HC0R,...,HC4R, respectively, depending on whether $x_{5}$ is set to F or
T).

\item[$x_{5}$ ...] indicates whether the weights $w$ are null-restricted (T)
or not (F).

\item[$x_{6}$ ...] indicates whether the bootstrap-scheme was based on
null-restricted residuals (T) (i.e., $\mathcal{A}=\mathfrak{M}_{0}$) or not
(F) (i.e. $\mathcal{A}=\limfunc{span}(X)$).

\item[$x_{7}$ ...] indicates whether the bootstrap scheme $y^{\ast }$ (T) or 
$y^{\maltese }$ (F) was used.
\end{enumerate}

As an example, the code\textquotedblleft -1:T:m:2:F:T:F\textquotedblright\
translates to the bootstrap-based test procedure which uses the test
statistic $\tilde{T}_{uc}$ (determined by the first two digits of the code)
and the following bootstrap scheme: $\Xi ^{\bullet }$ based on the Mammen
distribution, modified by HC2 weights (based on unrestricted residuals), and 
$y^{\maltese }$ with $\mathcal{A}=\mathfrak{M}_{0}$.

Before interpreting the results, we also need to explain why some test
procedures are struck out in Tables \ref{tab:A}, \ref{tab:B}, and \ref{tab:C}%
. Recall (e.g., from the discussion in the last but one paragraph of Section %
\ref{sec:descr}) that some of the $960$ procedures studied are in fact \emph{%
equivalent}, meaning that they lead to exactly the same bootstrap-based
test. Because the design matrices are generated anew for every of the $960$
procedures, the results found for equivalent procedures can be different.
Therefore, it can happen that a procedure passes the worst-case check, but
an equivalent version of this procedure does not. Now, a procedure is struck
out in a given table, if -- while this procedure passed the worst-case check
underlying the table -- an \emph{equivalent} procedure did not (and thus
does not appear in this table). Procedures that are struck out in a table
are now \emph{no longer} considered as having passed the worst-case check
(although they appear in that table).

We exploit the following three reasons for equivalence between test
procedures: (1) Lemma \ref{VGK} shows that a bootstrap-based test using a
covariance matrix estimator based on unrestricted residuals does not depend
on whether $y^{\ast }$ or $y^{\maltese }$ is used as a bootstrap scheme.
Therefore, two procedures with codes $x$ and $x^{\prime }$ are equivalent in
case $x_{2}=x_{2}^{\prime }=F$ and $x_{i}=x_{i}^{\prime }$ for $i=1,\ldots
,6 $. (2) Changing the weights vector from HC0 to HC0R (or vice versa) in
the construction of $\Xi $ does not change the bootstrap-based test.
Therefore, two procedures with codes $x$ and $x^{\prime }$ are equivalent in
case $x_{4}=x_{4}^{\prime }=0$, and $x_{i}=x_{i}^{\prime }$ for all $i\neq 5$%
. (3) Changing HC0 (HC0R) weights to HC1 (HC1R) weights in the definition of 
$T_{Het}$ ($\tilde{T}_{Het}$) also does not change the resulting
bootstrap-based test, cf.~Remark \ref{rem:equiv}. Therefore, two procedures
with codes $x$ and $x^{\prime }$ are equivalent in case $x_{1}\in \{0,1\}$, $%
x_{1}^{\prime }\in \{0,1\}$, and $x_{i}=x_{i}^{\prime }$ for all $i=2,\ldots
,7$. Procedures $x$ that are struck out by a slash, i.e., $\cancel{x}$, were
eliminated based on reason (1); procedures that are struck out by a
backslash, i.e., $\bcancel{x}$, were eliminated based on reason (2); and
procedures that are struck out by a horizontal line, i.e., \sout{$x$}, were
eliminated based on reason (3). Note that a procedure can be struck out,
e.g., based on reasons (1) and (2), and thus is then crossed out, i.e., is
marked by $\xcancel{x}$, etc. 
\begin{table}[th]
\centering
\begin{tabular}{l|llll}
\hline\hline
\multirow{6}{*}{$\alpha = .05$} & -1:T:m:0:F:T:F & -1:T:m:0:T:T:F & 
-1:T:m:1:F:T:F & -1:T:m:1:T:T:F \\ 
& -1:T:m:2:F:T:F & -1:T:m:2:T:T:F & -1:T:m:3:F:T:F & -1:T:m:3:T:T:F \\ 
& \xcancel{3:F:r:0:F:T:T} & \xcancel{3:F:m:0:T:T:F} & \cancel{3:F:m:1:F:T:F}
& \cancel{3:F:m:1:T:T:F} \\ 
& 3:F:m:2:F:T:F & 3:F:m:2:F:T:T & 3:T:r:2:T:T:T & 3:T:r:3:F:T:F \\ 
& 3:T:r:3:T:T:T & \bcancel{3:T:m:0:F:T:T} & 3:T:m:1:T:T:T & 3:T:m:2:T:T:T \\ 
& 3:T:m:3:F:F:F & \bcancel{4:T:m:0:T:T:T} &  &  \\ \hline
\multirow{2}{*}{$\alpha = .1$} & -1:T:m:2:F:T:F & -1:T:m:3:F:T:F & %
\cancel{3:F:r:2:F:T:F} & 3:F:m:2:F:T:F \\ 
& 3:F:m:2:F:T:T & \cancel{4:F:m:2:F:T:F} &  &  \\ \hline
\multirow{1}{*}{both} & -1:T:m:2:F:T:F & -1:T:m:3:F:T:F & 3:F:m:2:F:T:F & 
3:F:m:2:F:T:T \\ \hline\hline
\end{tabular}%
\caption{Surviving test procedures in Setting A.}
\label{tab:A}
\end{table}

\begin{table}[th]
\centering
\begin{tabular}{l|llll}
\hline\hline
\multirow{7}{*}{$\alpha = .05$} & -1:T:r:3:F:T:F & \sout{\cancel{1:F:m:3:F:F:F}} & %
\cancel{2:F:r:3:F:F:F} & \xcancel{2:F:m:0:T:T:F} \\ 
& \xcancel{3:F:m:0:F:T:F} & \cancel{3:F:m:2:F:T:F} & \cancel{3:F:m:3:F:F:F}
& \bcancel{3:T:r:0:T:F:T} \\ 
& \bcancel{3:T:r:0:T:T:T} & 3:T:r:1:T:F:T & 3:T:r:1:T:T:T & 3:T:r:2:T:F:T \\ 
& 3:T:r:3:F:T:F & 3:T:m:0:F:F:T & 3:T:m:0:F:T:T & 3:T:m:0:T:F:T \\ 
& 3:T:m:0:T:T:T & 3:T:m:1:F:F:T & 3:T:m:1:F:T:T & 3:T:m:1:T:F:T \\ 
& 3:T:m:1:T:T:T & 3:T:m:2:F:F:T & 3:T:m:2:T:F:T & 3:T:m:2:T:T:T \\ 
& \bcancel{4:T:m:0:T:T:T} &  &  &  \\ \hline
\multirow{6}{*}{$\alpha = .1$} & -1:T:r:3:F:F:F & -1:T:r:3:F:T:F & %
\sout{\cancel{1:F:r:3:F:F:F}} & \sout{\cancel{1:F:m:3:F:F:F}} \\ 
& \xcancel{3:F:m:0:T:T:T} & 3:F:m:2:F:T:F & 3:F:m:2:F:T:T & %
\cancel{3:F:m:3:F:T:T} \\ 
& \bcancel{3:T:r:0:T:T:T} & 3:T:r:2:F:F:T & 3:T:r:2:F:T:T & 3:T:r:2:T:F:T \\ 
& 3:T:r:3:F:F:F & \bcancel{3:T:m:0:F:F:T} & 3:T:m:0:F:T:T & 3:T:m:0:T:T:T \\ 
& 3:T:m:1:F:F:T & 3:T:m:1:F:T:T & 3:T:m:1:T:F:T & 3:T:m:1:T:T:T \\ 
& 3:T:m:2:F:F:T & 3:T:m:2:T:T:T &  &  \\ \hline
\multirow{4}{*}{both} & -1:T:r:3:F:T:F & \sout{\cancel{1:F:m:3:F:F:F}} & %
\cancel{3:F:m:2:F:T:F} & \cancel{3:T:r:0:T:T:T} \\ 
& 3:T:r:2:T:F:T & \bcancel{3:T:m:0:F:F:T} & 3:T:m:0:F:T:T & 3:T:m:0:T:T:T \\ 
& 3:T:m:1:F:F:T & 3:T:m:1:F:T:T & 3:T:m:1:T:F:T & 3:T:m:1:T:T:T \\ 
& 3:T:m:2:F:F:T & 3:T:m:2:T:T:T &  &  \\ \hline\hline
\end{tabular}%
\caption{Surviving test procedures in Setting B.}
\label{tab:B}
\end{table}

\begin{table}[th]
\centering
\begin{tabular}{l|llll}
\hline\hline
\multirow{8}{*}{$\alpha = .05$} & \sout{\cancel{0:F:m:3:F:F:F}} & %
\sout{\cancel{1:F:m:3:F:F:T}} & \xcancel{2:F:m:0:T:T:F} & \cancel{2:F:m:3:T:T:T} \\ 
& \xcancel{3:F:r:0:T:T:T} & \cancel{3:F:r:1:T:T:F} & 3:F:r:2:F:T:F & 
3:F:r:2:F:T:T \\ 
& \xcancel{3:F:m:0:F:T:T} & \cancel{3:F:m:1:F:T:T} & 3:F:m:2:F:T:F & 
3:F:m:2:F:T:T \\ 
& \cancel{3:F:m:2:T:T:F} & \bcancel{3:T:r:0:T:F:T} & \bcancel{3:T:r:0:T:T:T}
& 3:T:r:2:F:F:T \\ 
& 3:T:r:2:T:T:T & 3:T:m:0:F:F:T & 3:T:m:0:F:T:T & 3:T:m:0:T:F:T \\ 
& 3:T:m:0:T:T:T & 3:T:m:1:F:F:T & 3:T:m:1:F:T:T & 3:T:m:1:T:F:T \\ 
& 3:T:m:1:T:T:T & 3:T:m:2:F:F:T & 3:T:m:2:T:F:T & 3:T:m:2:T:T:T \\ 
& \bcancel{4:T:r:0:F:T:T} & 4:T:m:0:F:T:T & 4:T:m:0:T:T:T &  \\ \hline
\multirow{8}{*}{$\alpha = .1$} & -1:T:r:3:F:F:F & \cancel{0:F:r:3:F:F:F} & %
\sout{\cancel{0:F:m:3:F:F:F}} & \cancel{1:F:r:3:F:F:F} \\ 
& \sout{\cancel{1:F:m:3:F:T:F}} & \cancel{2:F:m:3:F:T:T} & \cancel{2:F:m:3:T:T:T} & %
\cancel{3:F:r:2:F:T:F} \\ 
& \cancel{3:F:r:2:T:T:F} & \cancel{3:F:m:1:T:T:T} & \cancel{3:F:m:2:F:T:F} & 
3:T:r:0:F:F:T \\ 
& 3:T:r:0:T:F:T & \bcancel{3:T:r:0:T:T:T} & 3:T:r:1:F:F:T & 3:T:r:1:T:T:T \\ 
& 3:T:r:2:F:F:T & 3:T:m:0:F:F:T & 3:T:m:0:F:T:T & 3:T:m:0:T:F:T \\ 
& 3:T:m:0:T:T:T & 3:T:m:1:F:F:T & 3:T:m:1:T:F:T & 3:T:m:1:T:T:T \\ 
& 3:T:m:2:F:F:T & 3:T:m:2:T:F:T & 3:T:m:2:T:T:T & 3:T:m:3:F:F:T \\ 
& 3:T:m:3:T:F:T &  &  &  \\ \hline
\multirow{5}{*}{both} & \sout{\cancel{0:F:m:3:F:F:F}} & \cancel{2:F:m:3:T:T:T} & %
\cancel{3:F:r:2:F:T:F} & \cancel{3:F:m:2:F:T:F} \\ 
& \bcancel{3:T:r:0:T:F:T} & \bcancel{3:T:r:0:T:T:T} & 3:T:r:2:F:F:T & 
3:T:m:0:F:F:T \\ 
& 3:T:m:0:F:T:T & 3:T:m:0:T:F:T & 3:T:m:0:T:T:T & 3:T:m:1:F:F:T \\ 
& 3:T:m:1:T:F:T & 3:T:m:1:T:T:T & 3:T:m:2:F:F:T & 3:T:m:2:T:F:T \\ 
& 3:T:m:2:T:T:T &  &  &  \\ \hline\hline
\end{tabular}%
\caption{Surviving test procedures in Setting C.}
\label{tab:C}
\end{table}

Concerning the questions (i)-(iii) raised above, inspection of Tables \ref%
{tab:A}, \ref{tab:B}, and \ref{tab:C} now delivers the following answers:

\begin{enumerate}
\item There is no bootstrap-based test procedure that passes the worst-case
check in all settings and for both significance levels.

\item The tests 3:T:m:1:T:T:T and 3:T:m:2:T:T:T pass the worst-case checks
in all three settings for the significance level $\alpha =0.05$. The
corresponding rejection probabilities shown in Figure \ref{fig:p} for the
tests 3:T:m:1:T:T:T and 3:T:m:2:T:T:T are $0.077$ and $0.067$ (Setting A), $%
0.033$ and $0.080$ (Setting B), $0.050$ and $0.033$ (Setting C),
respectively. For $\alpha =0.1$ there is no test that passes the worst-case
checks in all three settings.

\item The majority of test procedures appearing in the tables is based on
the \emph{HC3 or HC3R covariance estimator}, and uses a bootstrap scheme
based on the \emph{Mammen-distribution}. In Settings B and C, the tests
passing the worst-case check are typically based on $\tilde{T}_{Het}$, i.e.,
they use \emph{restricted residuals in the construction of the covariance
estimator}.
\end{enumerate}

On the one hand our results issue a distinct warning: overall \emph{none} of
the bootstrap-based tests considered comes with a guarantee that its size is
(about) right. On the other hand, when restricting attention only to $\alpha
=0.05$, the tests 3:T:m:1:T:T:T and 3:T:m:2:T:T:T did not break down in our
worst-case analysis.\footnote{%
However, recall that we have only computed a lower bound for the size.}
These two tests are based on $\tilde{T}_{Het}$ using a HC3R covariance
estimator, and use the Mammen-distribution in the bootstrap scheme;
properties that are common to many of the tests that pass the check (in one
of the settings considered). While this obviously does not prove that
3:T:m:1:T:T:T and 3:T:m:2:T:T:T always will have perfect size properties for 
$\alpha =0.05$, it shows that in the settings considered (and for $\alpha
=0.05$) they seem to have the best size performance among all
bootstrap-based tests considered, and should therefore perhaps be preferred
(over the other procedures) by practitioners who insist on applying a
bootstrap-based test. However one should keep in mind that, while we have
examined a considerable and reasonable range of scenarios and design
matrices, the size-behavior of the bootstrap-based tests outside of the
range studied can potentially be even worse.

In light of the findings above a better way forward seems to use
heteroskedasticity robust test procedures that guarantee size-control as
expounded in \cite{PP5HC}.

\appendix{}

\section{Appendix: A basic theorem\label{App A}}

\begin{theorem}
\label{thm:main}Let $T_{i}:\mathbb{R}^{n}\rightarrow \mathbb{R\cup \{\infty
\}}$ be Borel-measurable and $G(\mathfrak{M}_{0})$-invariant for $i=1,2$. Let%
$\mathbb{\ }\Xi $ be a (Borel) probability measure on $\mathbb{R}^{n}$, and
let $\mathcal{A}$ be an affine subspace of $\mathbb{R}^{n}$ with $\mathfrak{M%
}_{0}\subseteq \mathcal{A}\subseteq \limfunc{span}(X)$. Define the function $%
T_{2}^{\ast }:\mathbb{R}^{n}\times \mathbb{R}^{n}\rightarrow \mathbb{R\cup
\{\infty \}}$ via 
\begin{equation}
T_{2}^{\ast }(y,\xi )=T_{2}\left( y^{\ast }(y,\xi )\right) ,
\label{eqn:T2boot}
\end{equation}%
where $y^{\ast }(y,\xi )$ has been defined in (\ref{boot_1}), and set 
\begin{equation*}
K(y)=\{\xi \in \mathbb{R}^{n}:y^{\ast }(y,\xi )\text{ is a continuity point
of }T_{2}\}.
\end{equation*}%
For every $y\in \mathbb{R}^{n}$ denote by $F_{y}$ the distribution function
of $\Xi \circ T_{2}^{\ast }(y,\cdot )$, i.e., $F_{y}(t)=\Xi \mathbb{(\{}\xi
:T_{2}^{\ast }(y,\xi )\leq t\})$ for $t\in \mathbb{R\cup \{\infty \}}$. For
every $\alpha \in (0,1)$, let $f_{1-\alpha }(y)$ denote a $(1-\alpha )$%
-quantile of $F_{y}$. For every $i=1,\ldots ,n$ define%
\begin{equation}
c_{i}=\sup_{\delta >0}\inf_{z\in B(e_{i}(n),\delta )}T_{1}(\mu _{0}+z),
\label{eqn:c_i}
\end{equation}%
for some $\mu _{0}\in \mathfrak{M}_{0}$, where $B(e_{i}(n),\delta )=\{z\in 
\mathbb{R}^{n}:\Vert z-e_{i}(n)\Vert <\delta \}$ (note that $c_{i}$ does not
depend on the choice of $\mu _{0}\in \mathfrak{M}_{0}$). Then, for every $%
\alpha \in (0,1)$ such that 
\begin{equation}
\alpha >1-\max_{i=1,\ldots ,n}\Xi \left( \left\{ \xi :T_{2}^{\ast }(\mu
_{0}+e_{i}(n),\xi )<c_{i},\xi \in K(\mu _{0}+e_{i}(n))\right\} \right)
\label{eqn:newbound}
\end{equation}%
for some (and hence all) $\mu _{0}\in \mathfrak{M}_{0}$, we have%
\begin{equation}
\sup_{\Sigma \in \mathfrak{C}_{Het}}P_{\mu _{0},\sigma ^{2}\Sigma }\left(
T_{1}\geq f_{1-\alpha }\right) \geq \sup_{\Sigma \in \mathfrak{C}%
_{Het}}P_{\mu _{0},\sigma ^{2}\Sigma }\left( T_{1}>f_{1-\alpha }\right) =1
\label{eqn:bootsize}
\end{equation}%
for every $\mu _{0}\in \mathfrak{M}_{0}$ and every $0<\sigma ^{2}<\infty $
(where the probabilities in (\ref{eqn:bootsize}) are to be interpreted as
inner probabilities\footnote{%
This allows one to ignore measurability issues regarding $f_{1-\alpha }$.}).%
\footnote{%
The set appearing in (\ref{eqn:newbound}) is a Borel set since $T_{2}^{\ast
}(y,\cdot )$ is Borel measurable and since $K(y)$ is a Borel set.}
\end{theorem}

\begin{proof}
By $G(\mathfrak{M}_{0})$-invariance of $T_{i}$, for every $z\in \mathbb{R}%
^{n}$ the expression $T_{i}(\mu _{0}+\gamma z)$ depends neither on the
choice of $\mu _{0}\in \mathfrak{M}_{0}$ nor on the value of $\gamma \in 
\mathbb{R}\backslash \{0\}$ (for $i=1,2$).\footnote{\label{FN:equi}Let $\mu
_{0}^{\prime }\in \mathfrak{M}_{0}$ and $\gamma ^{\prime }\in \mathbb{R}%
\backslash \{0\}$ be arbitrary. By $G(\mathfrak{M}_{0})$-invariance, $%
T_{i}(\mu _{0}+\gamma z)=T_{i}(h(\mu _{0}+\gamma z))$ where $h(v)=\gamma
^{\prime }\gamma ^{-1}(v-\mu _{0})+\mu _{0}^{\prime }\in G(\mathfrak{M}_{0})$%
. This shows that $T_{i}(\mu _{0}+\gamma z)=T_{i}(\mu _{0}^{\prime }+\gamma
^{\prime }z)$.} This shows, in particular, that the $c_{i}$'s do not depend
on the choice of $\mu _{0}\in \mathfrak{M}_{0}$. Furthermore, since $%
\mathfrak{M}_{0}\subseteq \mathcal{A}$ has been assumed, it is easy to see
that $\mathcal{A-}\mu _{0}$ is a linear space containing $\mathfrak{M}%
_{0}^{lin}$ for every choice of $\mu _{0}\in \mathfrak{M}_{0}$ (with $%
\mathcal{A-}\mu _{0}$ being the same space regardless of the choice of $\mu
_{0}\in \mathfrak{M}_{0}$), and that $y-X\tilde{\beta}_{\mathcal{A}}(y)=\Pi
_{(\mathcal{A-}\mu _{0})^{\bot }}(y-\mu _{0})$ holds for every $y\in \mathbb{%
R}^{n}$ and for every choice of $\mu _{0}\in \mathfrak{M}_{0}$. It now
easily follows that $y^{\ast }(\mu _{0}+y,\xi )=y^{\ast }(\mu _{0}^{\prime
}+y,\xi )-\mu _{0}^{\prime }+\mu _{0}$ for every $\mu _{0}$, $\mu
_{0}^{\prime }\in \mathfrak{M}_{0}$. In view of $G(\mathfrak{M}_{0})$%
-invariance of $T_{2}$ we can conclude that the r.h.s. of (\ref{eqn:newbound}%
) also does not depend on the choice of $\mu _{0}\in \mathfrak{M}_{0}$. For
later use we also make the following observation: Since $X\tilde{\beta}_{%
\mathfrak{M}_{0}}(y)$ obviously belongs to $\mathfrak{M}_{0}$, we have for
every $y\in \mathbb{R}^{n}$, every $\xi \in \mathbb{R}^{n}$, and every $\mu
_{0}\in \mathfrak{M}_{0}$%
\begin{eqnarray}
T_{2}^{\ast }(y,\xi ) &=&T_{2}(X\tilde{\beta}_{\mathfrak{M}_{0}}(y)+\limfunc{%
diag}(\xi )(y-X\tilde{\beta}_{\mathcal{A}}(y))=T_{2}(\mu _{0}+\limfunc{diag}%
(\xi )(y-X\tilde{\beta}_{\mathcal{A}}(y))  \notag \\
&=&T_{2}(\mu _{0}+\limfunc{diag}(\xi )\Pi _{(\mathcal{A-}\mu _{0})^{\bot
}}(y-\mu _{0}))  \label{eqn:invboot}
\end{eqnarray}%
in view of what has been said at the beginning of the proof.

Now, denote the set of continuity points of $T_{2}$ by $K_{2}$, and define $%
\bar{T}_{2}:=T_{2}\mathbf{1}_{K_{2}}+\infty \mathbf{1}_{\mathbb{R}%
^{n}\backslash K_{2}}$ (with the convention $\infty \cdot 0=0$). Define $%
\bar{T}_{2}^{\ast }$ analogously to $T_{2}^{\ast }$ (cf. (\ref{eqn:T2boot}%
)), but with $T_{2}$ replaced by $\bar{T}_{2}$. For every $y\in \mathbb{R}%
^{n}$, we let $\bar{F}_{y}$ denote the distribution function of $\Xi \circ 
\bar{T}_{2}^{\ast }(y,\cdot )$, i.e., $\bar{F}_{y}(t)=\Xi (\mathbb{\{}\xi :%
\bar{T}_{2}^{\ast }(y,\xi )\leq t\})$ for $t\in \mathbb{R\cup \{\infty \}}$.
It easily follows that for every $y\in \mathbb{R}^{n}$ and every $t\in 
\mathbb{R}$ we have $\bar{F}_{y}(t)=\Xi (\{\xi :T_{2}^{\ast }(y,\xi )\leq
t,\xi \in K(y)\})$. Consequently, for every $y\in \mathbb{R}^{n}$ and every $%
t\in \mathbb{R}\cup \{\infty \}$, we have $\bar{F}_{y}(t-)=\Xi (\{\xi
:T_{2}^{\ast }(y,\xi )<t,\xi \in K(y)\})$, where $\bar{F}_{y}(t-)$ denotes
the left-hand side limit of $\bar{F}_{y}$ at $t$. In particular, (\ref%
{eqn:newbound}) is equivalent to 
\begin{equation*}
\max_{i=1,\ldots ,n}\bar{F}_{\mu _{0}+e_{i}(n)}(c_{i}-)>1-\alpha
\end{equation*}%
(with the convention that $\bar{F}_{\mu _{0}+e_{i}(n)}(c_{i}-)=0$ in case $%
c_{i}=-\infty $). From now on, let $i$ be an index that realizes the maximum
in the previous display. If $\bar{F}_{\mu _{0}+e_{i}(n)}(c_{i}-)=0$, there
is nothing to prove and we are done. Hence, it remains to consider the case
where $\bar{F}_{\mu _{0}+e_{i}(n)}(c_{i}-)>0$. Note that this implies $%
c_{i}>-\infty $. In this case, let $\alpha \in (0,1)$ now be such that $\bar{%
F}_{\mu _{0}+e_{i}(n)}(c_{i}-)>1-\alpha $ (where $\mu _{0}\in \mathfrak{M}%
_{0}$ can be chosen arbitrarily). From $\bar{F}_{\mu
_{0}+e_{i}(n)}(c_{i}-)>1-\alpha $ we can then conclude existence of a real
number $x_{i}$ smaller than $c_{i}$ (recall $c_{i}>-\infty $ must hold) such
that $\bar{F}_{\mu _{0}+e_{i}(n)}(x_{i})>1-\alpha $ holds and such that $%
x_{i}$ is a continuity point of $\bar{F}_{\mu _{0}+e_{i}(n)}$. In view of (%
\ref{eqn:c_i}), there exists a $\delta >0$ such that every $z\in
B(e_{i}(n),\delta )$ satisfies $T_{1}(\mu _{0}+z)>x_{i}$ (and the same is
true if we replace $\delta $ by a smaller positive number). We claim that
for every sequence $z_{m}\rightarrow e_{i}(n)$ ($z_{m}\in \mathbb{R}^{n}$)
we have 
\begin{equation*}
\liminf_{m\rightarrow \infty }F_{\mu _{0}+z_{m}}(x_{i})\geq \bar{F}_{\mu
_{0}+e_{i}(n)}(x_{i}).
\end{equation*}%
Define $V_{m}=V_{m}(\xi )=y^{\ast }(\mu _{0}+z_{m},\xi )$ and $V=V(\xi
)=y^{\ast }(\mu _{0}+e_{i}(n),\xi )$, which can be viewed as random vectors
defined on $\mathbb{R}^{n}$ (equipped with the Borel $\sigma $-field) and
where the probability measure is given by $\Xi $. Note that $V_{m}$
converges to $V$ everywhere as $m\rightarrow \infty $ (as $y^{\ast }(y,\xi )$
is continuous w.r.t. $y$). Furthermore, $F_{\mu _{0}+z_{m}}(x_{i})=\Xi
(T_{2}(V_{m})\leq x_{i})$ and $\bar{F}_{\mu _{0}+e_{i}(n)}(x_{i})=\Xi (\bar{T%
}_{2}(V)\leq x_{i})$ hold (recall that $x_{i}$ is a real number). The
statement in the previous display now follows from Lemma \ref{lem:convdisc},
recalling that we have chosen $x_{i}$ as a continuity point of $\bar{F}_{\mu
_{0}+e_{i}(n)}$, which implies $\Xi (\bar{T}_{2}(V)=x_{i})=0$. Summarizing,
we hence arrive, replacing $\delta $ by another element of $(0,\delta )$ if
necessary, at%
\begin{equation}
T_{1}(\mu _{0}+z)>x_{i}\text{ and }F_{\mu _{0}+z}(x_{i})>1-\alpha \text{ for
every }z\in B(e_{i}(n),\delta ).  \label{eqn:bootdelta_2}
\end{equation}%
From (\ref{eqn:invboot}) and the observation in the first sentence in this
proof it readily follows that for every $z\in \mathbb{R}^{n}$ and every $%
\gamma \neq 0$ we have 
\begin{align*}
F_{\mu _{0}+\gamma z}(x_{i})& =\Xi (\xi :T_{2}^{\ast }(\mu _{0}+\gamma z,\xi
)\leq x_{i}) \\
& =\Xi (\xi :T_{2}(\mu _{0}+\limfunc{diag}(\xi )\Pi _{(\mathcal{A}-\mu
_{0})^{\bot }}(\mu _{0}+\gamma z-\mu _{0}))\leq x_{i}) \\
& =\Xi (\xi :T_{2}(\mu _{0}+\gamma \limfunc{diag}(\xi )\Pi _{(\mathcal{A}%
-\mu _{0})^{\bot }}(z))\leq x_{i}) \\
& =\Xi (\xi :T_{2}(\mu _{0}+\limfunc{diag}(\xi )\Pi _{(\mathcal{A}-\mu
_{0})^{\bot }}(z))\leq x_{i})=F_{\mu _{0}+z}(x_{i}).
\end{align*}%
Using the observation in the first sentence of the proof again, the
preceding display and (\ref{eqn:bootdelta_2}) thus give%
\begin{equation*}
T_{1}(\mu _{0}+y)>x_{i}\text{ and }F_{\mu _{0}+y}(x_{i})>1-\alpha \text{ for
every }y\in \{\gamma z:\gamma \neq 0,z\in B(e_{i}(n),\delta )\}=:\Delta ,
\end{equation*}%
or equivalently%
\begin{equation}
T_{1}(y)>x_{i}\text{ and }F_{y}(x_{i})>1-\alpha \text{ for every }y\in \mu
_{0}+\Delta .  \label{eqn:bootdelta_3}
\end{equation}%
By Lemma \ref{lem:quant} we see that $F_{y}(x_{i})>1-\alpha $ implies $%
x_{i}\geq f_{1-\alpha }(y)$. Hence,%
\begin{equation*}
\left\{ y\in \mathbb{R}^{n}:T_{1}(y)>x_{i},\text{ }F_{y}(x_{i})>1-\alpha
\right\} \subseteq \left\{ y\in \mathbb{R}^{n}:T_{1}(y)>f_{1-\alpha
}(y)\right\} .
\end{equation*}%
This, together with (\ref{eqn:bootdelta_3}), gives 
\begin{equation}
\left\{ y\in \mathbb{R}^{n}:T_{1}(y)>f_{1-\alpha }(y)\right\} \supseteq \mu
_{0}+\Delta \supseteq \mu _{0}+\limfunc{int}(\Delta )\supseteq \mu _{0}+(%
\limfunc{span}(e_{i}(n))\backslash \{0\}),  \label{eqn:bootincl_2}
\end{equation}%
where $\limfunc{int}(\Delta )$ denotes the interior of $\Delta $. Finally,
to complete the proof of (\ref{eqn:bootsize}), let $\Sigma _{m}$ be a
sequence in $\mathfrak{C}_{Het}$ that converges to $e_{i}(n)e_{i}(n)^{\prime
}$ and let $0<\sigma ^{2}<\infty $. Lemma E.1 in \cite{PP2016} then shows
that $P_{\mu _{0},\sigma ^{2}\Sigma _{m}}$ converges weakly to $P_{\mu
_{0},\sigma ^{2}e_{i}(n)e_{i}(n)^{\prime }}$. The previous display implies
that for every $m\in \mathbb{N}$%
\begin{equation*}
P_{\mu _{0},\sigma ^{2}\Sigma _{m}}\left( \left\{ y\in \mathbb{R}%
^{n}:T_{1}(y)>f_{1-\alpha }(y)\right\} \right) \geq P_{\mu _{0},\sigma
^{2}\Sigma _{m}}\left( \mu _{0}+\limfunc{int}(\Delta )\right) ,
\end{equation*}%
where the left-hand side is to be interpreted as an inner probability. The
Portmanteau theorem delivers%
\begin{eqnarray*}
\liminf_{m\rightarrow \infty }P_{\mu _{0},\sigma ^{2}\Sigma _{m}}\left( \mu
_{0}+\limfunc{int}(\Delta )\right) &\geq &P_{\mu _{0},\sigma
^{2}e_{i}(n)e_{i}(n)^{\prime }}\left( \mu _{0}+\limfunc{int}(\Delta )\right)
\\
&\geq &P_{\mu _{0},\sigma ^{2}e_{i}(n)e_{i}(n)^{\prime }}\left( \mu _{0}+(%
\limfunc{span}(e_{i}(n))\backslash \{0\})\right) =1,
\end{eqnarray*}%
where the second inequality follows from (\ref{eqn:bootincl_2}), and the
final equality from $P_{\mu _{0},\sigma ^{2}e_{i}(n)e_{i}(n)^{\prime }}$
being supported by $\mu _{0}+\limfunc{span}(e_{i}(n))$ and assigning
probability zero to $\{\mu _{0}\}$. This establishes (\ref{eqn:bootsize}).
\end{proof}

\begin{remark}
(i) If $T_{1}$ is lower semi-continuous at $\mu _{0}+e_{i}(n)$ for some (and
hence all) $\mu _{0}\in \mathfrak{M}_{0}$, then $c_{i}=T_{1}(\mu
_{0}+e_{i}(n))$.

(ii) If the set $K(\mu _{0}+e_{i}(n))$ has $\Xi $-measure $1$, then $\Xi
\left( \xi :T_{2}^{\ast }(\mu _{0}+e_{i}(n),\xi )<c_{i},\xi \in K(\mu
_{0}+e_{i}(n))\right) $ reduces to $F_{\mu _{0}+e_{i}(n)}(c_{i}-)$.

(iii) The lower bound for $\alpha $ in (\ref{eqn:newbound}) depends only on
observable quantities and thus can be computed.
\end{remark}

\begin{lemma}
\label{lem:convdisc} Let $(\Omega ,\mathcal{A},\mathbb{P})$ be a probability
space, and let $V_{m}$ be a sequence of $\mathbb{R}^{p}$-valued random
vectors ($p\in \mathbb{N}$) defined on that space that converges almost
everywhere to the random vector $V$. Let $T:\mathbb{R}^{p}\rightarrow 
\mathbb{R}\cup \{\infty \}$ be Borel measurable, denote the set of
continuity points of $T$ by $K_{T}$, and define $\bar{T}:=T\mathbf{1}%
_{K_{T}}+\infty \mathbf{1}_{\mathbb{R}^{p}\backslash K_{T}}$ (with the
convention $\infty \cdot 0=0$). Then, for every $t\in \mathbb{R}$ such that $%
\mathbb{P}(\bar{T}(V)=t)=0$ as well as for $t=\infty $, it holds that 
\begin{equation}
\mathbb{P}\left( \bar{T}(V)\leq t\right) \leq \liminf_{m\rightarrow \infty }%
\mathbb{P}\left( T(V_{m})\leq t\right) .  \label{eqn:ineqaux}
\end{equation}
\end{lemma}

\begin{proof}
It is well known that $K_{T}$ is a countable intersection of open sets, and
hence is a Borel set. As a consequence $\bar{T}:\mathbb{R}^{p}\rightarrow 
\mathbb{R}\cup \{\infty \}$ is Borel measurable. For $t=\infty $ the
probabilities in (\ref{eqn:ineqaux}) are all equal to $1$, and the
inequality therefore trivially holds in this case. Next, let $t\in \mathbb{R}
$ be such that $\mathbb{P}(\bar{T}(V)=t)=0$. Set $A_{m}:=\{\omega \in \Omega
:T(V_{m}(\omega ))\leq t\}$, $A:=\{\omega \in \Omega :\bar{T}(V(\omega
))\leq t\}$ and $\Omega ^{\prime }=\{\omega \in \Omega :\bar{T}(V(\omega
))\neq t,V_{m}(\omega )\rightarrow V(\omega )\}$. From the definition of $%
\bar{T}$ we obtain 
\begin{equation}
\limsup_{m\rightarrow \infty }T(V_{m}(\omega ))\leq \bar{T}(V(\omega ))\text{
for every }\omega \in \Omega \text{ such that }V_{m}(\omega )\rightarrow
V(\omega ).  \label{eqn:obsaux}
\end{equation}%
Furthermore, if $\omega \in A\cap \Omega ^{\prime }$, we have that $\bar{T}%
(V(\omega ))<t$ must hold and hence $T(V_{m}(\omega ))<t$ must hold
eventually in view of (\ref{eqn:obsaux}), i.e., that $\omega \in A_{m}$
eventually must be true. This implies that $\liminf_{m\rightarrow \infty }%
\mathbf{1}_{A_{m}}(\omega )\geq \mathbf{1}_{A}(\omega )$ for every $\omega
\in \Omega ^{\prime }$ (the inequality being trivial for $\omega \notin A$),
and thus almost everywhere. Fatou's lemma now yields 
\begin{equation*}
\liminf_{m\rightarrow \infty }\mathbb{P}(T(V_{m})\leq
t)=\liminf_{m\rightarrow \infty }\mathbb{E}(\mathbf{1}_{A_{m}})\geq \mathbb{E%
}(\liminf_{m\rightarrow \infty }\mathbf{1}_{A_{m}})\geq \mathbb{E}(\mathbf{1}%
_{A})=\mathbb{P}(\bar{T}(V)\leq t),
\end{equation*}%
where $\mathbb{E}$ denotes the expectation operator associated with $\mathbb{%
P}$.
\end{proof}

\begin{lemma}
\label{lem:quant}Let $F$ be the distribution function of a random variable
taking values in $\mathbb{R\cup \{\infty \}}$ and let $\delta \in (0,1)$. If 
$s\in \mathbb{\ R\cup \{\infty \}}$ satisfies $F(s)>\delta $, then $s$ is
larger than or equal to any $\delta $-quantile of $F$.
\end{lemma}

\begin{proof}
If $r\in \mathbb{R\cup \{\infty \}}$ satisfies $r>s$, it follows that $%
F(s)\leq F(r-)$ must hold. Together with the hypothesis, we conclude that $%
F(r-)>\delta $. But this shows that $r$ can not be a $\delta $-quantile of $%
F $, cf. Footnote \ref{qu}.
\end{proof}

\section{Appendix: Proofs for Section \protect\ref{unrestr_res}\label{App B}}

The facts collected in the subsequent remark will be used in the proofs
further below.

\begin{remark}
\label{F-type}(i) Suppose Assumption \ref{R_and_X} holds. Then the test
statistic $T_{Het}$ is a non-sphericity corrected F-type test statistic in
the sense of Section 5.4 in \cite{PP2016}. More precisely, $T_{Het}$ is of
the form (28) in \cite{PP2016} and Assumption 5 in the same reference is
satisfied with $\check{\beta}=\hat{\beta}$, $\check{\Omega}=\hat{\Omega}%
_{Het}$, and $N=\emptyset $. Furthermore, the set $N^{\ast }$ defined in
(27) of \cite{PP2016} satisfies $N^{\ast }=\mathsf{B}$. And also Assumptions
6 and 7 of \cite{PP2016} are satisfied. All these claims follow easily in
view of Lemma 4.1 in \cite{PP2016}, see also the proof of Theorem 4.2 in
that reference.

(ii) The test statistic $T_{uc}$ is also a non-sphericity corrected F-type
test statistic in the sense of Section 5.4 in \cite{PP2016} (terminology
being somewhat unfortunate here as no correction for the non-sphericity is
being attempted). More precisely, $T_{uc}$ is of the form (28) in \cite%
{PP2016} and Assumption 5 in the same reference is satisfied with $\check{%
\beta}=\hat{\beta}$, $\check{\Omega}=\hat{\sigma}^{2}R\left( X^{\prime
}X\right) ^{-1}R^{\prime }$, and $N=\emptyset $. Furthermore, the set $%
N^{\ast }$ defined in (27) of \cite{PP2016} satisfies $N^{\ast }=\limfunc{%
span}(X)$. And also Assumptions 6 and 7 of \cite{PP2016} are satisfied. All
these claims are evident (and obviously do not rely on Assumption \ref%
{R_and_X}).

(iii) We note that any non-sphericity corrected F-type test statistic (for
testing (\ref{testing problem})) in the sense of Section 5.4 in \cite{PP2016}%
, i.e., any test statistic $T$ of the form (28) in \cite{PP2016} that also
satisfies Assumption 5 in that reference, is invariant under the group $G(%
\mathfrak{M}_{0})$. Furthermore, the associated set $N^{\ast }$ defined in
(27) of \cite{PP2016} is even invariant under the larger group $G(\mathfrak{M%
})$. See Sections 5.1 and 5.4 of \cite{PP2016} as well as Lemma 5.16 in \cite%
{PP3} for more information.
\end{remark}

\textbf{Proof of Theorem \ref{Theo_Het_unrestr}:} We first prove Part (b)
and apply Theorem \ref{thm:main} in Appendix \ref{App A} with $%
T_{1}=T_{2}=T_{Het}^{\blacktriangle }$. Borel-measurability and $G(\mathfrak{%
M}_{0})$-invariance of $T_{1}$ and $T_{2}$ follow from the corresponding
properties of $T_{Het}$, see the discussion in Section \ref{test_statistic}
and, in particular, Remark \ref{rem:GM0}. Furthermore, we see that $F_{y}$
in Theorem \ref{thm:main} coincides with $F_{Het,y}^{\blacktriangle }$,
implying that we may set $f_{1-\alpha }(y)=f_{Het,1-\alpha }^{\blacktriangle
}(y)$. For indices $i$ such that $e_{i}(n)\notin \mathsf{B}$, we have also $%
\mu _{0}+e_{i}(n)\notin \mathsf{B}$ and thus continuity of $T_{1}$ at $\mu
_{0}+e_{i}(n)$ (as $T_{1}$ coincides with $T_{Het}$ on the complement of the
closed set $\mathsf{B}$); cf. Remarks \ref{rem:GM0}, \ref{F-type}, and Lemma
5.15 in \cite{PP2016}. In particular, $c_{i}=T_{1}(\mu
_{0}+e_{i}(n))=T_{Het}(\mu _{0}+e_{i}(n))$ follows for such indices $i$.
Also note that $y^{\ast }(\mu _{0}+e_{i}(n),\xi )\notin \mathsf{B}$ implies
that $y^{\ast }(\mu _{0}+e_{i}(n),\xi )$ is a continuity point of $T_{2}$.
This shows that $\vartheta _{1,Het}^{\blacktriangle }$ defined by 
\begin{equation*}
\vartheta _{1,Het}^{\blacktriangle }=\max_{\substack{ i=1,\ldots ,n, \\ %
e_{i}(n)\notin \mathsf{B}}}\Xi \left( \left\{ \xi :T_{Het}^{\blacktriangle
,\ast }(\mu _{0}+e_{i}(n),\xi )<T_{Het}(\mu _{0}+e_{i}(n)),y^{\ast }(\mu
_{0}+e_{i}(n),\xi )\notin \mathsf{B}\right\} \right) 
\end{equation*}%
is less than or equal to the maximum appearing on the r.h.s. of (\ref%
{eqn:newbound}) (where we set $\vartheta _{1,Het}^{\blacktriangle }=0$ if
the maximum operator is taken over an empty index set). It is also obvious
that $\vartheta _{1,Het}^{\blacktriangle }=\vartheta _{1,Het}$ holds, since $%
T_{Het}^{\blacktriangle ,\ast }(\mu _{0}+e_{i}(n),\xi )$ coincides with $%
T_{Het}^{\ast }(\mu _{0}+e_{i}(n),\xi )$ when $y^{\ast }(\mu
_{0}+e_{i}(n),\xi )\notin \mathsf{B}$. Next, consider an index $i$ such that 
$e_{i}(n)\in \limfunc{span}(X)$ and $R\hat{\beta}(e_{i}(n))\neq 0$ hold. We
show that then $c_{i}=\infty $: For this it suffices to show that $T_{1}(\mu
_{0}+v_{m})\rightarrow \infty $ for any sequence $v_{m}$ with $%
v_{m}\rightarrow e_{i}(n)$ for $m\rightarrow \infty $. Now, for all $m$ such
that $\mu _{0}+v_{m}\in \mathsf{B}$ holds we have $T_{1}(\mu
_{0}+v_{m})=T_{Het}^{\blacktriangle }(\mu _{0}+v_{m})=\infty $; and for all $%
m$ with $\mu _{0}+v_{m}\notin \mathsf{B}$ we obtain%
\begin{equation*}
T_{1}(\mu _{0}+v_{m})=T_{Het}(\mu _{0}+v_{m})\geq \left\Vert R\hat{\beta}%
(\mu _{0}+v_{m})-r\right\Vert ^{2}\lambda _{\max }^{-1}(\hat{\Omega}%
_{Het}(\mu _{0}+v_{m})),
\end{equation*}%
where $\lambda _{\max }(\cdot )$ denotes the largest eigenvalue of the
matrix indicated. Note that $R\hat{\beta}(\mu _{0}+v_{m})-r\rightarrow R\hat{%
\beta}(\mu _{0}+e_{i}(n))-r=R\hat{\beta}(e_{i}(n))\neq 0$ and that $\hat{%
\Omega}_{Het}(\mu _{0}+v_{m})\rightarrow \hat{\Omega}_{Het}(\mu
_{0}+e_{i}(n))=0$, the last equality following from $\mu _{0}+e_{i}(n)\in 
\limfunc{span}(X)$, which in turn is a consequence of $e_{i}(n)\in \limfunc{%
span}(X)$. Obviously, $T_{1}(\mu _{0}+v_{m})\rightarrow \infty $ for $%
m\rightarrow \infty $ now follows. Now define%
\begin{equation*}
\vartheta _{2,Het}^{\blacktriangle }=\max_{\substack{ i=1,\ldots ,n, \\ %
e_{i}(n)\in \limfunc{span}(X),R\hat{\beta}(e_{i}(n))\neq 0}}\Xi \left(
\left\{ \xi :T_{Het}^{\blacktriangle ,\ast }(\mu _{0}+e_{i}(n),\xi )<\infty
,y^{\ast }(\mu _{0}+e_{i}(n),\xi )\notin \mathsf{B}\right\} \right) ,
\end{equation*}%
where we set $\vartheta _{2,Het}^{\blacktriangle }=0$ if the maximum
operator is taken over an empty index set. Then $\vartheta
_{2,Het}^{\blacktriangle }$ is obviously less than or equal to the maximum
appearing on the r.h.s. of (\ref{eqn:newbound}), again using that $y^{\ast
}(\mu _{0}+e_{i}(n),\xi )\notin \mathsf{B}$ implies $y^{\ast }(\mu
_{0}+e_{i}(n),\xi )$ being a continuity point of $T_{2}$. Observe now that
the condition $T_{Het}^{\blacktriangle ,\ast }(\mu _{0}+e_{i}(n),\xi
)<\infty $ in the above set is always satisfied because of $y^{\ast }(\mu
_{0}+e_{i}(n),\xi )\notin \mathsf{B}$. Hence, $\vartheta
_{2,Het}^{\blacktriangle }=\vartheta _{2,Het}$ holds (of course, also in the
trivial case where the index set in question is empty). Consequently, any $%
\alpha $ satisfying $\alpha >\vartheta _{Het}=1-\max (\vartheta
_{1,Het},\vartheta _{2,Het})$ also satisfies (\ref{eqn:newbound}). An
application of Theorem \ref{thm:main} now yields (\ref%
{eqn:bootsize_het_triangle}) but with $T_{Het}$ replaced by $%
T_{Het}^{\blacktriangle }$. Since both the latter functions agree outside of 
$\mathsf{B}$, a $\lambda _{\mathbb{R}^{n}}$-null set, and since the measures 
$P_{\mu _{0},\sigma ^{2}\Sigma }$ are absolutely continuous with respect to $%
\lambda _{\mathbb{R}^{n}}$, also (\ref{eqn:bootsize_het_triangle}) as given
in the theorem follows. Independence of $\vartheta _{1,Het}$ and $\vartheta
_{2,Het}$ from the choice of $\mu _{0}\in \mathfrak{M}_{0}$ follows since $%
y^{\ast }(\mu _{0}+y,\xi )=y^{\ast }(\mu _{0}^{\prime }+y,\xi )-\mu
_{0}^{\prime }+\mu _{0}$ for every $\mu _{0}^{\prime }\in \mathfrak{M}_{0}$
as shown in the proof of Theorem \ref{thm:main} and since $T_{Het}$ and $%
\mathsf{B}$ are $G(\mathfrak{M}_{0})$-invariant, see Remark \ref{rem:GM0}.

Part (a) now follows from the already established Part (b), noting that we
may choose $f_{Het,1-\alpha }^{\blacktriangle }$ such that $f_{Het,1-\alpha
}^{\blacktriangle }(y)\geq f_{Het,1-\alpha }(y)$ holds for every $y\in 
\mathbb{R}^{n}$.\footnote{%
E.g., by choosing $f_{Het,1-\alpha }^{\blacktriangle }(y)$ as the largest $%
(1-\alpha )$-quantile of $F_{Het,y}^{\blacktriangle }$.} $\blacksquare $

\textbf{Proof of Theorem \ref{Theo_uc_unrestr}:} We first prove Part (b) and
apply Theorem \ref{thm:main} in Appendix \ref{App A} with $%
T_{1}=T_{2}=T_{uc}^{\blacktriangle }$. Borel-measurability and $G(\mathfrak{M%
}_{0})$-invariance of $T_{1}$ and $T_{2}$ follow from the corresponding
properties of $T_{uc}$, see the discussion in Section \ref{test_statistic}
and, in particular, Remark \ref{rem:GM0}. Furthermore, we see that $F_{y}$
in Theorem \ref{thm:main} coincides with $F_{uc,y}^{\blacktriangle }$,
implying that we may set $f_{1-\alpha }(y)=f_{uc,1-\alpha }^{\blacktriangle
}(y)$. For indices $i$ such that $e_{i}(n)\notin \limfunc{span}(X)$, we have
also $\mu _{0}+e_{i}(n)\notin \limfunc{span}(X)$ and thus continuity of $%
T_{1}$ at $\mu _{0}+e_{i}(n)$ (as $T_{1}$ coincides with $T_{uc}$ on the
complement of the closed set $\limfunc{span}(X)$); cf. Remarks \ref{rem:GM0}%
, \ref{F-type}, and Lemma 5.15 in \cite{PP2016}. In particular, $%
c_{i}=T_{1}(\mu _{0}+e_{i}(n))=T_{uc}(\mu _{0}+e_{i}(n))$ follows for such
indices $i$. Also note that $y^{\ast }(\mu _{0}+e_{i}(n),\xi )\notin 
\limfunc{span}(X)$ implies that $y^{\ast }(\mu _{0}+e_{i}(n),\xi )$ is a
continuity point of $T_{2}$. This shows that $\vartheta
_{1,uc}^{\blacktriangle }$ defined by 
\begin{equation*}
\vartheta _{1,uc}^{\blacktriangle }=\max_{\substack{ i=1,\ldots ,n, \\ %
e_{i}(n)\notin \limfunc{span}(X)}}\Xi \left( \left\{ \xi
:T_{uc}^{\blacktriangle ,\ast }(\mu _{0}+e_{i}(n),\xi )<T_{uc}(\mu
_{0}+e_{i}(n)),y^{\ast }(\mu _{0}+e_{i}(n),\xi )\notin \limfunc{span}%
(X)\right\} \right) 
\end{equation*}%
is less than or equal to the maximum appearing on the r.h.s. of (\ref%
{eqn:newbound}). It is also obvious that $\vartheta _{1,uc}^{\blacktriangle
}=\vartheta _{1,uc}$ holds, since $T_{uc}^{\blacktriangle ,\ast }(\mu
_{0}+e_{i}(n),\xi )$ coincides with $T_{uc}^{\ast }(\mu _{0}+e_{i}(n),\xi )$
when $y^{\ast }(\mu _{0}+e_{i}(n),\xi )\notin \limfunc{span}(X)$. Next,
consider an index $i$ such that $e_{i}(n)\in \limfunc{span}(X)$ and $R\hat{%
\beta}(e_{i}(n))\neq 0$ hold. We show that then $c_{i}=\infty $: For this it
suffices to show that $T_{1}(\mu _{0}+v_{m})\rightarrow \infty $ for any
sequence $v_{m}$ with $v_{m}\rightarrow e_{i}(n)$ for $m\rightarrow \infty $%
. Now, for all $m$ such that $\mu _{0}+v_{m}\in \limfunc{span}(X)$ holds we
have $T_{1}(\mu _{0}+v_{m})=T_{uc}^{\blacktriangle }(\mu _{0}+v_{m})=\infty $%
; and for all $m$ with $\mu _{0}+v_{m}\notin \limfunc{span}(X)$ we obtain%
\begin{equation*}
T_{1}(\mu _{0}+v_{m})=T_{uc}(\mu _{0}+v_{m})\geq \left\Vert R\hat{\beta}(\mu
_{0}+v_{m})-r\right\Vert ^{2}\lambda _{\max }^{-1}(R\left( X^{\prime
}X\right) ^{-1}R^{\prime })\hat{\sigma}^{-2}(\mu _{0}+v_{m}),
\end{equation*}%
where $\lambda _{\max }(\cdot )$ denotes the largest eigenvalue of the
matrix indicated. Note that $R\hat{\beta}(\mu _{0}+v_{m})-r\rightarrow R\hat{%
\beta}(\mu _{0}+e_{i}(n))-r=R\hat{\beta}(e_{i}(n))\neq 0$ and that $\hat{%
\sigma}^{2}(\mu _{0}+v_{m})\rightarrow \hat{\sigma}^{2}(\mu _{0}+e_{i}(n))=0$%
, the last equality following from $\mu _{0}+e_{i}(n)\in \limfunc{span}(X)$,
which in turn is a consequence of $e_{i}(n)\in \limfunc{span}(X)$.
Obviously, $T_{1}(\mu _{0}+v_{m})\rightarrow \infty $ for $m\rightarrow
\infty $ now follows. Now define%
\begin{equation*}
\vartheta _{2,uc}^{\blacktriangle }=\max_{\substack{ i=1,\ldots ,n, \\ %
e_{i}(n)\in \limfunc{span}(X),R\hat{\beta}(e_{i}(n))\neq 0}}\Xi \left(
\left\{ \xi :T_{uc}^{\blacktriangle ,\ast }(\mu _{0}+e_{i}(n),\xi )<\infty
,y^{\ast }(\mu _{0}+e_{i}(n),\xi )\notin \limfunc{span}(X)\right\} \right) ,
\end{equation*}%
where we set $\vartheta _{2,uc}^{\blacktriangle }=0$ if the maximum operator
is taken over an empty index set. Then $\vartheta _{2,uc}^{\blacktriangle }$
is obviously less than or equal to the maximum appearing on\ the r.h.s. of (%
\ref{eqn:newbound}), again using that $y^{\ast }(\mu _{0}+e_{i}(n),\xi
)\notin \limfunc{span}(X)$ implies $y^{\ast }(\mu _{0}+e_{i}(n),\xi )$ being
a continuity point of $T_{2}$. Observe now that the condition $%
T_{uc}^{\blacktriangle ,\ast }(\mu _{0}+e_{i}(n),\xi )<\infty $ in the above
set is always satisfied because of $y^{\ast }(\mu _{0}+e_{i}(n),\xi )\notin 
\limfunc{span}(X)$. Hence, $\vartheta _{2,uc}^{\blacktriangle }=\vartheta
_{2,uc}$ holds (of course, also in the trivial case where the index set in
question is empty). Consequently, any $\alpha $ satisfying $\alpha
>\vartheta _{uc}=1-\max (\vartheta _{1,uc},\vartheta _{2,uc})$ also
satisfies (\ref{eqn:newbound}). An application of Theorem \ref{thm:main} now
yields (\ref{eqn:bootsize_uncorr_triangle}) but with $T_{uc}$ replaced by $%
T_{uc}^{\blacktriangle }$. Since both the latter functions agree outside of $%
\limfunc{span}(X)$, a $\lambda _{\mathbb{R}^{n}}$-null set, and since the
measures $P_{\mu _{0},\sigma ^{2}\Sigma }$ are absolutely continuous with
respect to $\lambda _{\mathbb{R}^{n}}$, also (\ref%
{eqn:bootsize_uncorr_triangle}) as given follows. Independence of $\vartheta
_{1,uc}$ and $\vartheta _{2,uc}$ from the choice of $\mu _{0}\in \mathfrak{M}%
_{0}$ follows since $y^{\ast }(\mu _{0}+y,\xi )=y^{\ast }(\mu _{0}^{\prime
}+y,\xi )-\mu _{0}^{\prime }+\mu _{0}$ for every $\mu _{0}^{\prime }\in 
\mathfrak{M}_{0}$ as shown in the proof of Theorem \ref{thm:main} and since $%
T_{uc}$ and $\limfunc{span}(X)$ are $G(\mathfrak{M}_{0})$-invariant, see
Remark \ref{rem:GM0}.

Part (a) now follows from the already established Part (b), noting that we
may choose $f_{uc,1-\alpha }^{\blacktriangle }$ such that $f_{uc,1-\alpha
}^{\blacktriangle }(y)\geq f_{uc,1-\alpha }(y)$ holds for every $y\in 
\mathbb{R}^{n}$. $\blacksquare $

\begin{remark}
\label{rem:weaker}A weaker version of Theorem \ref{Theo_Het_unrestr}
(Theorem \ref{Theo_uc_unrestr}, respectively) is obtained if $\vartheta
_{Het}$ is replaced by $1-\vartheta _{1,Het}$ ($\vartheta _{uc}$ is replaced
by $1-\vartheta _{1,uc}$, respectively). These weaker versions not only have
simpler proofs in that one does not need to deal with the quantities $%
\vartheta _{2,Het}$ ($\vartheta _{2,uc}$, respectively), but can also be
derived from Theorem \ref{thm:main} more directly by setting $T_{1}=T_{Het}$
and $T_{2}=T_{Het}^{\blacktriangle }$ ($T_{1}=T_{uc}$ and $%
T_{2}=T_{uc}^{\blacktriangle }$, respectively), leading to a somewhat
simpler proof that directly establishes (\ref{eqn:bootsize_het_triangle}) ((%
\ref{eqn:bootsize_uncorr_triangle}), respectively) instead of establishing
these relations for $T_{Het}^{\blacktriangle }$ ($T_{uc}^{\blacktriangle }$,
respectively) first. Cf. the proofs of Theorems \ref{Theo_Het_restr_1} and %
\ref{Theo_uc_restr_1} further below, which have a similar structure.
\end{remark}

\textbf{Proof of Lemma \ref{VGK}:} Let $y\in \mathbb{R}^{n}$ and $\xi \in 
\mathbb{R}^{n}$ be arbitrary. Observe that $y^{\ast }(y,\xi )-y^{\maltese
}(y,\xi )=X\tilde{\beta}_{\mathfrak{M}_{0}}(y)-X\hat{\beta}(y)$. It is then
elementary to see that $R\hat{\beta}(y^{\ast }(y,\xi ))-r=R\hat{\beta}%
(y^{\maltese }(y,\xi ))-R\hat{\beta}(y)$ since $R\tilde{\beta}_{\mathfrak{M}%
_{0}}(y)=r$ certainly holds. Another immediate consequence is that $\hat{u}%
(y^{\ast }(y,\xi ))=\hat{u}(y^{\maltese }(y,\xi ))$ holds, which implies $%
\hat{\Omega}_{Het}(y^{\ast }(y,\xi ))=\hat{\Omega}_{Het}(y^{\maltese }(y,\xi
))$ as well as $\hat{\sigma}^{2}(y^{\ast }(y,\xi ))=\hat{\sigma}%
^{2}(y^{\maltese }(y,\xi ))$. From the first observation we see that $%
y^{\ast }(y,\xi )$ and $y^{\maltese }(y,\xi )$ differ only by an element of $%
\limfunc{span}(X)$. Hence, $y^{\ast }(y,\xi )\in \mathsf{B}$ if and only if $%
y^{\maltese }(y,\xi )\in \mathsf{B}$, since $\mathsf{B}+\limfunc{span}(X)=%
\mathsf{B}$ holds as noted in Lemma \ref{lem_B}. And similarly, $y^{\ast
}(y,\xi )\in \limfunc{span}(X)$ if and only if $y^{\maltese }(y,\xi )\in 
\limfunc{span}(X)$. This proves all the claims. $\blacksquare $

\section{Appendix: Proofs for Section \protect\ref{restrict_res}\label{App C}%
}

\textbf{Proof of Lemma \ref{lem:B_tilde}: }Observe that $\tilde{\Omega}%
_{Het}\left( y\right) =\tilde{B}\left( y\right) \limfunc{diag}\left( \tilde{d%
}_{1},\ldots ,\tilde{d}_{n}\right) \tilde{B}^{\prime }\left( y\right) $.
Given that $\tilde{d}_{i}>0$, this immediately establishes Parts (a) and (b)
of the lemma. We next prove Part (c). Let $s$ be as in Assumption \ref%
{R_and_X_tilde} and consider first the case where this assumption is
satisfied. If now $y\in \mathsf{\tilde{B}}$ it follows that $\tilde{u}%
_{l}(y)=0$ must hold at least for some $l\notin \left\{ i_{1},\ldots
i_{s}\right\} $ where $l$ may depend on $y$. But this means that $y$
satisfies $e_{l}^{\prime }(n)\Pi _{(\mathfrak{M}_{0}^{lin})^{\bot }}(y-\mu
_{0})=0$. Since $e_{l}^{\prime }(n)\Pi _{(\mathfrak{M}_{0}^{lin})^{\bot
}}\neq 0$ by construction of $l$, it follows that $\mathsf{\tilde{B}}$ is
contained in a finite union of proper affine subspaces, and hence is a $%
\lambda _{\mathbb{R}^{n}}$-null set. Next consider the case where Assumption %
\ref{R_and_X_tilde} is not satisfied. Observe that then $s>0$ must hold.
Note that $\tilde{u}_{i}(y)=0$ holds for all $y\in \mathbb{R}^{n}$ and all $%
i\in \left\{ i_{1},\ldots i_{s}\right\} $ by construction of $\left\{
i_{1},\ldots i_{s}\right\} $. But then for every $y\in \mathbb{R}^{n}$%
\begin{eqnarray*}
\limfunc{rank}(\tilde{B}\left( y\right) ) &=&\limfunc{rank}\left(
R(X^{\prime }X)^{-1}X^{\prime }\left( \lnot (i_{1},\ldots i_{s})\right)
A(y)\right) \\
&\leq &\limfunc{rank}\left( R(X^{\prime }X)^{-1}X^{\prime }\left( \lnot
(i_{1},\ldots i_{s})\right) \right) <q
\end{eqnarray*}%
is satisfied where $A(y)$ is obtained from $\limfunc{diag}\left( \tilde{u}%
_{1}(y),\ldots ,\tilde{u}_{n}(y)\right) $ by deleting rows and columns $i$
with $i\in \left\{ i_{1},\ldots i_{s}\right\} $. This completes the proof of
Part (c). To prove Part (d) recall that $\mathsf{\tilde{B}}$ is the set
where 
\begin{eqnarray*}
\tilde{B}(y) &=&R(X^{\prime }X)^{-1}X^{\prime }\limfunc{diag}\left(
e_{1}^{\prime }(n)\Pi _{(\mathfrak{M}_{0}^{lin})^{\bot }}(y-\mu _{0}),\ldots
,e_{n}^{\prime }(n)\Pi _{(\mathfrak{M}_{0}^{lin})^{\bot }}(y-\mu _{0})\right)
\\
&=&R(X^{\prime }X)^{-1}X^{\prime }\left[ e_{1}(n)e_{1}^{\prime }(n)\Pi _{(%
\mathfrak{M}_{0}^{lin})^{\bot }}(y-\mu _{0}),\ldots ,e_{n}(n)e_{n}^{\prime
}(n)\Pi _{(\mathfrak{M}_{0}^{lin})^{\bot }}(y-\mu _{0})\right]
\end{eqnarray*}%
has rank less than $q$. Define the set 
\begin{equation*}
D=\left\{ (j_{1},\ldots ,j_{s}):1\leq s\leq n,1\leq j_{1}<\ldots <j_{s}\leq
n,\limfunc{rank}\left( R(X^{\prime }X)^{-1}X^{\prime }\left[
e_{j_{1}}(n),\ldots ,e_{j_{s}}(n)\right] \right) <q\right\} ,
\end{equation*}%
which may be empty in case $q=1$. Consider first the case where $D$ is
nonempty: Since $R(X^{\prime }X)^{-1}X^{\prime }$ has rank $q$, it is then
easy to see that we have $y\in \mathsf{\tilde{B}}$ if and only if there
exists $(j_{1},\ldots ,j_{s})\in D$ such that $e_{j}^{\prime }(n)\Pi _{(%
\mathfrak{M}_{0}^{lin})^{\bot }}(y-\mu _{0})=0$ for $j\neq j_{i}$ for $%
i=1,\ldots ,s$. This shows, that $\mathsf{\tilde{B}}-\mu _{0}$ is a finite
union of (not necessarily distinct) linear subspaces. In case $D$ is empty, $%
\limfunc{rank}(R(X^{\prime }X)^{-1}X^{\prime })=q$ implies that $y\in 
\mathsf{\tilde{B}}$ if and only if $e_{j}^{\prime }(n)\Pi _{(\mathfrak{M}%
_{0}^{lin})^{\bot }}(y-\mu _{0})=0$ for all $1\leq j\leq n$, i.e., if and
only if $y\in \mathfrak{M}_{0}$. That is, in this case $\mathsf{\tilde{B}}%
-\mu _{0}=\mathfrak{M}_{0}^{lin}$, a linear subspace. That the linear
subspaces making up $\mathsf{\tilde{B}}-\mu _{0}$ are proper, follows since
otherwise $\mathsf{\tilde{B}}-\mu _{0}$, and thus $\mathsf{\tilde{B}}$,
would be all of $\mathbb{R}^{n}$, which is impossible under Assumption \ref%
{R_and_X_tilde} as shown in Part (c). To prove the second claim, observe
that in case $q=1$ the condition that $\limfunc{rank}(\tilde{B}(y))$ is less
than $q$ is equivalent to $\tilde{B}(y)=0$. Since the expressions $%
R(X^{\prime }X)^{-1}X^{\prime }e_{j}(n)$ are now scalar, we may thus write
the condition $\tilde{B}(y)=0$ equivalently as%
\begin{equation*}
\left[ R(X^{\prime }X)^{-1}X^{\prime }e_{1}(n)e_{1}(n),\ldots ,R(X^{\prime
}X)^{-1}X^{\prime }e_{n}(n)e_{n}(n)\right] \Pi _{(\mathfrak{M}%
_{0}^{lin})^{\bot }}(y-\mu _{0})=0.
\end{equation*}%
But this shows that $\mathsf{B}-\mu _{0}$ is a linear subspace, namely the
kernel of the matrix appearing on the l.h.s. of the preceding display. It
remains to prove Part (e). Closedness of $\mathsf{\tilde{B}}$ follows from
Parts (c) and (d), and the remaining claims are trivial (note that $\tilde{B}%
(y)$ only depends on $\tilde{u}(y)$, which obviously is $G(\mathfrak{M}_{0})$%
-invariant). $\blacksquare $

\textbf{Proof of Theorem \ref{Theo_Het_restr_1}:} We first prove Part (b)
and apply Theorem \ref{thm:main} in Appendix \ref{App A} with $T_{1}=\tilde{T%
}_{Het}$ and $T_{2}=\tilde{T}_{Het}^{\blacktriangle }$. Borel-measurability
and $G(\mathfrak{M}_{0})$-invariance of $T_{1}$ and $T_{2}$ follow from the
corresponding properties of $\tilde{T}_{Het}$, see Remark \ref{rem:tildeGM0}%
. Furthermore, we see that $F_{y}$ in Theorem \ref{thm:main} coincides with $%
\tilde{F}_{Het,y}^{\blacktriangle }$, implying that we may set $f_{1-\alpha
}(y)=\tilde{f}_{Het,1-\alpha }^{\blacktriangle }(y)$. For indices $i$ such
that $\mu _{0}+e_{i}(n)\notin \mathsf{\tilde{B}}$, the function $T_{1}$ is
continuous at $\mu _{0}+e_{i}(n)$. In particular, $c_{i}=T_{1}(\mu
_{0}+e_{i}(n))=\tilde{T}_{Het}(\mu _{0}+e_{i}(n))$ follows for such indices $%
i$. Also note that $y^{\ast }(\mu _{0}+e_{i}(n),\xi )\notin \mathsf{\tilde{B}%
}$ implies that $y^{\ast }(\mu _{0}+e_{i}(n),\xi )$ is a continuity point of 
$T_{2}$. This shows that $\tilde{\vartheta}_{Het}^{\blacktriangle }$ defined
by 
\begin{equation*}
\tilde{\vartheta}_{Het}^{\blacktriangle }=1-\max_{\substack{ i=1,\ldots ,n, 
\\ \mu _{0}+e_{i}(n)\notin \mathsf{\tilde{B}}}}\Xi \left( \left\{ \xi :%
\tilde{T}_{Het}^{\blacktriangle ,\ast }(\mu _{0}+e_{i}(n),\xi )<\tilde{T}%
_{Het}(\mu _{0}+e_{i}(n)),y^{\ast }(\mu _{0}+e_{i}(n),\xi )\notin \mathsf{%
\tilde{B}}\right\} \right)
\end{equation*}%
is not less than the r.h.s. of (\ref{eqn:newbound}) (where we set $\tilde{%
\vartheta}_{Het}^{\blacktriangle }=1$ if the maximum operator is taken over
an empty index set). It is also obvious that $\tilde{\vartheta}%
_{Het}^{\blacktriangle }=\tilde{\vartheta}_{Het}$ holds, since $\tilde{T}%
_{Het}^{\blacktriangle ,\ast }(\mu _{0}+e_{i}(n),\xi )$ coincides with $%
\tilde{T}_{Het}^{\ast }(\mu _{0}+e_{i}(n),\xi )$ when $y^{\ast }(\mu
_{0}+e_{i}(n),\xi )\notin \mathsf{\tilde{B}}$. Consequently, any $\alpha $
satisfying $\alpha >\tilde{\vartheta}_{Het}$ also satisfies (\ref%
{eqn:newbound}). An application of Theorem \ref{thm:main} now yields (\ref%
{eqn:bootsize_het_tilde_triangle}). Independence of $\tilde{\vartheta}_{Het}$
from the choice of $\mu _{0}\in \mathfrak{M}_{0}$ follows since $y^{\ast
}(\mu _{0}+y,\xi )=y^{\ast }(\mu _{0}^{\prime }+y,\xi )-\mu _{0}^{\prime
}+\mu _{0}$ for every $\mu _{0}^{\prime }\in \mathfrak{M}_{0}$ as shown in
the proof of Theorem \ref{thm:main} and since $\tilde{T}_{Het}$ and $\mathsf{%
\tilde{B}}$ are $G(\mathfrak{M}_{0})$-invariant, see Remark \ref%
{rem:tildeGM0}.

Part (a) now follows from the already established Part (b), noting that we
may choose $\tilde{f}_{Het,1-\alpha }^{\blacktriangle }$ such that $\tilde{f}%
_{Het,1-\alpha }^{\blacktriangle }(y)\geq \tilde{f}_{Het,1-\alpha }(y)$
holds for every $y\in \mathbb{R}^{n}$. $\blacksquare $

\textbf{Proof of Theorem \ref{Theo_uc_restr_1}:}\footnote{%
In view of the relationship between $\tilde{T}_{uc}$ and $T_{uc}$ discussed
in Section 5.1.1 in \cite{PP5HC}, one could attempt to use this relationship
to derive Theorem \ref{Theo_uc_restr_1} from Theorem \ref{Theo_uc_unrestr}.
However, it is not obvious how this can actually be executed since the
relationship mentioned only holds outside of a Lebesque null-set, but $%
y^{\ast }(y,\xi )$ typically follows a discrete distribution.} We first
prove Part (b) and apply Theorem \ref{thm:main} in Appendix \ref{App A} with 
$T_{1}=\tilde{T}_{uc}$ and $T_{2}=\tilde{T}_{uc}^{\blacktriangle }$.
Borel-measurability and $G(\mathfrak{M}_{0})$-invariance of $T_{1}$ and $%
T_{2}$ follow from the corresponding properties of $\tilde{T}_{uc}$, see
Remark \ref{rem:tildeGM0}. Furthermore, we see that $F_{y}$ in Theorem \ref%
{thm:main} coincides with $\tilde{F}_{uc,y}^{\blacktriangle }$, implying
that we may set $f_{1-\alpha }(y)=\tilde{f}_{uc,1-\alpha }^{\blacktriangle
}(y)$. For indices $i$ such that $\mu _{0}+e_{i}(n)\notin \mathfrak{M}_{0}$,
the function $T_{1}$ is continuous at $\mu _{0}+e_{i}(n)$. In particular, $%
c_{i}=T_{1}(\mu _{0}+e_{i}(n))=\tilde{T}_{uc}(\mu _{0}+e_{i}(n))$ follows
for such indices $i$. Also note that $y^{\ast }(\mu _{0}+e_{i}(n),\xi
)\notin \mathfrak{M}_{0}$ implies that $y^{\ast }(\mu _{0}+e_{i}(n),\xi )$
is a continuity point of $T_{2}$. This shows that $\tilde{\vartheta}%
_{uc}^{\blacktriangle }$ defined by 
\begin{equation*}
\tilde{\vartheta}_{uc}^{\blacktriangle }=1-\max_{\substack{ i=1,\ldots ,n, 
\\ \mu _{0}+e_{i}(n)\notin \mathfrak{M}_{0}}}\Xi \left( \left\{ \xi :\tilde{T%
}_{uc}^{\blacktriangle ,\ast }(\mu _{0}+e_{i}(n),\xi )<\tilde{T}_{uc}(\mu
_{0}+e_{i}(n)),y^{\ast }(\mu _{0}+e_{i}(n),\xi )\notin \mathfrak{M}%
_{0}\right\} \right)
\end{equation*}%
is not less than the r.h.s. of (\ref{eqn:newbound}). It is also obvious that 
$\tilde{\vartheta}_{uc}^{\blacktriangle }=\tilde{\vartheta}_{uc}$ holds,
since $\tilde{T}_{uc}^{\blacktriangle ,\ast }(\mu _{0}+e_{i}(n),\xi )$
coincides with $\tilde{T}_{uc}^{\ast }(\mu _{0}+e_{i}(n),\xi )$ when $%
y^{\ast }(\mu _{0}+e_{i}(n),\xi )\notin \mathfrak{M}_{0}$. Consequently, any 
$\alpha $ satisfying $\alpha >\tilde{\vartheta}_{uc}$ also satisfies (\ref%
{eqn:newbound}). An application of Theorem \ref{thm:main} now yields (\ref%
{eqn:bootsize_uncorr_tilde_triangle}). Independence of $\tilde{\vartheta}%
_{uc}$ from the choice of $\mu _{0}\in \mathfrak{M}_{0}$ follows since $%
y^{\ast }(\mu _{0}+y,\xi )=y^{\ast }(\mu _{0}^{\prime }+y,\xi )-\mu
_{0}^{\prime }+\mu _{0}$ for every $\mu _{0}^{\prime }\in \mathfrak{M}_{0}$
as shown in the proof of Theorem \ref{thm:main} and since $\tilde{T}_{uc}$
and $\mathfrak{M}_{0}$ are $G(\mathfrak{M}_{0})$-invariant, see Remark \ref%
{rem:tildeGM0}.

Part (a) now follows from the already established Part (b), noting that we
may choose $\tilde{f}_{uc,1-\alpha }^{\blacktriangle }$ such that $\tilde{f}%
_{uc,1-\alpha }^{\blacktriangle }(y)\geq \tilde{f}_{uc,1-\alpha }(y)$ holds
for every $y\in \mathbb{R}^{n}$. $\blacksquare $

\begin{lemma}
\label{lem:T_tilde_invariance}Let $\mathcal{A}$ be an affine subspace of $%
\mathbb{R}^{n}$ satisfying $\mathfrak{M}_{0}\subseteq \mathcal{A}\subseteq 
\limfunc{span}(X)$, and let $y^{\maltese }(y,\xi )$ be as defined in (\ref%
{boot_2}).

(a) For every $\gamma \in \mathbb{R}$, every $\mu _{0}^{\prime }\in 
\mathfrak{M}_{0}$, every $\mu _{0}^{\prime \prime }\in \mathfrak{M}_{0}$,
every $y\in \mathbb{R}$, and every $\xi \in \mathbb{R}^{n}$ we have%
\begin{equation*}
y^{\maltese }(\gamma (y-\mu _{0}^{\prime })+\mu _{0}^{\prime \prime },\xi
)=\gamma \left( y^{\maltese }(y,\xi )-\mu _{0}^{\prime }\right) +\mu
_{0}^{\prime \prime },
\end{equation*}%
\begin{equation*}
R\hat{\beta}(y^{\maltese }(\gamma (y-\mu _{0}^{\prime })+\mu _{0}^{\prime
\prime },\xi ))-R\hat{\beta}(\gamma (y-\mu _{0}^{\prime })+\mu _{0}^{\prime
\prime })=\gamma \left[ R\hat{\beta}(y^{\maltese }(y,\xi ))-R\hat{\beta}(y)%
\right] ,
\end{equation*}%
\begin{equation*}
\tilde{u}(y^{\maltese }(\gamma (y-\mu _{0}^{\prime })+\mu _{0}^{\prime
\prime },\xi ))=\gamma \tilde{u}(y^{\maltese }(y,\xi )),
\end{equation*}%
\begin{equation*}
\tilde{\Omega}_{Het}\left( y^{\maltese }(\gamma (y-\mu _{0}^{\prime })+\mu
_{0}^{\prime \prime },\xi )\right) =\gamma ^{2}\tilde{\Omega}_{Het}\left(
y^{\maltese }(y,\xi )\right)
\end{equation*}%
where $\tilde{u}(y)=y-X\tilde{\beta}_{\mathfrak{M}_{0}}(y)$ denotes the
vector of restricted residuals corresponding to the restricted estimator $%
\tilde{\beta}_{\mathfrak{M}_{0}}$.

(b) The set $\{y\in \mathbb{R}^{n}:\det \tilde{\Omega}_{Het}\left(
y^{\maltese }(y,\xi )\right) =0\}$ as well as $\tilde{T}_{Het}^{\maltese
}(\cdot ,\xi )$ are $G(\mathfrak{M}_{0})$-invariant for every $\xi \in 
\mathbb{R}^{n}$.

(c) The set $\{y\in \mathbb{R}^{n}:y^{\maltese }(y,\xi )\in \mathfrak{M}%
_{0}\}$ as well as $\tilde{T}_{uc}^{\maltese }(\cdot ,\xi )$ are $G(%
\mathfrak{M}_{0})$-invariant for every $\xi \in \mathbb{R}^{n}$.
\end{lemma}

\begin{proof}
(a) Using elementary properties of the least squares estimator $\hat{\beta}$%
, using that $y-X\tilde{\beta}_{\mathcal{A}}(y)=\Pi _{(\mathcal{A-}\mu
_{0})^{\bot }}(y-\mu _{0})$ as noted in the proof of Theorem \ref{thm:main}
in Appendix \ref{App A}, and noting that $\Pi _{(\mathcal{A-}\mu _{0})^{\bot
}}(\mu _{0}^{\prime \prime }-\mu _{0})=0=\Pi _{(\mathcal{A-}\mu _{0})^{\bot
}}(\mu _{0}-\mu _{0}^{\prime })$ (since $\mu _{0}^{\prime \prime }-\mu _{0}$
as well as $\mu _{0}-\mu _{0}^{\prime }$ belong to $\mathfrak{M}_{0}^{lin}$
and since $(\mathcal{A-}\mu _{0})^{\bot }\subseteq (\mathfrak{M}%
_{0}^{lin})^{\bot })$ we get%
\begin{eqnarray*}
y^{\maltese }(\gamma (y-\mu _{0}^{\prime })+\mu _{0}^{\prime \prime },\xi )
&=&X\hat{\beta}(\gamma (y-\mu _{0}^{\prime })+\mu _{0}^{\prime \prime })+%
\limfunc{diag}(\xi )\Pi _{(\mathcal{A-}\mu _{0})^{\bot }}(\gamma (y-\mu
_{0}^{\prime })+\mu _{0}^{\prime \prime }-\mu _{0}) \\
&=&\gamma (X\hat{\beta}(y)-\mu _{0}^{\prime })+\mu _{0}^{\prime \prime
}+\gamma \limfunc{diag}(\xi )\Pi _{(\mathcal{A-}\mu _{0})^{\bot }}(y-\mu
_{0}) \\
&=&\gamma \left( y^{\maltese }(y,\xi )-\mu _{0}^{\prime }\right) +\mu
_{0}^{\prime \prime }.
\end{eqnarray*}%
The second displayed equation is then an immediate consequence. Using what
has already been established%
\begin{eqnarray*}
\tilde{u}(y^{\maltese }(\gamma (y-\mu _{0}^{\prime })+\mu _{0}^{\prime
\prime },\xi )) &=&\tilde{u}(\gamma \left( y^{\maltese }(y,\xi )-\mu
_{0}^{\prime }\right) +\mu _{0}^{\prime \prime }) \\
&=&\Pi _{(\mathfrak{M}_{0}^{lin})^{\bot }}(\gamma \left( y^{\maltese }(y,\xi
)-\mu _{0}^{\prime }\right) +\mu _{0}^{\prime \prime }-\mu _{0})=\gamma \Pi
_{(\mathfrak{M}_{0}^{lin})^{\bot }}\left( y^{\maltese }(y,\xi )-\mu
_{0}^{\prime }\right) \\
&=&\gamma \Pi _{(\mathfrak{M}_{0}^{lin})^{\bot }}\left( y^{\maltese }(y,\xi
)-\mu _{0}\right) =\gamma \tilde{u}(y^{\maltese }(y,\xi )),
\end{eqnarray*}%
which then also immediately gives the fourth displayed equation.

(b)\&(c). Follows from (a).
\end{proof}

\textbf{Proof of Theorem \ref{Theo_Het_restr_2}:} We prove Part (b) first.
By $G(\mathfrak{M}_{0})$-invariance of $\tilde{T}_{Het}^{\maltese }(\cdot
,\xi )$ and of the set $\{y\in \mathbb{R}^{n}:y^{\maltese }(y,\xi )\in 
\mathsf{\tilde{B}}\}$ (for every $\xi $) established in Lemma \ref%
{lem:T_tilde_invariance} we can conclude that for every $z\in \mathbb{R}^{n}$
the expression $\tilde{T}_{Het}^{\maltese }(\mu _{0}+\gamma z,\xi )$ depends
neither on the choice of $\mu _{0}\in \mathfrak{M}_{0}$ nor on the value of $%
\gamma \in \mathbb{R}\backslash \{0\}$, cf. Footnote \ref{FN:equi}. By $G(%
\mathfrak{M}_{0})$-invariance (see Remark \ref{rem:tildeGM0}), the same is
true also for $\tilde{T}_{Het}(\mu _{0}+\gamma z)$. By Lemma \ref%
{lem:T_tilde_invariance} and by $G(\mathfrak{M}_{0})$-invariance of $\mathsf{%
\tilde{B}}$, the truth-value of the statement $y^{\maltese }(\mu _{0}+\gamma
z,\xi )\notin \mathsf{\tilde{B}}$ depends neither on the choice of $\mu
_{0}\in \mathfrak{M}_{0}$ nor on the value of $\gamma \in \mathbb{R}%
\backslash \{0\}$. A similar comment applies to the statement $\mu
_{0}+e_{i}(n)\notin \mathsf{\tilde{B}}$. This shows, in particular, that $%
\tilde{\theta}_{Het}$ does not depend on the choice of $\mu _{0}\in 
\mathfrak{M}_{0}$. Next, observe that for every $y\in \mathbb{R}^{n}$ and
every $t\in \mathbb{R}$ we have $\tilde{H}_{Het,y}^{\blacktriangle }(t)=\Xi
(\{\xi :\tilde{T}_{Het}\left( y^{\maltese }(y,\xi )\right) \leq
t,y^{\maltese }(y,\xi )\notin \mathsf{\tilde{B}}\})$. In particular, the
condition $\alpha >\tilde{\theta}_{Het}$ is equivalent to 
\begin{equation*}
\max_{\substack{ i=1,\ldots ,n,  \\ \mu _{0}+e_{i}(n)\notin \mathsf{\tilde{B}%
} }}\tilde{H}_{Het,\mu _{0}+e_{i}(n)}^{\blacktriangle }(\tilde{T}_{Het}(\mu
_{0}+e_{i}(n))-)>1-\alpha ,
\end{equation*}%
with the convention that the left-hand side is zero if the maximum operator
extends over an empty index set, in which case there is then nothing to
prove. Otherwise, let from now on $i$ be an index that realizes the maximum
in the previous display. In the case where $\tilde{H}_{Het,\mu
_{0}+e_{i}(n)}^{\blacktriangle }(\tilde{T}_{Het}(\mu _{0}+e_{i}(n))-)=0$,
there is again nothing to prove and we are done. Hence, it remains to
consider the case where $\tilde{H}_{Het,\mu _{0}+e_{i}(n)}^{\blacktriangle }(%
\tilde{T}_{Het}(\mu _{0}+e_{i}(n))-)>0$. In this case then, let $\alpha \in
(0,1)$ be such that $\tilde{H}_{Het,\mu _{0}+e_{i}(n)}^{\blacktriangle }(%
\tilde{T}_{Het}(\mu _{0}+e_{i}(n))-)>1-\alpha $ (where $\mu _{0}\in 
\mathfrak{M}_{0}$ can be chosen arbitrarily). From this inequality we can
conclude existence of a real number $x_{i}$ smaller than $\tilde{T}%
_{Het}(\mu _{0}+e_{i}(n))$ such that $\tilde{H}_{Het,\mu
_{0}+e_{i}(n)}^{\blacktriangle }(x_{i})>1-\alpha $ holds and such that $%
x_{i} $ is a continuity point of $\tilde{H}_{Het,\mu
_{0}+e_{i}(n)}^{\blacktriangle }$. Since $\mu _{0}+e_{i}(n)\notin \mathsf{%
\tilde{B}}$, it is obvious that $\tilde{T}_{Het}$ is continuous at $\mu
_{0}+e_{i}(n)$. In view of this, there exists a $\delta >0$ such that every $%
z\in B(e_{i}(n),\delta )$ satisfies $\tilde{T}_{Het}(\mu _{0}+z)>x_{i}$ (and
the same is true if we replace $\delta $ by a smaller positive number). We
claim that for every sequence $z_{m}\rightarrow e_{i}(n)$ ($z_{m}\in \mathbb{%
R}^{n}$) we have%
\begin{equation}
\liminf_{m\rightarrow \infty }\tilde{H}_{Het,\mu _{0}+z_{m}}^{\blacktriangle
}(x_{i})\geq \tilde{H}_{Het,\mu _{0}+e_{i}(n)}^{\blacktriangle }(x_{i}).
\label{eqn:ineq}
\end{equation}%
Define $V_{m}=V_{m}(\xi )=(\mu _{0}+z_{m},y^{\maltese }(\mu _{0}+z_{m},\xi
)) $ and $V=V(\xi )=(\mu _{0}+e_{i}(n),y^{\maltese }(\mu _{0}+e_{i}(n),\xi
)) $, which can be viewed as random vectors defined on $\mathbb{R}^{n}$
(equipped with the Borel $\sigma $-field) and where the probability measure
is given by $\Xi $. Note that $V_{m}$ converges to $V$ everywhere as $%
m\rightarrow \infty $ (since $y^{\maltese }(y,\xi )$ is continuous w.r.t. $y$%
). Define $S:\mathbb{R}^{n}\times \mathbb{R}^{n}\rightarrow \mathbb{R\cup
\{\infty \}}$ via%
\begin{equation*}
S(z,x)=\left\{ 
\begin{array}{cc}
(R\hat{\beta}\left( x\right) -R\hat{\beta}\left( z\right) )^{\prime }\tilde{%
\Omega}_{Het}^{-1}\left( x\right) (R\hat{\beta}\left( x)\right) -R\hat{\beta}%
\left( z\right) ) & \text{if }x\notin \mathsf{\tilde{B}} \\ 
\infty & \text{if }x\in \mathsf{\tilde{B}}%
\end{array}%
\right.
\end{equation*}%
and note that $\tilde{H}_{Het,\mu _{0}+z_{m}}^{\blacktriangle }(t)=\Xi
(S(V_{m})\leq t)$ and $\tilde{H}_{Het,\mu _{0}+e_{i}(n)}^{\blacktriangle
}(t)=\Xi (S(V)\leq t)$ holds for $t\in \mathbb{R}$ and thus also for $%
t=x_{i} $ (recall that $x_{i}$ is a real number). The statement in (\ref%
{eqn:ineq}) now follows from Lemma \ref{lem:convdisc} in Appendix \ref{App A}%
, recalling that we have chosen $x_{i}$ as a continuity point of $\tilde{H}%
_{Het,\mu _{0}+e_{i}(n)}^{\blacktriangle }$, which implies $\Xi
(S(V)=x_{i})=0$, and noting that $\bar{S}=S$ here holds. Summarizing, we
hence arrive, replacing $\delta $ by another element of $(0,\delta )$ if
necessary, at%
\begin{equation}
\tilde{T}_{Het}(\mu _{0}+z)>x_{i}\text{ and }\tilde{H}_{Het,\mu
_{0}+z}^{\blacktriangle }(x_{i})>1-\alpha \text{ for every }z\in
B(e_{i}(n),\delta ).  \label{eqn:bootdelta_het_tilde_triangle_malt2}
\end{equation}%
Now, for every $z\in \mathbb{R}^{n}$ and every $\gamma \neq 0$ we have%
\begin{eqnarray*}
\tilde{H}_{Het,\mu _{0}+\gamma z}^{\blacktriangle }(x_{i}) &=&\Xi \left(
\left\{ \xi :\tilde{T}_{Het}^{\blacktriangle ,\maltese }(\mu _{0}+\gamma
z,\xi )\leq x_{i}\right\} \right) \\
&=&\Xi \left( \left\{ \xi :\tilde{T}_{Het}^{\blacktriangle ,\maltese }(\mu
_{0}+z,\xi )\leq x_{i}\right\} \right) =\tilde{H}_{Het,\mu
_{0}+z}^{\blacktriangle }(x_{i})
\end{eqnarray*}%
by what has been shown at the very beginning of the proof. Using also the
observation for $\tilde{T}_{Het}(\mu _{0}+\gamma z)$ made at the beginning
of the proof, the preceding display and (\ref%
{eqn:bootdelta_het_tilde_triangle_malt2}) thus give%
\begin{equation*}
\tilde{T}_{Het}(\mu _{0}+y)>x_{i}\text{ and }\tilde{H}_{Het,\mu
_{0}+y}^{\blacktriangle }(x_{i})>1-\alpha \text{ for every }y\in \{\gamma
z:\gamma \neq 0,z\in B(e_{i}(n),\delta )\}=:\Delta ,
\end{equation*}%
or equivalently%
\begin{equation}
\tilde{T}_{Het}(y)>x_{i}\text{ and }\tilde{H}_{Het,y}^{\blacktriangle
}(x_{i})>1-\alpha \text{ for every }y\in \mu _{0}+\Delta .
\label{eqn:bootdelta_het_tilde_triangle_malt3}
\end{equation}%
By Lemma \ref{lem:quant} we see that $\tilde{H}_{Het,y}^{\blacktriangle
}(x_{i})>1-\alpha $ implies $x_{i}\geq \tilde{h}_{Het,1-\alpha
}^{\blacktriangle }(y)$. Hence,%
\begin{equation*}
\left\{ y\in \mathbb{R}^{n}:\tilde{T}_{Het}(y)>x_{i},\text{ }\tilde{H}%
_{Het,y}^{\blacktriangle }(x_{i})>1-\alpha \right\} \subseteq \left\{ y\in 
\mathbb{R}^{n}:\tilde{T}_{Het}(y)>\tilde{h}_{Het,1-\alpha }^{\blacktriangle
}(y)\right\} .
\end{equation*}%
This, together with (\ref{eqn:bootdelta_het_tilde_triangle_malt3}), gives%
\begin{equation*}
\left\{ y\in \mathbb{R}^{n}:\tilde{T}_{Het}(y)>\tilde{h}_{Het,1-\alpha
}^{\blacktriangle }(y)\right\} \supseteq \mu _{0}+\Delta \supseteq \mu _{0}+%
\limfunc{int}(\Delta )\supseteq \mu _{0}+(\limfunc{span}(e_{i}(n))\backslash
\{0\}),
\end{equation*}%
where $\limfunc{int}(\Delta )$ denotes the interior of $\Delta $. The proof
is now completed by following the steps after (\ref{eqn:bootincl_2}) in the
proof of Theorem \ref{thm:main} in Appendix \ref{App A}.

Part (a) now follows from the already established Part (b), noting that we
may choose $\tilde{h}_{Het,1-\alpha }^{\blacktriangle }$ such that $\tilde{h}%
_{Het,1-\alpha }^{\blacktriangle }(y)\geq \tilde{h}_{Het,1-\alpha }(y)$
holds for every $y\in \mathbb{R}^{n}$. $\blacksquare $

\textbf{Proof of Theorem \ref{Theo_uc_restr_2}:} We prove Part (b) first. By 
$G(\mathfrak{M}_{0})$-invariance of $\tilde{T}_{uc}^{\maltese }(\cdot ,\xi )$
and of the set $\{y\in \mathbb{R}^{n}:y^{\maltese }(y,\xi )\in \mathfrak{M}%
_{0}\}$ (for every $\xi $) established in Lemma \ref{lem:T_tilde_invariance}
we can conclude that for every $z\in \mathbb{R}^{n}$ the expression $\tilde{T%
}_{uc}^{\maltese }(\mu _{0}+\gamma z,\xi )$ depends neither on the choice of 
$\mu _{0}\in \mathfrak{M}_{0}$ nor on the value of $\gamma \in \mathbb{R}%
\backslash \{0\}$, cf. Footnote \ref{FN:equi}. By $G(\mathfrak{M}_{0})$%
-invariance (see Remark \ref{rem:tildeGM0}), the same is true also for $%
\tilde{T}_{uc}(\mu _{0}+\gamma z)$. By Lemma \ref{lem:T_tilde_invariance},
the truth-value of the statement $y^{\maltese }(\mu _{0}+\gamma z,\xi
)\notin \mathfrak{M}_{0}$ depends neither on the choice of $\mu _{0}\in 
\mathfrak{M}_{0}$ nor on the value of $\gamma \in \mathbb{R}\backslash \{0\}$%
. A similar comment applies to the statement $\mu _{0}+e_{i}(n)\notin 
\mathfrak{M}_{0}$. This shows, in particular, that $\tilde{\theta}_{uc}$
does not depend on the choice of $\mu _{0}\in \mathfrak{M}_{0}$. Next,
observe that for every $y\in \mathbb{R}^{n}$ and every $t\in \mathbb{R}$ we
have $\tilde{H}_{uc,y}^{\blacktriangle }(t)=\Xi (\{\xi :\tilde{T}_{uc}\left(
y^{\maltese }(y,\xi )\right) \leq t,y^{\maltese }(y,\xi )\notin \mathfrak{M}%
_{0}\})$. In particular, the condition $\alpha >\tilde{\theta}_{uc}$ is
equivalent to 
\begin{equation*}
\max_{\substack{ i=1,\ldots ,n,  \\ \mu _{0}+e_{i}(n)\notin \mathfrak{M}_{0} 
}}\tilde{H}_{uc,\mu _{0}+e_{i}(n)}^{\blacktriangle }(\tilde{T}_{uc}(\mu
_{0}+e_{i}(n))-)>1-\alpha .
\end{equation*}%
From now on let $i$ be an index that realizes the maximum in the previous
display (note that the index set is not empty since $k-q\leq k<n$ by
assumption). In the case where $\tilde{H}_{uc,\mu
_{0}+e_{i}(n)}^{\blacktriangle }(\tilde{T}_{uc}(\mu _{0}+e_{i}(n))-)=0$,
there is nothing to prove and we are done. Hence, it remains to consider the
case where $\tilde{H}_{uc,\mu _{0}+e_{i}(n)}^{\blacktriangle }(\tilde{T}%
_{uc}(\mu _{0}+e_{i}(n))-)>0$. In this case then, let $\alpha \in (0,1)$ be
such that $\tilde{H}_{uc,\mu _{0}+e_{i}(n)}^{\blacktriangle }(\tilde{T}%
_{uc}(\mu _{0}+e_{i}(n))-)>1-\alpha $ (where $\mu _{0}\in \mathfrak{M}_{0}$
can be chosen arbitrarily). From this inequality we can conclude existence
of a real number $x_{i}$ smaller than $\tilde{T}_{uc}(\mu _{0}+e_{i}(n))$
such that $\tilde{H}_{uc,\mu _{0}+e_{i}(n)}^{\blacktriangle
}(x_{i})>1-\alpha $ holds and such that $x_{i}$ is a continuity point of $%
\tilde{H}_{uc,\mu _{0}+e_{i}(n)}^{\blacktriangle }$. Since $\mu
_{0}+e_{i}(n)\notin \mathfrak{M}_{0}$, it is obvious that $\tilde{T}_{uc}$
is continuous at $\mu _{0}+e_{i}(n)$. In view of this, there exists a $%
\delta >0$ such that every $z\in B(e_{i}(n),\delta )$ satisfies $\tilde{T}%
_{uc}(\mu _{0}+z)>x_{i}$ (and the same is true if we replace $\delta $ by a
smaller positive number). We claim that for every sequence $z_{m}\rightarrow
e_{i}(n)$ ($z_{m}\in \mathbb{R}^{n}$) we have%
\begin{equation}
\liminf_{m\rightarrow \infty }\tilde{H}_{uc,\mu _{0}+z_{m}}^{\blacktriangle
}(x_{i})\geq \tilde{H}_{uc,\mu _{0}+e_{i}(n)}^{\blacktriangle }(x_{i}).
\label{eqn:ineq_2}
\end{equation}%
Define $V_{m}=V_{m}(\xi )=(\mu _{0}+z_{m},y^{\maltese }(\mu _{0}+z_{m},\xi
)) $ and $V=V(\xi )=(\mu _{0}+e_{i}(n),y^{\maltese }(\mu _{0}+e_{i}(n),\xi
)) $, which can be viewed as random vectors defined on $\mathbb{R}^{n}$
(equipped with the Borel $\sigma $-field) and where the probability measure
is given by $\Xi $. Note that $V_{m}$ converges to $V$ everywhere as $%
m\rightarrow \infty $ (since $y^{\maltese }(y,\xi )$ is continuous w.r.t. $y$%
). Define $S:\mathbb{R}^{n}\times \mathbb{R}^{n}\rightarrow \mathbb{R\cup
\{\infty \}}$ via%
\begin{equation*}
S(z,x)=\left\{ 
\begin{array}{cc}
(R\hat{\beta}\left( x\right) -R\hat{\beta}\left( z\right) )^{\prime }\left( 
\tilde{\sigma}^{2}(x)R\left( X^{\prime }X\right) ^{-1}R^{\prime }\right)
^{-1}(R\hat{\beta}\left( x)\right) -R\hat{\beta}\left( z\right) ) & \text{if 
}x\notin \mathfrak{M}_{0} \\ 
\infty & \text{if }x\in \mathfrak{M}_{0}%
\end{array}%
\right.
\end{equation*}%
and note that $\tilde{H}_{uc,\mu _{0}+z_{m}}^{\blacktriangle }(t)=\Xi
(S(V_{m})\leq t)$ and $\tilde{H}_{uc,\mu _{0}+e_{i}(n)}^{\blacktriangle
}(t)=\Xi (S(V)\leq t)$ holds for $t\in \mathbb{R}$ and thus also for $%
t=x_{i} $ (recall that $x_{i}$ is a real number). The statement in (\ref%
{eqn:ineq_2}) now follows from Lemma \ref{lem:convdisc} in Appendix \ref{App
A}, recalling that we have chosen $x_{i}$ as a continuity point of $\tilde{H}%
_{uc,\mu _{0}+e_{i}(n)}^{\blacktriangle }$, which implies $\Xi
(S(V)=x_{i})=0 $, and noting that $\bar{S}=S$ here holds. Summarizing, we
hence arrive, replacing $\delta $ by another element of $(0,\delta )$ if
necessary, at%
\begin{equation}
\tilde{T}_{uc}(\mu _{0}+z)>x_{i}\text{ and }\tilde{H}_{uc,\mu
_{0}+z}^{\blacktriangle }(x_{i})>1-\alpha \text{ for every }z\in
B(e_{i}(n),\delta ).  \label{eqn:bootdelta_uncorr_tilde_triangle_malt2}
\end{equation}%
Now, for every $z\in \mathbb{R}^{n}$ and every $\gamma \neq 0$ we have%
\begin{eqnarray*}
\tilde{H}_{uc,\mu _{0}+\gamma z}^{\blacktriangle }(x_{i}) &=&\Xi \left(
\left\{ \xi :\tilde{T}_{uc}^{\blacktriangle ,\maltese }(\mu _{0}+\gamma
z,\xi )\leq x_{i}\right\} \right) \\
&=&\Xi \left( \left\{ \xi :\tilde{T}_{uc}^{\blacktriangle ,\maltese }(\mu
_{0}+z,\xi )\leq x_{i}\right\} \right) =\tilde{H}_{uc,\mu
_{0}+z}^{\blacktriangle }(x_{i})
\end{eqnarray*}%
by what has been shown at the very beginning of the proof. Using also the
observation for $\tilde{T}_{uc}(\mu _{0}+\gamma z)$ made at the beginning of
the proof, the preceding display and (\ref%
{eqn:bootdelta_uncorr_tilde_triangle_malt2}) thus give%
\begin{equation*}
\tilde{T}_{uc}(\mu _{0}+y)>x_{i}\text{ and }\tilde{H}_{uc,\mu
_{0}+y}^{\blacktriangle }(x_{i})>1-\alpha \text{ for every }y\in \{\gamma
z:\gamma \neq 0,z\in B(e_{i}(n),\delta )\}=:\Delta ,
\end{equation*}%
or equivalently%
\begin{equation}
\tilde{T}_{uc}(y)>x_{i}\text{ and }\tilde{H}_{uc,y}^{\blacktriangle
}(x_{i})>1-\alpha \text{ for every }y\in \mu _{0}+\Delta .
\label{eqn:bootdelta_uncorr_tilde_triangle_malt3}
\end{equation}%
By Lemma \ref{lem:quant} we see that $\tilde{H}_{uc,y}^{\blacktriangle
}(x_{i})>1-\alpha $ implies $x_{i}\geq \tilde{h}_{uc,1-\alpha
}^{\blacktriangle }(y)$. Hence,%
\begin{equation*}
\left\{ y\in \mathbb{R}^{n}:\tilde{T}_{uc}(y)>x_{i},\text{ }\tilde{H}%
_{uc,y}^{\blacktriangle }(x_{i})>1-\alpha \right\} \subseteq \left\{ y\in 
\mathbb{R}^{n}:\tilde{T}_{uc}(y)>\tilde{h}_{uc,1-\alpha }^{\blacktriangle
}(y)\right\} .
\end{equation*}%
This, together with (\ref{eqn:bootdelta_uncorr_tilde_triangle_malt3}), gives%
\begin{equation*}
\left\{ y\in \mathbb{R}^{n}:\tilde{T}_{uc}(y)>\tilde{h}_{uc,1-\alpha
}^{\blacktriangle }(y)\right\} \supseteq \mu _{0}+\Delta \supseteq \mu _{0}+%
\limfunc{int}(\Delta )\supseteq \mu _{0}+(\limfunc{span}(e_{i}(n))\backslash
\{0\}),
\end{equation*}%
where $\limfunc{int}(\Delta )$ denotes the interior of $\Delta $. The proof
is now completed by following the steps after (\ref{eqn:bootincl_2}) in the
proof of Theorem \ref{thm:main} in Appendix \ref{App A}.

Part (a) now follows from the already established Part (b), noting that we
may choose $\tilde{h}_{uc,1-\alpha }^{\blacktriangle }$ such that $\tilde{h}%
_{uc,1-\alpha }^{\blacktriangle }(y)\geq \tilde{h}_{uc,1-\alpha }(y)$ holds
for every $y\in \mathbb{R}^{n}$. $\blacksquare $

\section{Appendix: Proofs for Section \protect\ref{special}\label%
{app_special_cases}}

\textbf{Proof of Theorem \ref{special_cases}:} Because of the assumption $%
q=k $, we have $\mathcal{A}=\mathfrak{M}_{0}=\{\mu _{0}\}$, where $\mu
_{0}=X\beta _{0}$ and $\beta _{0}$ is defined as $R^{-1}r$. Observe that $%
y^{\ast }(y,\xi )=X\beta _{0}+\limfunc{diag}(\xi )(y-X\beta _{0})$, and
hence $y^{\ast }(\mu _{0}+e_{i}(n),\xi )=\mu _{0}+\xi _{i}e_{i}(n)$ for
every $i=1,\ldots ,n$. Furthermore, note that our assumption on $\Xi $ is
equivalent to $\Xi (\left\{ \xi \in \mathbb{R}^{n}:\tprod_{i=1}^{n}\xi
_{i}\neq 0\right\} )=1$.

(a) Because Assumption \ref{R_and_X} is assumed in Theorem \ref%
{Theo_Het_unrestr}, we may conclude that $e_{i}(n)\notin \limfunc{span}(X)$
holds for every $i=1,\ldots ,n$.\footnote{%
If $e_{i}(n)\in \limfunc{span}(X)$ would hold for some $i$, then there would
exist a $k\times 1$ vector $a$, sucht that $Xa=e_{i}(n)$. It would follow
that $(X^{\prime }\left( \lnot (i_{1},\ldots i_{s})\right) )^{\prime }a=0$,
where $X^{\prime }\left( \lnot (i_{1},\ldots i_{s})\right) $ is the matrix
appearing in Assumption \ref{R_and_X}. As a consequence, this assumption
would be violated, a contradiction.} This shows that $\vartheta _{2,Het}=0$.
We now turn to $\vartheta _{1,Het}$: Consider $i=1,\ldots ,n$, such that $%
e_{i}(n)\notin \mathsf{B}$ (if no such $i$ exists $\vartheta _{1,Het}=0$
follows immediately from our convention). Observe that $T_{Het}^{\ast }(\mu
_{0}+e_{i}(n),\xi )=T_{Het}(\mu _{0}+\xi _{i}e_{i}(n))$ by definition and
that, in view of $G(\mathfrak{M}_{0})$-invariance of $T_{Het}$ (Remark \ref%
{rem:GM0}), this equals $T_{Het}(\mu _{0}+e_{i}(n))$ provided that $\xi
_{i}\neq 0$ (cf. the argument at the beginning of the proof of Theorem \ref%
{thm:main}). By our assumption on $\Xi $, we can conclude that $\vartheta
_{1,Het}=0$. This now gives $\vartheta _{Het}=1$.

(b) Consider first $i=1,\ldots ,n$, such that $e_{i}(n)\in \limfunc{span}(X)$%
. Then $y^{\ast }(\mu _{0}+e_{i}(n),\xi )=\mu _{0}+\xi _{i}e_{i}(n)\in 
\limfunc{span}(X)$ also holds for every $\xi $, since $\mu _{0}\in \mathfrak{%
M}_{0}\subseteq \limfunc{span}(X)$ and since $\limfunc{span}(X)$ is a linear
space. This shows that $\vartheta _{2,uc}=0$. We turn now to $\vartheta
_{1,uc}$: Consider $i=1,\ldots ,n$, such that $e_{i}(n)\notin \limfunc{span}%
(X)$ (which has to be the case for at least one $i$ in view of the
assumption $k<n$). Observe that $T_{uc}^{\ast }(\mu _{0}+e_{i}(n),\xi
)=T_{uc}(\mu _{0}+\xi _{i}e_{i}(n))$ by definition and that, in view of $G(%
\mathfrak{M}_{0})$-invariance of $T_{uc}$ (Remark \ref{rem:GM0}), this
equals $T_{uc}(\mu _{0}+e_{i}(n))$ provided that $\xi _{i}\neq 0$ (cf. the
argument at the beginning of the proof of Theorem \ref{thm:main}). By our
assumption on $\Xi $, we can conclude that $\vartheta _{1,uc}=0$. This now
gives $\vartheta _{uc}=1$.

(c) Consider $i=1,\ldots ,n$, such that $\mu _{0}+e_{i}(n)\notin \mathsf{%
\tilde{B}}$ (if no such $i$ exists $\tilde{\vartheta}_{Het}=0$ follows
immediately from our convention). Observe that $\tilde{T}_{Het}^{\ast }(\mu
_{0}+e_{i}(n),\xi )=\tilde{T}_{Het}(\mu _{0}+\xi _{i}e_{i}(n))$ by
definition and that, in view of $G(\mathfrak{M}_{0})$-invariance of $\tilde{T%
}_{Het}$ (Remark \ref{rem:tildeGM0}), this equals $\tilde{T}_{Het}(\mu
_{0}+e_{i}(n))$ provided that $\xi _{i}\neq 0$ (cf. the argument at the
beginning of the proof of Theorem \ref{thm:main}). By our assumption on $\Xi 
$, we can conclude that $\tilde{\vartheta}_{Het}=1$.

(d) Consider $i=1,\ldots ,n$ such that $\mu _{0}+e_{i}(n)\notin \mathfrak{M}%
_{0}$ (which is actually the case here for every $i$ since $\mathfrak{M}%
_{0}=\{\mu _{0}\}$). Observe that $\tilde{T}_{uc}^{\ast }(\mu
_{0}+e_{i}(n),\xi )=\tilde{T}_{uc}(\mu _{0}+\xi _{i}e_{i}(n))$ by definition
and that, in view of $G(\mathfrak{M}_{0})$-invariance of $\tilde{T}_{uc}$
(Remark \ref{rem:tildeGM0}), this equals $\tilde{T}_{uc}(\mu _{0}+e_{i}(n))$
provided that $\xi _{i}\neq 0$ (cf. the argument at the beginning of the
proof of Theorem \ref{thm:main}). By our assumption on $\Xi $, we can
conclude that $\tilde{\vartheta}_{uc}=1$.

\section{Appendix: Computational details\label{App compu}}

For every Setting A, B, and C and every bootstrap-based test procedure
(i.e., combination of test statistic and bootstrap scheme) the two-step
procedure of computations outlined in Section \ref{sec:comp} proceeds as
detailed in the following two subsections.

\subsection{Step 1: Search for design matrices leading to a small $\protect%
\vartheta $\label{sec:stp1}}

In every scenario, i.e., for every combination of $k=2,\ldots ,5$ and $%
q=1,\ldots ,k-1$, we do the following: In a loop, we generate $n\times k$%
-dimensional design matrices $X$ with first column the intercept, and the
remaining coordinates drawn independently from a log-(standard) normal
distribution. For every $X$ generated in this way, we determine the
corresponding value of $\vartheta $ for testing the restriction $R=(0:I_{q})$
and $r=0$ (after checking whether $X$ has full rank and satisfies the
assumption in the theorem corresponding to $\vartheta $ (if applicable)).%
\footnote{%
That is, Assumption \ref{R_and_X} (Assumption \ref{R_and_X_tilde},
respectively) is checked if the test under consideration falls under the
regime of Theorem \ref{Theo_Het_unrestr} (Theorem \ref{Theo_Het_restr_1} or %
\ref{Theo_Het_restr_2}, respectively). All these checks, including the full
rank check, were always passed (which should not come as a surprise as these
conditions are generically satisfied).} If a matrix $X$ is found such that
the corresponding $\vartheta $ satisfies $\vartheta <0.01$, or if $150$
design matrices have been generated in total we stop the loop.

Then, we determine a matrix $X_{\ast }(k,q)$, say, that realizes the minimal
value of $\vartheta $ among the (at most $150$) matrices considered. In
preparation for Step 2, for $X_{\ast }(k,q)$ (and the restriction given by
the associated $q$) we also compute, by Monte Carlo based on $300$
replications for computing each probability (and assuming normality), the
null rejection probabilities $\pi _{\alpha ,\rho }$, say, of the
bootstrap-based test under consideration for $\alpha =0.05$ and $\alpha =0.1$
under the parameters $\beta =0$, $\sigma ^{2}=1$, and $\Sigma (\rho ,i^{\ast
})=\limfunc{diag}(\tau _{1}^{2}(\rho ,i^{\ast }),\ldots ,\tau _{n}^{2}(\rho
,i^{\ast }))$ given by 
\begin{equation}
\tau _{i^{\ast }}^{2}(\rho ,i^{\ast })=\rho ,\text{ and }\tau _{i}^{2}(\rho
,i^{\ast })=(1-\rho )/(n-1)\text{ for }i\neq i^{\ast },  \label{eqn:extrvar}
\end{equation}%
for $\rho \in \{n^{-1},0.9,0.99,0.999,0.9999\}$ and where $i^{\ast }$
denotes the first index which realizes the optimum in the optimization
problem defining $\vartheta =\vartheta (X_{\ast }(k,q))$ in the appropriate
theorem that applies to the bootstrap-based test under consideration. The
rationale for computing $\pi _{\alpha ,\rho }$ is the following: Inspection
of the proofs of the just mentioned theorems shows that if $\vartheta
<\alpha $ holds, then the null rejection probabilities of the
bootstrap-based test converge to $1$ along a sequence of $\Sigma _{m}\in 
\mathfrak{C}_{Het}$ converging to $e_{i^{\ast }}(n)e_{i^{\ast }}(n)^{\prime }
$. Thus if $\vartheta <\alpha $, the null rejection probability $\pi
_{\alpha ,\rho }$ should be large for the variance specification considered
in (\ref{eqn:extrvar}) and $\rho $ close to $1$. This will be exploited as a
numerical check in Step 2.

Note also that $\rho =n^{-1}$ corresponds to homoskedastic errors.

For $n\in \{20,30\}$ (i.e., in Settings B and C) the \emph{empirical}
distribution $\Xi ^{\bullet }$ used in the bootstrap scheme is generated
once for each combination of $k$ and $q$ and then held fixed in the
computations of $\vartheta $ and the rejection probabilities described
before.

\subsection{Step 2: Numerical check and computation of additional size lower
bounds if necessary\label{sec:stp2}}

First, we determine the matrix $X_{\ast \ast }$, say, that corresponds to
the minimal value of $\vartheta $ among the matrices $\{X_{\ast
}(k,q):k=2,\ldots ,5,\ q=1,\ldots ,k-1\}$ determined in Step 1. We then
conduct the following numerical check: if, for $\alpha =0.05$ or $\alpha =0.1
$, the value of $\vartheta $ corresponding to $X_{\ast \ast }$ (and the
associated $q$) was smaller than $\alpha $, but $\max_{\rho }\pi _{\alpha
,\rho }<0.4$, we took this as an indication of numerical unreliability of $%
\vartheta $ (cf.~the discussion after (\ref{eqn:extrvar}) for an
explanation, and see the discussion in Section \ref{sec:numch} below
concerning some numerical issues that make determining $\vartheta $
nontrivial). If the value of $\vartheta $ corresponding to $X_{\ast \ast }$
was deemed unreliable, we replaced $X_{\ast \ast }$ by the matrix that led
to the minimal value of $\vartheta $ among all matrices in $\{X_{\ast
}(k,q):k=2,\ldots ,5,\ q=1,\ldots ,k-1\}\backslash \{X_{\ast \ast }\}$, and
iterated this procedure until the check was passed (which eventually was
always the case before the set of all matrices $X_{\ast }(k,q)$ was
exhausted). Typically, the check was passed right away.\footnote{%
More specifically, in Setting A, B, and C no replacements were necessary for~%
$99.38\%$,~$98.96\%$, and $98.33\%$, respectively, of the $960$ bootstrap
based test procedures considered.} Denote by $\vartheta _{\min }$ the value
of $\vartheta $ corresponding to the so-obtained matrix $X_{\ast \ast }$.

Once the check was passed, we proceeded as follows, separately for $\alpha
=0.05$ and $\alpha =0.1$: On the one hand, if $\vartheta _{\min }<\alpha $
(and the maximal rejection probability computed thus was guaranteed to be at
least $0.4$ by the numerical check), or if $\max_{\rho }\pi _{\alpha ,\rho
}\geq 3\alpha $, no additional computations were carried out, which was the
case for the vast majority of test procedures.\footnote{\label{FN_vast}For~$%
\alpha =0.05$ no additional computations were necessary for~$921$,~$904$,
and~$887$ of the~$960$ bootstrap based test procedures in Setting A, B, and
C, respectively; for~$\alpha =0.1$ no additional computations were necessary
for $947$, $926$, and $920$ of the $960$ bootstrap based test procedures in
Setting A, B, and C, respectively} In this case the values of $\vartheta
_{\min }$ and of $\max_{\rho }\pi _{\alpha ,\rho }$ are reported. Note that
if $\vartheta _{\min }<\alpha $ our theoretical results indicate that the
test has size $1$ (in the scenario corresponding to $X_{\ast \ast }$); and
if $\max_{\rho }\pi _{\alpha ,\rho }\geq 3\alpha $, the null rejection
probability is exceedingly large (even if $\vartheta _{\min }\not<\alpha $)
and thus the test is also not reliable. On the other hand, if $\vartheta
_{\min }\geq \alpha $ and $\max_{\rho }\pi _{\alpha ,\rho }<3\alpha $, we
started a second set of computations to determine null rejection
probabilities of the test over a range of additional design matrices.

In this second set of computations, we focus exclusively on the scenario $k$
and $q$ pertaining to $X_{\ast \ast }$. In this scenario we first generate
in a loop further design matrices in the same way as in Step 1, and compute
for each matrix the value of $\vartheta $ (after checking whether $X$ has
full rank and satisfies the assumption in the theorem corresponding to $%
\vartheta $ (if applicable)).\footnote{%
Again all these checks, including the full rank check, were always passed.}
If one of these new design matrices leads to a $\vartheta $ smaller than $%
\alpha $, or once $150$ matrices have been considered, we stop the loop. The
matrix corresponding to the smallest value of $\vartheta $ is then used as
the starting value in the following routine, in addition to (at most) $29$
newly generated design matrices (of dimension $n\times k$) that are
generated in the same way as the design matrices in Step 1:

\begin{enumerate}
\item[(I)] For every $X$ considered we compute via a Monte Carlo method
(again based on $300$ replications and normality) the maximal null rejection
probability for $\beta =0$, $\sigma ^{2}=1$ and $\Sigma =\limfunc{diag}(\tau
_{1}^{2},\ldots ,\tau _{n}^{2})$ where the vector $(\tau _{1}^{2},\ldots
,\tau _{n}^{2})^{\prime }$ varies in the set of vectors that is obtained from%
\begin{equation*}
\{(\tau _{1}^{2}(0.99,i),\ldots ,\tau _{n}^{2}(0.99,i))^{\prime
}+s|G|:i=1,\ldots ,n,\ s=0,0.1,0.5\}
\end{equation*}%
after dividing each vector by its $l_{1}$-norm (cf. (\ref{eqn:extrvar}));
here the vector $|G|$ denotes coordinate-wise absolute values of $G$, and
the latter is obtained as a draw from an $n$-variate standard normal
distribution (here $G$ was re-drawn for each $X$, $i$ and $s$). Once an $X$
is found such that the maximal null rejection probability exceeds $4\alpha $%
, or once all $30$ design matrices have been considered, we stop these
computations.

\item[(II)] Denote by $X_{\ast \ast \ast }$ the matrix for which the largest
rejection probability was obtained in (I). All of the following computations
are done on $X_{\ast \ast \ast }$.

If the largest rejection probability in (I) is greater than $4\alpha $, we
re-compute this rejection probability on a new Monte Carlo sample (of size $%
300$ and under normality) and stop.

Otherwise, we run (at most) $20$ iterations of a Nelder-Mead optimization
algorithm initialized at the vector $(\tau _{1}^{2},\ldots ,\tau
_{n}^{2})^{\prime }$ that leads to the maximal rejection probability in (I)
to maximize the rejection probabilities further (all probabilities are again
determined by Monte Carlo with $300$ replications assuming normality each
time a rejection probability is required during the Nelder-Mead
optimization). After these iterations, the rejection probability for the
\textquotedblleft optimal\textquotedblright\ $(\tau _{1}^{2},\ldots ,\tau
_{n}^{2})^{\prime }$ so obtained is re-computed on a new Monte Carlo sample
(of size 300 and assuming normality).

In both cases we report the last null rejection probability computed.
Furthermore, we compute the value of $\vartheta $ corresponding to $X_{\ast
\ast \ast }$ (and the restrictions implied by the value of $q$ under
consideration) and report this value.\footnote{%
The matrix $X_{\ast \ast \ast }$ was also checked to have full rank and to
satisfy the assumption in the theorem corresponding to $\vartheta $ (if
applicable). This checks were always passed.}$^{\text{,}}$\footnote{%
The values of $\vartheta $ computed in the second set of computations are
not subject to a numerical reliability check. While we report these
(unvetted) values of $\vartheta $ as auxiliary information also in cases
where this second set of computations are needed, we base our classification
of the corresponding bootstrap-based test procedure as reliable or
unreliable only on the magnitude of the computed null rejection
probabilities and not on the value of such a $\vartheta $. See Section \ref%
{sec:results}.}
\end{enumerate}

In the second set of computations in Step 2, for $n\in \{20,30\}$, and in
contrast to Step 1, whenever a new design matrix $X$ was considered, the
empirical distribution $\Xi ^{\bullet }$ used in the bootstrap scheme was
newly initialized, and computations of $\vartheta $ and of rejection
probabilities for this $X$ were done based on this corresponding empirical
distribution $\Xi ^{\bullet }$.

As a point of interest we note that for every bootstrap-based test we also
performed the following check: For every design matrix $X$ (with associated
scenario $(k,q)$) that underlies a result shown in Figure \ref{fig:p} or
Tables \ref{tab:A}-\ref{tab:C} we have checked that all the relevant
assumptions for size-controllability of the corresponding \emph{non}%
-bootstrap-based test given in \cite{PP5HC} are satisfied.\footnote{%
We have also checked that all these design matrices $X$ satisfy $h_{ii}<1$
for all $i=1,\ldots ,n$ (which implies some, but not all, of the before
mentioned conditions in \cite{PP5HC}).} Hence, for all these design matrices 
$X$ (and their associated scenarios $(k,q)$), that we have shown to result
in oversized bootstrap-based tests in the vast majority of cases, in
contrast the testing procedures established in \cite{PP5HC} indeed work and
deliver valid tests, i.e., tests that are \emph{not} oversized.

\subsection{Numerical details concerning the computation of $\protect%
\vartheta $\label{sec:numch}}

Evaluating $\vartheta $ numerically is a nontrivial task. In order to be on
the safe side and to bias our results in favor of the bootstrap-based tests
(recall we are after negative results), the $\vartheta $ we shall obtain
will actually be a numerical obtained upper bound for the true $\vartheta $
(as can be seen from the subsequent discussion). In the following we discuss
how the routines implemented in the R-package \textbf{wbsd} address the
numerical challenges encountered in determining $\vartheta $. We give the
discussion for the case where $\vartheta _{Het}$ is to be computed. The
routines proceed similarly for the other cases.

As can be seen from Theorem \ref{Theo_Het_unrestr}, to obtain $\vartheta
_{1,Het}$ one repeatedly needs to compute $T_{Het}(z)$, after checking that $%
z\notin \mathsf{B}$, for various vectors $z$, say. Because this check is
needed anyway for computing $T_{Het}(z)$ (see (\ref{T_het})), it is carried
out in \textbf{wbsd} within the sub-routine computing $T_{Het}(z)$. To this
end, the subroutine uses the function \textquotedblleft
isInvertible\textquotedblright\ associated to the rank-revealing LU
decomposition (with complete pivoting) from the Eigen-library for linear
algebra in C++ (recall that $z\notin \mathsf{B}$ is equivalent to $\hat{%
\Omega}_{Het}\left( z\right) $ being invertible). The package \textbf{wbsd}
makes use of this library through the package \textbf{RcppEigen} by \cite%
{Rcppeigen}. The function \textquotedblleft isInvertible\textquotedblright\
requires the specification of a tolerance parameter, which we set to $%
10^{-6} $. Note that the function \textquotedblleft
isInvertible\textquotedblright\ works in such a way that the larger the
tolerance parameter, the more likely it is that the input matrix is
categorized as numerically non-invertible (i.e., that $z$ is categorized as
satisfying $z\in \mathsf{B}$). As a consequence, the larger the tolerance
parameter, the smaller the numerically determined value of $\vartheta
_{1,Het}$, because violated rank conditions decrease $\vartheta _{1,Het}$.%
\footnote{%
For the tests~$T_{uc}$ ($\tilde{T}_{uc}$, respectively) we checked whether
the residual variance estimate $\hat{\sigma}^{2}(z)$ ($\tilde{\sigma}^{2}(z)$%
, respectively) exceeds~$10^{-6}$ instead of using the
function~\textquotedblleft isInvertible\textquotedblright .}

In addition to rank computations, the definition of $\vartheta _{1,Het}$
requires to numerically check the inequality in the events defining $%
\vartheta _{1,Het}$, see (\ref{theta1_Het}). To this end, we introduce
another tolerance parameter and instead of checking $"\ldots <\ldots "$
directly, we checked $"\ldots +10^{-5}<\ldots "$ Note that also here,
increasing the tolerance parameter decreases the numerically determined
value of $\vartheta _{1,Het}$.

In the computation of $\vartheta _{2,Het}$, checks of the form $z\notin 
\mathsf{B}$ were done exactly in the same way as in the computation of $%
\vartheta _{1,Het}$ described above.\footnote{%
When computing $\vartheta _{2,uc}$, we checked whether the residual variance
estimate $\hat{\sigma}^{2}(z)$ exceeded~$10^{-6}$ instead of using the
function~\textquotedblleft isInvertible\textquotedblright .} Again, the
larger the tolerance parameter the more likely it is that the input matrix
becomes numerically rank deficient. Thus, increasing the tolerance parameter
leads to a possible decrease in the numerically determined value of $%
\vartheta _{2,Het}$. The computation of $\vartheta _{2,Het}$ also requires
checking whether $R\hat{\beta}(e_{i}(n))\neq 0$. Numerically we checked this
via determining whether $\Vert R\hat{\beta}(e_{i}(n))\Vert _{\infty
}>10^{-6} $ where $\Vert \cdot \Vert _{\infty }$ denotes the maximum norm
(again the larger the tolerance parameter, the smaller the numerically
determined value of $\vartheta _{2,Het}$). To compute $\vartheta _{2,Het}$
one also needs to check if $\limfunc{rank}((X:e_{i}(n)))<k+1$. These rank
computations are implemented in \textbf{wbsd} based on the rank-revealing LU
decomposition (with complete pivoting) and the corresponding function
\textquotedblleft rank\textquotedblright\ obtained from the Eigen-library
mentioned above. Checking rank conditions via the function \textquotedblleft
rank\textquotedblright\ requires the specification of a tolerance parameter.
To check $\limfunc{rank}((X:e_{i}(n)))<k+1$, we chose the tolerance
parameter equal to the machine epsilon $2.220446\times 10^{-16}$, because --
in contrast to the previous checks -- decreasing the tolerance parameter
used in this check decreases the numerically determined value of $\vartheta
_{2,Het}$. [Note that decreasing $\vartheta _{1,Het}$ or $\vartheta _{2,Het}$
increases $\vartheta _{Het}$.]

Theorem \ref{Theo_Het_unrestr} requires Assumption \ref{R_and_X} to hold,
which (as noted above) was checked throughout. This condition can be
verified by a series of rank computations, which was done based on the
function \textquotedblleft rank\textquotedblright\ as discussed above, but
with a tolerance parameter of $10^{-7}$. A similar remark applies to
Theorems \ref{Theo_Het_restr_1} and \ref{Theo_Het_restr_2} with regard to
Assumption \ref{R_and_X_tilde}.

% (uses file "ims.bst")
\bibliographystyle{ims}
\bibliography{refs}
% expects file "refs.bib"

\end{document}